\documentclass[opre,nonblindrev]{informs3}

\OneAndAHalfSpacedXI 

\usepackage{endnotes}
\let\footnote=\endnote
\let\enotesize=\normalsize
\def\notesname{Endnotes}%
\def\makeenmark{$^{\theenmark}$}
\def\enoteformat{\rightskip0pt\leftskip0pt\parindent=1.75em
  \leavevmode\llap{\theenmark.\enskip}}

\usepackage{natbib}
 \bibpunct[, ]{(}{)}{,}{a}{}{,}%
 \def\bibfont{\small}%
 \def\bibsep{\smallskipamount}%
 \def\bibhang{24pt}%

\usepackage{booktabs}
\usepackage{enumitem}
\usepackage{subcaption}
\usepackage{algorithmic, algorithm}
\usepackage{pbox}

\newcommand{\U}{{\mathcal{U}}}
\newcommand{\R}{{\mathbb{R}}}

\newcommand{\E}{{\mathbb{E}}}
\renewcommand{\P}{\mathbb{P}}
\newcommand{\Q}{\mathbb{Q}}
\newcommand{\I}{\mathbb{I}}
\newcommand{\burv}{{\mathbf{\tilde{u}}}}
\newcommand{\urv}{{\tilde{u}}}
\newcommand{\bzero}{\mathbf{0}}
\newcommand{\buhat}{\hat{\mathbf{u}}}
\newcommand{\uhat}{\hat{u}}
\newcommand{\bu}{\mathbf{u}}
\renewcommand{\S}{\mathcal{S}}
\newcommand{\InfCalP}{\mathcal{P}}

\newcommand{\bx}{\mathbf{x}}
\newcommand{\by}{\mathbf{y}}
\newcommand{\bz}{\mathbf{z}}
\renewcommand{\be}{\mathbf{e}}
\renewcommand{\bzero}{\mathbf{0}}
\newcommand{\bt}{\mathbf{t}}
\newcommand{\bb}{\mathbf{b}}
\newcommand{\bm}{\mathbf{m}}
\newcommand{\bs}{\mathbf{s}}
\newcommand{\bc}{\mathbf{c}}

\newcommand{\br}{\mathbf{r}}
\newcommand{\bp}{ \mathbf{p}}
\newcommand{\bq}{ \mathbf{q}}
\newcommand{\bv}{\mathbf{v}}
\newcommand{\bw}{\mathbf{w}}

\newcommand{\bI}{\mathbf{I}}

\newcommand{\ba}{\mathbf{a}}
\newcommand{\bZ}{\mathbf{Z}}

\newcommand{\bA}{\mathbf{A}}

\newcommand{\bC}{\mathbf{C}}

\newcommand{\bmu}{\boldsymbol{\mu}}
\newcommand{\bSigma}{\boldsymbol{\Sigma}}
\newcommand{\bepsilon}{\boldsymbol{\epsilon}}

\newcommand{\btheta}{\boldsymbol{\theta}}

\newcommand{\blambda}{\boldsymbol{\lambda}}
\newcommand{\bzeta}{\boldsymbol{\zeta}}

\usepackage{epstopdf}
\DeclareMathOperator{\CVAR}{\text{CVaR}}
\DeclareMathOperator{\supp}{\text{supp}}
\newcommand{\quant}[2]{\text{VaR}_{#1}^{#2}}

\TheoremsNumberedThrough     
\ECRepeatTheorems

\EquationsNumberedThrough    

\newcommand{\edit}{\textcolor{black}}
\newcommand{\blockedit}{\color{black}}

\begin{document}


\RUNAUTHOR{Bertsimas and Gupta and Kallus}

\RUNTITLE{Data-Driven Robust Optimization}

\TITLE{Data-Driven Robust Optimization}

\ARTICLEAUTHORS{%
\AUTHOR{Dimitris Bertsimas}
\AFF{Sloan School of Management, Massachusetts Institute of Technology, Cambridge, MA 02139, \EMAIL{dbertsim@mit.edu}} 
\AUTHOR{Vishal Gupta}
\AFF{Operations Research Center, Massachusetts Institute of Technology, Cambridge, MA 02139, \EMAIL{vgupta1@mit.edu}}
\AUTHOR{Nathan Kallus}
\AFF{Operations Research Center, Massachusetts Institute of Technology, Cambridge, MA 02139, \EMAIL{kallus@mit.edu}}
} 

\ABSTRACT{
The last decade witnessed an explosion in the availability of data for operations research applications.  Motivated by this growing availability, we propose a novel schema for utilizing data to design uncertainty sets for robust optimization using statistical hypothesis tests.  The approach is flexible and widely applicable, and robust optimization problems built from our new sets are computationally tractable, both theoretically and practically.  Furthermore, optimal solutions to these problems enjoy a strong, finite-sample probabilistic guarantee.  \edit{We describe concrete procedures for choosing an appropriate set for a given application and applying our approach to multiple uncertain constraints.  Computational evidence in portfolio management and queuing confirm that our data-driven sets significantly outperform traditional robust optimization techniques whenever data is available.}  
}

\HISTORY{This paper was
first submitted in August 2013.}

\maketitle

\section{Introduction}
Robust optimization is a popular approach to optimization under uncertainty.  
The key idea is to define an uncertainty set of possible realizations of the uncertain parameters and then optimize against worst-case realizations within this set.
Computational experience suggests that with well-chosen sets, robust models yield tractable optimization problems whose solutions perform as well or better than other approaches.  With poorly chosen sets, however, robust models may be overly-conservative or computationally intractable.  Choosing a good set is crucial.
Fortunately, there are several theoretically motivated and experimentally validated proposals for constructing good uncertainty sets \citep{ben2000robust, bertsimas2004price, ben2009robust, bandi2012tractable}.  These proposals share a common paradigm; they combine a priori reasoning with mild assumptions on the uncertainty to motivate the construction of the set.  

On the other hand, the last decade witnessed an explosion in the availability of data.  Massive amounts of data are now routinely collected in many industries.  
Retailers archive terabytes of transaction data.  Suppliers track order patterns across their supply chains.  
Energy markets can access global weather data, historical demand profiles, and, in some cases, real-time power consumption information.  
These data have motivated a shift in thinking -- away from a priori reasoning and assumptions and towards a new data-centered paradigm.  A natural question, then, is how should robust optimization techniques be tailored to this new paradigm?  

In this paper, we propose a general schema for designing uncertainty sets for robust optimization from data.  
{\blockedit We consider uncertain constraints of the form $f(\burv, \bx) \leq 0$ where $\bx \in \R^k$ is the optimization variable,  and $\burv \in \R^d$ is an uncertain parameter.  We model this constraint by choosing a set $\U$ and forming the corresponding robust constraint 
\begin{equation}
\label{eq:NonlinearRobust}
f(\bu, \bx) \leq 0 \ \ \forall \bu \in \U.
\end{equation}
We assume throughout that $f(\bu, \bx)$ is concave in $\bu$ for any $\bx$.  

In many applications, robust formulations decompose into a series constraints of the form \eqref{eq:NonlinearRobust} through an appropriate transformation of variables, including uncertain linear optimization and multistage adaptive optimization (see, e.g., \cite{ben2009robust}).  In this sense, \eqref{eq:NonlinearRobust} is a fundamental building block for more complex robust optimization models.} 

Many approaches \citep{bertsimas2004price, ben2009robust, chen2010cvar} to constructing uncertainty sets for \eqref{eq:NonlinearRobust} assume $\burv$ is a random variable whose distribution $\P^*$ is not known except for some assumed structural features.  For example, 
they may assume that $\P^*$ has independent components, while its marginal distributions are not known.  \edit{Given $\epsilon > 0$,} these approaches seek sets $\U_\epsilon$ that satisfy two key properties:  
{\blockedit

\begin{enumerate}[label=(P\arabic*)]
 \item \label{prop:tractability}The robust constraint \eqref{eq:NonlinearRobust} is \emph{computationally tractable.}  
\item \label{prop:guarantee} The set $\U_\epsilon$ \emph{implies a probabilistic guarantee for $\P^*$ at level $\epsilon$}, that is, for any $\bx^* \in \R^k$ and for every function $f(\bu, \bx)$ concave in $\bu$ for all $\bx$, we have the implication:
\begin{equation}
\label{def:Guarantee}
\text{If } f(\bu, \bx^*)  \leq 0 \ \  \forall \bu \in \U_\epsilon, 
\text{ then } \P^*(f(\burv, \bx^*) \leq 0 ) \geq  1-\epsilon.
\end{equation}  
\end{enumerate}
}
\ref{prop:guarantee} ensures that a feasible solution to the robust constraint will also be feasible with probability $1-\epsilon$ with respect to $\P^*$, despite not knowing $\P^*$ exactly.  Existing proposals achieve \ref{prop:guarantee} by leveraging the a priori structural features of $\P^*$.  \edit{Some of these approaches, e.g., \citep{bertsimas2004price}, only consider the special case when $f(\bu, \bx)$ is bi-affine, but one can generalize them to \eqref{def:Guarantee} using techniques from \cite{ben2012deriving} (see also Sec.~\ref{sec:nonlinear}).}  

Like previous proposals, we also assume $\burv$ is a random variable whose distribution $\P^*$ is not known exactly, and seek sets $\U_\epsilon$ that satisfy these properties.  Unlike previous proposals -- and this is critical -- we assume that we have data $\S=\{\buhat^1, \ldots, \buhat^N\}$ drawn i.i.d. according to $\P^*$.  
By combining these data with the a priori structural features of $\P^*$, we can design new sets that imply similar probabilistic guarantees, but which are much smaller with respect to subset containment than their traditional counterparts.  Consequently, robust models built from our new sets yield less conservative solutions than traditional counterparts, while retaining their robustness properties.  

The key to our schema is using the confidence region of a statistical hypothesis test to quantify what we learn about $\P^*$ from the data.  Specifically, our constructions depend on three ingredients: the a priori assumptions on $\P^*$, the data, and a hypothesis test.  By pairing different a priori assumptions and tests, we obtain distinct data-driven uncertainty sets, each with its own geometric shape, computational properties, and modeling power. These sets can capture a variety of features of $\P^*$, including skewness, heavy-tails and correlations.

In principle, there is a multitude of possible pairings of a priori assumptions and tests.  We focus on pairings we believe are most relevant to applied robust modeling.  Specifically, we consider a priori assumptions that are common in practice and tests that lead to tractable uncertainty sets.  Our list is non-exhaustive; there may exist other pairings that yield effective sets.  
Specifically, we consider situations where: 
\begin{itemize}
\item $\P^*$ has known, finite discrete support (Sec.~\ref{sec:Discrete}).
\item $\P^*$ {may} have continuous support, and the components of $\burv$ are independent (Sec.~\ref{sec:Independence}).
\item $\P^*$ may have continuous support, but data are drawn from its marginal distributions asynchronously (Sec.~\ref{sec:marginals}).  This situation models the case of missing values.
\item $\P^*$ may have continuous support, and data are drawn from its joint distribution (Sec.~\ref{sec:Correlated}).  This is the general case.
\end{itemize}
Table~\ref{tab:Results} summarizes the a priori structural assumptions, hypothesis tests, and resulting uncertainty sets that we propose.  
Each set is convex and admits a tractable, explicit description; see the referenced equations.

\begin{table}
\TABLE
{\label{tab:Results} Summary of data-driven uncertainty sets proposed in this paper.}  
{ \footnotesize	{ 
\begin{tabular}{ l l l l l l}
\toprule
Assumptions on $\P^*$  & Hypothesis Test&  \pbox{20cm}{Geometric \\ Description} & Eqs. & Separation
\\\midrule 
Discrete support & $\chi^2$-test & SOC & \eqref{def:ChisqFamily} \eqref{eq:ChisqU}
\\
Discrete support & G-test  & Polyhedral* & \eqref{def:ChisqFamily} \eqref{eq:GU}
\\ 
Independent marginals
&  KS Test  & Polyhedral* &  \eqref{def:UI} & line search
\\
 Independent marginals
& K Test  & Polyhedral* &  \eqref{def:UI2}  & line search
\\
 Independent marginals
 &  CvM Test & SOC* & \eqref{def:UI2}  \eqref{eq:CvMConic}
\\
 Independent marginals
 & W Test & SOC* & \eqref{def:UI2}  \eqref{eq:WConic}
\\
 Independent marginals 
& AD Test & EC &  \eqref{def:UI2}  \eqref{eq:ADConic} 
\\
Independent marginals & \cite{chen2007robust} & SOC & \eqref{def:UFB} & closed-form
\\
None & Marginal Samples & Box & \eqref{def:MarginalU} & closed-form
\\
None & Linear Convex Ordering  & Varies &  \eqref{eq:ULCX} & linear optimization
 \\
None & 
			\pbox{20cm}{ Shawe-Taylor \&  Cristianini (2003)} & SOC &  \eqref{def:UCS} & closed-form
\\
None  & 
			Delage \& Ye (2010) & LMI &  \eqref{def:UDY} 
\\
\bottomrule
\end{tabular}
}  }
{SOC, EC and LMI denote second-order cone representable sets, exponential cone representable sets, and linear matrix inequalities, respectively.  The additional ``*" notation indicates a set of of the above type with one additional, relative entropy constraint.  $KS$, $K$, $CvM$, $W$, and $AD$ denote the Kolmogorov-Smirnov, Kuiper, Cramer-von Mises, Watson and Anderson-Darling goodness of fit tests, respectively.  In some cases, we can separate over the constraint \eqref{eq:NonlinearRobust} for bi-affine $f$ with a specialized algorithm.  
In these cases, the column ``Separation" roughly describes this algorithm.}
\end{table}

\edit{For each of our sets, we provide an explicit, equivalent reformulation of \eqref{eq:NonlinearRobust}.  
The complexity of optimizing over this reformulation depends both on the function $f(\bu, \bx)$ and the set $\U$.  For each of our sets, we show that this reformulation is polynomial time tractable for a large class of functions $f$ including bi-affine functions, separable functions, conic-quadratic representable functions and certain sums of uncertain exponential functions.  
By exploiting special structure in some of our sets, we can provide specialized routines for directly separating over \eqref{eq:NonlinearRobust} for bi-affine $f$.  In these cases, the column ``Separation" in Table~\ref{tab:Results} roughly describes these routines.
Utilizing this separation routine within a cutting-plane method may offer performance superior to reformulation based-approaches (\cite{bertsimasreformulations, mutapcic2009cutting}).}  

{ \blockedit

We are not the first to consider using hypothesis tests in data-driven optimization.  Recently, \cite{ben2013robust} proposed a class of data-driven uncertainty sets based on phi-divergences.  (Phi divergences are closely related to some types of hypothesis tests.)  They focus on the case where the uncertain parameter is a probability distribution with known, finite, discrete support.  By contrast, we design uncertainty sets for general uncertain parameters with potentially continuous support such as future product demand, service times, and asset returns.  
Many existing robust optimization applications utilize similar general uncertain parameters.  Consequently, retrofitting these applications with our new data-driven sets to yield data-driven variants is perhaps more straightforward than using sets for uncertain probabilities.  From a methodological perspective, treating general uncertain parameters requires combining ideas from a variety of hypothesis tests (not just those based on phi-divergences of discrete distributions) with techniques from convex analysis and risk theory.  (See Sec.~\ref{sec:Schema}.)  

Other authors have also considered more specialized applications of hypothesis testing in data-driven optimization.  \citet{klabjan2013robust} proposes a distributionally robust dynamic program based on Pearson's $\chi^2$-test for a particular inventory problem.  \citet{goldfarb2003robust} calibrate an uncertainty set for the mean and covariance of a distribution using linear regression and the $t$-test.  It is not clear how to generalize these methods to other settings, e.g., distributions with continuous support in the first case or general parameter uncertainty in the second. By contrast, we offer a comprehensive study of the connection between hypothesis testing and uncertainty set design, addressing a number of cases with general machinery.

Moreover, our hypothesis testing perspective provides a unified view of many other data-driven methods from the literature.  
For example, \citet{calafiore2006distributionally} and \citet{delage2010distributionally} have proposed data-driven methods for chance-constrained and distributionally robust problems, respectively without using hypothesis testing.  We show how these works can be reinterpreted through the lens of hypothesis testing.  Leveraging this viewpoint enables us to apply state-of-the-art methods from statistics, such as the bootstrap, to refine these methods and improve their numerical performance.  Moreover, applying our schema, we can design data-driven uncertainty sets for robust optimization based upon these methods.    Although we focus on \citet{calafiore2006distributionally} and \citet{delage2010distributionally} in this paper, this strategy applies equally well to a host of other methods, such as the likelihood estimation approach of \cite{wang2009likelihood}.  
In this sense, we believe hypothesis testing and uncertainty set design provide a common framework in which to compare and contrast different approaches.

Finally, we note that \cite{campi2008exact} propose a very different data-driven method for robust optimization not based on hypothesis tests.  In their approach, one replaces the uncertain constraint $f(\burv, \bx) \leq 0$ with $N$ sampled constraints over the data, $f(\buhat^j, \bx) \leq 0$, for $j=1, \ldots, N$.  For $f(\bu, \bx)$ convex in $\bx$ with arbitrary dependence in $\bu$, they provide a tight bound $N(\epsilon)$ such that if $N \geq N(\epsilon)$, then, with high probability with respect to the sampling, any $\bx$ which is feasible in the $N$ sampled constraints satisfies $\P^*(f(\burv, \bx) \leq 0) \geq 1-\epsilon$.  Various refinements of this base method have also been proposed yielding smaller bounds $N(\epsilon)$, including incorporating $\ell_1$-regularization \citep{campi2013random} and allowing $\bx$ to violate a small fraction of the constraints \citep{calafiore2012data}.  Compared to our approach, these methods are more generally applicable and provide a similar probabilistic guarantee.  In the special case we treat where $f(\burv, \bx)$ is concave in $\bu$, however, our proposed approach offers some advantages.  First, 
because it leverages the concave structure of $f(\bu, \bx)$, our approach 
generally yields less conservative solutions (for the same $N$ and $\epsilon$) than  \citet{campi2008exact}.  (See Sec.~\ref{sec:Schema}.)  Second, for fixed $\epsilon > 0$, our approach is applicable even if $N < N(\epsilon)$, while theirs is not.  This distinction is important when $\epsilon$ is very small and there may not exist enough data.  Finally, as we will show, our approach reformulates \eqref{eq:NonlinearRobust} as a series of (relatively) sparse convex constraints, while the \citet{campi2008exact} approach will in general yield $N$ dense constraints which may be numerically challenging when $N$ is large.  For these reasons, practitioners may prefer our proposed approach in certain applications.
}

We summarize our contributions:
\begin{enumerate}
	\item We propose a new, systematic schema for constructing uncertainty sets from data using statistical hypothesis tests.  When the data are drawn i.i.d. from an unknown distribution $\P^*$, sets built from our schema imply a probabilistic guarantee for $\P^*$ at any desired level $\epsilon$.
	\item We illustrate our schema by constructing a multitude of uncertainty sets.  Each set is applicable under slightly different a priori assumptions on $\P^*$ as described in Table~\ref{tab:Results}.  
	\item \edit{We prove that robust optimization problems over each of our sets are generally tractable.  Specifically, for each set, we derive an explicit robust counterpart to \eqref{eq:NonlinearRobust} and show that for a large class of functions $f(\bu, \bx)$ optimizing over this counterpart can be accomplished in polynomial time using off-the-shelf software.} 		
	\item \edit{We unify several existing data-driven methods through the lens of hypothesis testing.  
	Through this lens, we motivate the use of common numerical techniques from statistics such as bootstrapping and gaussian approximation to improve their performance.  Moreover, we apply our schema to derive new uncertainty sets for \eqref{eq:NonlinearRobust} inspired by the refined versions of these methods.}
	\item \edit{We propose a new approach to modeling multiple uncertain constraints simultaneously with our sets by optimizing the parameters chosen for each individual constraint.  We prove that this technique is tractable and yields solutions which will satisfy all the uncertain constraints simultaneously for any desired level $\epsilon$.  
}
	\item \edit{We provide guidelines for practitioners on choosing an appropriate set and calibrating its parameters by leveraging techniques from model selection in machine learning.}
	\item \edit{Through applications in queueing and portfolio allocation, we assess the relative strengths and weaknesses of our sets.  Overall, we find that although all of our sets shrink in size as $N\rightarrow \infty$, they differ in their ability to represent features of $\P^*$.  Consequently, they may perform very differently in a given application.  In the above two settings, we find that our model selection technique frequently identifies a good set choice, and a robust optimization model built with this set performs as well or better than other robust data-driven approaches.}  
\end{enumerate}

The remainder of the paper is structured as follows.  Sec.~\ref{sec:Background} reviews background to keep the paper self-contained.  Sec.~\ref{sec:Schema} presents our schema for constructing uncertainty sets.  Sec.~\ref{sec:Discrete}-\ref{sec:Correlated} describe the various constructions in Table~\ref{tab:Results}.  
\edit{Sec.~\ref{sec:HypTestPerspective} reinterprets several techniques in the literature through the lens of hypothesis testing and, subsequently, uses them to motivate new uncertainty sets.}
Sec.~\ref{sec:Multiple} and Sec.~\ref{sec:Choose} discuss modeling multiple constraints and choosing the right set for an application, respectively.  Sec.~\ref{sec:computational} presents numerical experiments, and Sec.~\ref{sec:Conclusion} concludes.  All proofs are in the electronic companion.


\subsection{Notation and Setup}
\label{sec:notation}
Boldfaced lowercase letters ($\bx, \btheta, \ldots$) denote vectors, boldfaced capital letters ($\bA, \bC, \ldots$) denote matrices, and ordinary lowercase letters ($x, \theta$) denote scalars.  Calligraphic type ($\mathcal{P}, \mathcal{S} \ldots$) denotes sets.  
The $i^\text{th}$ coordinate vector is $\be_i$, and the vector of all ones is $\be$.  
We always use $\burv \in \R^d$ to denote a \emph{random} vector and $\urv_i$ to denote its components.  $\P$ denotes a generic probability measure for $\burv$, and $\P^*$ denotes its true (unknown) measure.  Moreover, $\P_i$ denotes the marginal measure of $\urv_i$.  We let $\S = \{\buhat^1, \ldots, \buhat^N\}$ be a sample of $N$ data points drawn i.i.d. according to $\P^*$, and let $\P^*_\S$ denote the measure of the sample $\S$, i.e., the $N$-fold product distribution of $\P^*$.  Finally, $\hat{\P}$ denotes the empirical distribution with respect to $\S$.  

\section{Background}
\label{sec:Background}
To keep the paper self-contained, we recall some results needed to prove our sets are tractable and imply a probabilistic guarantee.

\subsection{ \blockedit{Tractability of Robust Nonlinear Constraints} }
\label{sec:nonlinear}
{\blockedit

\cite{ben2012deriving} study constraint \eqref{eq:NonlinearRobust} and prove that for nonempty, convex, compact $\U$ satisfying a mild, regularity condition\footnote{An example of a sufficient regularity condition is that $ri(\U) \cap ri(dom(f(\cdot, \bx))) \neq \emptyset$, $\forall \bx \in \R^k$.
Here $ri(\U)$ denotes the \emph{relative interior} of $\U$.  Recall that for any non-empty convex set $\U$, $ri(\U) \equiv \{ \bu \in \U \ : \ \forall \bz \in \U, \ \exists \lambda > 1 \text{ s.t. } \lambda \bu + (1-\lambda) \bz \in \U \}$ (cf. \cite{bertsekas2003convex}).  
}, 
 \eqref{eq:NonlinearRobust} is equivalent to
\begin{equation} \label{eq:NonlinReform}
\exists \bv \in \R^d \ \text{s.t. }  \delta^*(\bv | \ \U ) - f_*(\bv, \bx) \leq 0.
\end{equation}
Here, $f_*(\bv, \bx)$ denotes the partial concave-conjugate of $f(\bu, \bx)$ 
and $\delta^*(\bv | \ \U)$ denotes the support function of $\U$,  
defined respectively as
\begin{align*} 
f_*(\bv, \bx) \equiv \sup_{\bu\in \R^d} \bu^T\bv - f(\bu, \bx),
&&
\delta^*(\bv | \ \U) \equiv \sup_{\bu \in \U} \bv^T\bu.
\end{align*}
For many $f(\bu, \bx)$, $f_*(\bv, \bx)$ admits a simple, explicit description.  
For example, for bi-affine $f(\bu, \bx) = \bu^T\mathbf{F}\bx + \mathbf{f}_\bu^T\bu + \mathbf{f}_\bx^T\bx + f_0$, we have
\[
f_*(\bv, \bx) = \begin{cases} -\mathbf{f}_\bx^T\bx - f_0 &\text{ if } v = \mathbf{\mathbf{F} \bx} + \mathbf{f}_\bu \					\\\ -\infty & \text{otherwise,}
			\end{cases}
\]
and \eqref{eq:NonlinReform} yields 
\begin{equation} \label{eq:ReformForBiAffine}
\delta^*( \mathbf{F}\bx + \mathbf{f}_\bu  | \ \U) + \mathbf{f}_\bx^T\bx + f_0 \leq 0.
\end{equation} 

In what follows, we concentrate on proving we can separate over $\{(\bv, t) : \delta^*(\bv | \ \U) \leq t\}$ in polynomial time for each of our sets $\U$, usually by representing this set as a small number of convex inequalities suitable for off-the-shelf solvers.  
From \eqref{eq:ReformForBiAffine}, this representation will imply that \eqref{eq:NonlinearRobust} is tractable for each of our sets whenever $f(\bu, \bx)$ is bi-affine.  

On the other hand, \cite{ben2012deriving} provide a number of other examples of $f(\bu, \bx)$ for which $f_*(\bv, \bx)$ is tractable, including:
\begin{description} \addtolength{\leftskip}{5 mm}
	\item[Separable Concave: $f(\bu, \bx) = \sum_{i=1}^k f_i(\bu) x_i$, ] for $f_i(\bu)$ concave and $x_i \geq 0$.
	\item[Uncertain Exponentials: $f(\bu, \bx) = -\sum_{i=1}^k x_i^{u_i}$, ] for $x_i > 1$ and $0 < u_i \leq 1$.
	\item[Conic Quadratic Representable: ] $f(\bu, \bx)$ such that the set $\{(t, \bu) \in \R \times \R^d : f(\bu, \bx) \geq t \} $ conic quadratic representable \citep[cf.][]{nemirovski2001lectures}.
\end{description}
Consequently, by providing a representation of $\{(\bv, t) : \delta^*(\bv | \ \U) \leq t\}$ for each of our sets, we will also have proven that \eqref{eq:NonlinearRobust} is tractable for each of these functions via \eqref{eq:NonlinReform}.  In other words, proving $\{(\bv, t) : \delta^*(\bv | \ \U) \leq t\}$ is tractable implies that \eqref{eq:NonlinearRobust} is tractable not only for bi-affine functions, but for many other concave functions as well.

For some sets, our formulation of $\{(\bv, t) : \delta^*(\bv | \ \U) \leq t\}$  will involve complex nonlinear constraints, such as exponential cone constraints (cf. Table~\ref{tab:Results}).  Although it is possible to optimize over these constraints directly in \eqref{eq:NonlinReform}, this approach may be numerically challenging.  As mentioned, an alternative is to use cutting-plane or bundle methods as in \cite{bertsimasreformulations, mutapcic2009cutting}.  
To this end, when appropriate, we provide specialized algorithms for separating over $\{(\bv, t) : \delta^*(\bv | \ \U) \leq t\}$ .  
}

\subsection{Hypothesis Testing} \label{sec:BackgroundHypTests}
{ \blockedit
We briefly review hypothesis testing as it relates to our set constructions.  See \cite{romano2005testing} for a more complete treatment.  

Given a null-hypothesis $H_0$ that makes a claim about an unknown distribution $\P^*$, a hypothesis test seeks to use data $\S$ drawn from $\P^*$ to either declare that $H_0$ is false, or, else, that there is insufficient evidence to determine its validity.  For a given significance level $0 < \alpha < 1$, a typical test prescribes a statistic $T \equiv T(\S, H_0)$, depending on the data and $H_0$, and a threshold $\Gamma \equiv \Gamma(\alpha, \S, H_0)$,  depending on $\alpha$, $\S$, and $H_0$.  If $T > \Gamma$, we reject $H_0$.  Since $T$ depends on $\S$, it is random.  The threshold $\Gamma$ is chosen so that the probability with respect to the sampling of \emph{incorrectly} rejecting $H_0$ is at most $\alpha$.  
The appropriate $\alpha$ is often application specific, although values of $\alpha = 1\%, 5\%$ and $10\%$ are common \citep[cf.,][Chapt. 3.1]{romano2005testing}.}

As an example, consider the two-sided Student's $t$-test \citep[][Chapt. 5]{romano2005testing}.  Given $\mu_0 \in \R$, the $t$-test considers the null-hypothesis
$
H_0: \E^{\P^*}[\urv]  = \mu_0$
using the statistic $T = | ({\hat{\mu} - \mu_0})/({ \hat{\sigma}\sqrt{N}})|$ and threshold $\Gamma = t_{N-1, 1-\alpha/2}$.  
  Here $\hat{\mu}, \hat{\sigma}$ are the sample mean and sample standard deviation, respectively, and $t_{N-1, 1-\alpha}$ is the $1-\alpha$ quantile of the Student $t$-distribution with $N-1$ degrees of freedom.  
Under the a priori assumption that $\P^*$ is Gaussian, the test guarantees that we will incorrectly reject $H_0$ with probability at most $\alpha$.

{\blockedit 
Many of the tests we consider are common in applied statistics, and tables for their thresholds are widely available.  Several of our tests, however, are novel (e.g., the deviations test in Sec.~\ref{sec:FwdBack}.)  In these cases, we propose using the \emph{bootstrap} to approximate a threshold (cf. Algorithm \ref{alg:Bootstrap}).  $N_B$ should be chosen to be fairly large; we take $N_B = 10^4$ in our experiments.  }  The bootstrap is a well-studied and widely-used technique in statistics \citep{efron1993introduction, romano2005testing}.  Strictly speaking, hypothesis tests based on the bootstrap are only asymptotically valid for large $N$.  (See the references for a precise statement.) Nonetheless, they are routinely used in applied statistics, even with $N$ as small as $100$, and a wealth of practical experience suggests they are extremely accurate.  Consequently, we believe practitioners can safely use bootstrapped thresholds in the above tests.  

\begin{algorithm} \blockedit
\caption{Bootstrapping a Threshold \label{alg:Bootstrap}}
\begin{algorithmic}
\renewcommand{\algorithmicrequire}{\textbf{Input:}}
\renewcommand{\algorithmicensure}{\textbf{Output:}}
\REQUIRE{$\S$, $T$, $H_0$, $0 < \alpha < 1$, $N_B \in \mathbb{Z}_+$}
\ENSURE{Approximate Threshold $\Gamma$}
\FOR{$j = 1\ldots N_B$}
	\STATE $\S^j \leftarrow$ Resample $|\S|$ data points from $\S$ with replacement
	\STATE $T^j \leftarrow T(\S^j, H_0)$
\ENDFOR
\RETURN{ $\lceil N_B ( 1-\alpha) \rceil$-largest value of $T^1, \ldots, T^{N_B}$. }
\end{algorithmic}
\end{algorithm}

Finally, we introduce the confidence region of a test, which will play a critical role in our construction.  
Given data $\S$, the $1-\alpha$ confidence region of a test is the set of null-hypotheses that would not be rejected for $\S$ at level $1-\alpha$.  For example, the $1-\alpha$ confidence region of the $t$-test is 
$\left\{ \mu \in \R : \left|\frac{\hat{\mu} - \mu}{ \hat{\sigma}\sqrt{N}} \right| \leq t_{N-1, 1-\alpha/2} \right\}.$  
In what follows, however, we commit a slight abuse of nomenclature and instead use the term confidence region to refer to the set of all measures that are consistent with any a priori assumptions of the test and also satisfy a null-hypothesis that would not be rejected.  In the case of the $t$-test, the confidence region in the context of this paper is
\begin{equation} \label{eq:tTestFirst}
\mathcal{P}^t \equiv \left\{ \P \in \Theta(-\infty, \infty): \P \text{ is Gaussian with mean } \mu, \text{ and } \left|\frac{\hat{\mu} - \mu}{ \hat{\sigma}\sqrt{N}} \right| \leq t_{N-1, 1-\alpha/2} \right\}, 
\end{equation}
where $\Theta(-\infty, \infty)$ is the set of Borel probability measures on $\R$.  

By construction, 
the probability (with respect to the sampling procedure) that $\P^*$ is a member of its confidence region is at least $1-\alpha$ as long as all a priori assumptions are valid.  This is a critical observation.  Despite not knowing $\P^*$, we can use a hypothesis test to create a set of distributions from the data that contains $\P^*$ for any specified probability.  

\section{Designing Data-Driven Uncertainty Sets}
\label{sec:Schema}

\subsection{\blockedit Geometric Characterization of the Probabilistic Guarantee}
As a first step towards our schema, we provide a geometric characterization of \ref{prop:guarantee}.  One might intuit that a set $\U$ implies a probabilistic guarantee at level $\epsilon$ only if $\P^*( \burv \in \U) \geq 1-\epsilon$.  As noted by other authors \cite[cf. pg. 32-33][]{ben2009robust}), however, this intuition is false.  Often, sets that are much smaller than the $1-\epsilon$ support will still imply a probabilistic guarantee at level $\epsilon$, and such sets should be preferred because they are less conservative.

{\blockedit
The crux of the issue is that there may be many realizations $\burv \not \in \U$ where nonetheless $f(\burv, \bx^*) \leq 0$.  Thus, $\P^*(\burv \in \U)$ is in general an underestimate of $\P^*(f(\burv, \bx^*) \leq 0)$.  One needs to exploit the dependence of $f$ on $\bu$ to refine the estimate.  We note in passing that many existing data-driven approaches for robust optimization, e.g., \cite{campi2008exact}, do not leverage this dependence.  Consequently, although these approaches are general purpose, they may yield overly conservative uncertainty sets for \eqref{eq:NonlinearRobust}.
}

In order to tightly characterize \ref{prop:guarantee}, we introduce the Value at Risk.  For any $\bv \in \R^d$ and measure $\P$, the Value at Risk at level $\epsilon$ with respect to $\bv$ is
\begin{equation}
\label{def:quantile}
\quant{\epsilon}{\P}(\bv)
\equiv \inf \left\{t : \P(\burv^T\bv \leq t) \geq 1- \epsilon\right\}.
\end{equation}
Value at Risk is positively homogenous (in $\bv$), but typically non-convex. (Recall a function $g(\bv)$ is positively homogenous if $g(\lambda \bv) = \lambda g(\bv)$ for all $\lambda > 0$.)  The critical result underlying our method is, then,

\begin{theorem}
\label{thm:support}\hfill
\begin{enumerate}[label=\alph*)] 
	\item Suppose $\U$ is nonempty, convex and compact.  Then, $\U$ implies a probabilistic guarantee at level $\epsilon$ for $\P$ for every $f(\bu, \bx)$ concave in $\bu$ for every $\bx$ if
\begin{equation*}
	\delta^*(\bv | \ \U) \geq \quant{\epsilon}{\P}(\bv) \ \ \forall \bv \in \R^d.
\end{equation*}
	\item Suppose $\exists \bv \in \R^d$ such that $\delta^*( \bv | \ \U^* ) < \quant{\epsilon}{\P}(\bv)$.  Then, there exists bi-affine functions $f(\bu, \bx)$ for which \eqref{def:Guarantee} does not hold.
\end{enumerate}
\end{theorem}

\edit{  
The first part generalizes a result implicitly used in \citep{ben2009robust, chen2007robust} when designing uncertainty sets for the special case of bi-affine functions.  To the best of our knowledge, the extension to general concave functions $f$ is new.  
}

\subsection{Our Schema}
\label{sec:Design}
The principal challenge in applying Theorem~\ref{thm:support} to designing uncertainty sets is that $\P^*$ is not known.  Recall, however, that the confidence region $\mathcal{P}$ of a hypothesis test, will contain $\P^*$ with probability at least $1-\alpha$.  This motivates the following schema:
Fix $0 < \alpha < 1$ and $0 < \epsilon < 1$.    
\begin{enumerate}
\item \label{step:I} Let $\mathcal{P}(\S, \alpha, \epsilon)$ be the confidence region of a hypothesis test at level $\alpha$.
\item \label{step:II} Construct a convex, positively homogenous (in $\bv$) upperbound $g(\bv, \S, \epsilon, \alpha)$ to the worst-case Value at Risk:
\[
\sup_{\P \in \mathcal{P}(\S, \alpha, \epsilon)} \quant{\epsilon}{\P}(\bv) \leq g(\bv, \S, \epsilon, \alpha) \ \ \forall \bv \in \R^d.
\]
\item \label{step:III} Identify the closed, convex set $\U(\S, \epsilon, \alpha)$ such that $g(\bv, \S, \epsilon, \alpha) = \delta^*( \bv | \ \U(\S, \epsilon, \alpha) )$.\footnote{The existence of such a set in Step~\ref{step:III} by the bijection between closed, positively homogenous convex functions and closed convex sets in convex analysis (see \cite{bertsekas2003convex}).}  
\end{enumerate}
{\blockedit
\begin{theorem} \label{thm:schema}
With probability at least $1-\alpha$ with respect to the sampling, the resulting set $\U(\S, \epsilon, \alpha)$ implies a probabilistic guarantee at level $\epsilon$ for $\P^*$.
\end{theorem} 
\begin{remark}We note in passing that $\delta^*( \bv | \ \U(\S, \epsilon, \alpha) ) \leq t $ is a safe-approximation to the ambiguous chance constraint $\sup_{\P \in \mathcal{P}(\S, \alpha, \epsilon)} \P(\bv^T\burv \leq t ) \geq 1-\epsilon$ as defined in \cite{ben2009robust}.  Ambiguous chance-constraints are closely related to sets which imply a probabilistic guarantee.  We refer the reader to \cite{ben2009robust} for more details. \end{remark}  

Theorem~\ref{thm:schema} ensures that with probability at least $1-\alpha$ with respect to the sampling, a robust feasible solution $\bx$ will satisfy a \emph{single} uncertain constraint $f(\burv, \bx) \leq 0$ with probability at least $1-\epsilon$.  Often, however, we face $m>1$ uncertain constraints $f_j(\burv, \bx) \leq 0$, \  $j=1, \ldots, m$, and seek $\bx$ that will simultaneously satisfy these constraints, i.e., 
\begin{equation} \label{eq:multProb}
\P\left( \max_{j=1, \ldots, m} f_j(\burv, \bx) \leq 0\right) \geq 1-\overline{\epsilon},
\end{equation}
for some given $\overline{\epsilon}$.
In this case, one approach is to replace each uncertain constraint with a corresponding robust constraint
\begin{equation} \label{eq:multRob}
f_j(\bu, \bx) \leq 0, \ \ \forall \bu \in \U(\S, \epsilon_j, \alpha),
\end{equation}
where $\U(\S, \epsilon_j, \alpha)$ is constructed via our schema at level $\epsilon_j = \epsilon/m$.  By the union bound and Theorem~\ref{thm:schema}, with probability at least $1-\alpha$ with respect to the sampling, any $\bx$ which satisfies \eqref{eq:multRob} will satisfy \eqref{eq:multProb}.  

The choice $\epsilon_j = \epsilon/m$ is somewhat arbitrary.  We would prefer to treat the $\epsilon_j$ as decision variables and optimize over them,  i.e., replace the $m$ uncertain constraints by
\begin{align} \notag
\min_{\epsilon_1 + \ldots + \epsilon_m \leq \overline{\epsilon}, \bepsilon \geq \bzero} 
	\Biggr\{ \max_{j=1, \ldots, m} &\Big\{ \max_{\bu \in \U(\S,\epsilon_j, \alpha)} f_j(\bu, \bx) \Big\} \Biggr\} \leq 0
\\  \label{eq:optMultRob}
	&\text{or, equivalently,}
\\ \notag
	\exists \epsilon_1 + \ldots \epsilon_m \leq \overline{\epsilon}, \ \bepsilon \geq \bzero \ : \ &f_j(\bu, \bx) \leq 0 \ \ \forall \bu \in \U(\S, \epsilon_j, \alpha), \ \ j=1, \ldots, m.
\end{align}
Unfortunately, we cannot use Theorem~\ref{thm:schema} to claim that with probability at least $1-\alpha$ with respect to the sampling, any feasible to solution to \eqref{eq:optMultRob} will satisfy \eqref{eq:multProb}.  Indeed, in general, this implication will hold with probability much less than $1-\alpha$.  The issue is that Theorem~\ref{thm:schema} requires selecting $\epsilon$ independently of $\S$, whereas the optimal $\epsilon_j$'s in \eqref{eq:optMultRob} \emph{will} depend on $\S$, creating an in-sample bias.  Consequently, we next extend Theorem~\ref{thm:schema} to lift this requirement.  

Given a family of sets indexed by $\epsilon$, $\{ \U(\epsilon) : 0 < \epsilon < 1 \}$, we say this family \emph{simultaneously} implies a probabilistic guarantee for $\P^*$ if, for all $0< \epsilon < 1$, each $\U(\epsilon)$ implies a probabilistic guarantee for $\P^*$ at level $\epsilon$.
Then, 
\begin{theorem} \label{thm:uniform}
Suppose $\mathcal{P}(\S, \alpha, \epsilon) \equiv \mathcal{P}(\S, \alpha)$ does not depend on $\epsilon$ in Step~\ref{step:I} above. Let $\{\U(\S, \epsilon, \alpha) : 0 < \epsilon < 1 \}$ be the resulting family of sets obtained from the our schema. 
\begin{enumerate}[label=\alph*)]
\item With probability at least $1-\alpha$ with respect to the sampling, $\{\U(\S, \epsilon, \alpha) : 0 < \epsilon < 1 \}$ simultaneously implies a probabilistic guarantee for $\P^*$.  
\item With probability at least $1-\alpha$ with respect to the sampling, any $\bx$ which satisfies \eqref{eq:optMultRob} will satisfy \eqref{eq:multProb}.  
\end{enumerate}
\end{theorem}

In what follows, all of our constructions will simultaneously imply a probabilistic guarantee with the exception of $\U^M_\epsilon$ in Sec.~\ref{sec:marginals}.  We provide numerical evidence in Sec.~\ref{sec:computational} that \eqref{eq:optMultRob} offers significant benefit over \eqref{eq:multRob}.  In some special cases, we can optimize the $\epsilon_j$'s in \eqref{eq:optMultRob} exactly (see Sec.~\ref{sec:queue}).  More generally, we must approximate this outer optimization numerically.  We postpone a treatment of this optimization problem until Sec.~\ref{sec:Multiple} after we have introduced our sets.  

The next four sections apply this schema to create uncertainty sets.  Often, $\epsilon$, $\alpha$ and $\mathcal{S}$ are typically fixed, so we may suppress some or all of them in the notation.
}  

\section{Uncertainty Sets Built from Discrete Distributions}
\label{sec:Discrete}
In this section, we assume $\P^*$ has known, finite support, i.e., $\supp(\P^*) \subseteq \{\ba_0, \ldots, \ba_{n-1}\}$.  We consider two hypothesis tests for this setup: Pearson's $\chi^2$ test and the $G$ test \citep{rice2007mathematical}.  Both tests consider the hypothesis $H_0 : \P^* = \P_0$
where $\P_0$ is some specified measure.  Specifically, let $p_i = \P_0(\burv = \ba_i)$ be the specified null-hypothesis, and let $\hat{\bp}$ denote the empirical probability distribution , i.e., 
\[
\hat{p}_i \equiv \frac{1}{N}\sum_{j=1}^N \I( \buhat^j = \ba_i) \ \ i =0, \ldots, n-1.
\] 
Pearson's $\chi^2$ test rejects $H_0$ at level $\alpha$ if
$N\sum_{i=0}^{n-1} \frac{ (p_i - \hat{p}_i)^2}{p_i} > \chi^2_{n-1, 1-\alpha},
$
where $\chi^2_{n-1, 1-\alpha}$ is the $1-\alpha$ quantile of a $\chi^2$ distribution with $n-1$ degrees of freedom.  
Similarly, the $G$ test rejects the null hypothesis at level $\alpha$ if
$ D(\hat{\bp}, \bp ) > \frac{1}{2N} \chi^2_{n-1, 1-\alpha}
$
where $D(\bp, \bq ) \equiv \sum_{i=0}^{n-1} p_i \log( p_i / q_i) $ is the relative entropy between $\bp$ and $\bq$.  

The confidence regions for Pearson's $\chi^2$ test and the $G$ test are, respectively, 

\begin{equation} \label{def:ChisqFamily}
\mathcal{P}^{\chi^2} = \left\{ \bp \in \Delta_n : \sum_{i=0}^{n-1} \frac{ (p_i - \hat{p}_i)^2}{2p_i} \leq \frac{1}{2N} \chi^2_{n-1, 1-\alpha} \right\},
\quad 
\mathcal{P}^{G} = \left\{ \bp \in \Delta_n : D(\hat{\bp},  \bp ) \leq \frac{1}{2N} \chi^2_{n-1, 1-\alpha} \right\}.
\end{equation}
Here $\Delta_n = \left\{ (p_0, \ldots, p_{n-1})^T : \be^T\bp = 1, \ \ p_i \geq 0 \ \ i = 0, \ldots, n-1 \right\}$ denotes the probability simplex.
We will use these two confidence regions in Step~\ref{step:I} of our schema.  

For a fixed measure $\P$, and vector $\bv \in \R^d$, recall the Conditional Value at Risk:
\begin{equation}
\label{def:CVAR}
\CVAR_{\epsilon}^{\P}(\bv) \equiv \min_{t} \left\{ t + \frac{1}{\epsilon} \E^\P[ (\burv^T\bv - t )^+] \right\}. 
\end{equation}
Conditional Value at Risk is well-known to be a convex upper bound  to Value at Risk \citep{acerbi2002coherence, rockafellar2000optimization} for a fixed $\P$.  We can compute a bound in Step~\ref{step:II} by considering the worst-case Conditional Value at Risk over the above confidence regions, yielding
\begin{theorem} \label{thm:discretesets}
Suppose $\supp(\P^*) \subseteq \{\ba_0, \ldots, \ba_{n-1}\}$.  With probability $1-\alpha$ over the sample, 
the families $\{ \U^{\chi^2}_\epsilon : 0 < \epsilon < 1 \}$ and $\{ \U^{G}_\epsilon : 0 < \epsilon < 1 \}$ simultaneously imply a probabilistic guarantee for $\P^*$, where 
\begin{align}
\label{eq:ChisqU}
\U_\epsilon^{\chi^2} = \left\{ \bu \in \R^d : \bu = \sum_{j=0}^{n-1} q_j \ba_j, \ \bq \in \Delta_n, \ \bq \leq \frac{1}{\epsilon} \bp, \ \bp \in \mathcal{P}^{\chi^2} \right\} , 
\\
\label{eq:GU}
\U_\epsilon^{G} = \left\{ \bu \in \R^d : \bu = \sum_{j=0}^{n-1} q_j \ba_j, \ \bq \in \Delta_n, \ \bq \leq \frac{1}{\epsilon} \bp, \ \bp \in \mathcal{P}^{G} \right\}.
\end{align}
{\blockedit
Their support functions are given by
\begin{equation} \label{eq:SuppFcnChiSq}
\begin{aligned}
\delta^*(\bv | \ \U_\epsilon^{\chi^2}) = \min_{\bw, \eta, \lambda, \bt} \quad & \beta + \frac{1}{\epsilon}\left(\eta + \frac{\lambda \chi^2_{n-1, 1-\alpha}}{N} + 2\lambda - 2 \sum_{i=0}^{n-1} \hat{p}_i s_i \right)
\\ \text{s.t.} \quad & \bzero \leq \bw \leq (\lambda + \eta) \be, \ \  \lambda \geq 0, \ \ \bs \geq \bzero, 
\\
& \left\| \begin{matrix} 2 s_i   \\ w_i  - \eta   \end{matrix} \right\| \leq 2\lambda  - w_i  + \eta, \quad i = 0, \ldots, n-1 
\\ & \ba_i^T \bv - w_i \leq \beta, \quad i = 0, \ldots, n-1,
\end{aligned}
\end{equation} 
\begin{equation}\label{eq:SuppFcnG}
\begin{aligned}
\delta^*(\bv | \ \U_\epsilon^{G} ) = \min_{\bw, \eta, \lambda} \quad &  \beta + \frac{1}{\epsilon}\left(\eta + \frac{\lambda \chi^2_{n-1, 1-\alpha}}{2N} - \lambda \sum_{i=0}^{n-1} \hat{p}_i \log\left(1 - \frac{w_i - \eta}{\lambda}\right) \right)
\\
\text{s.t}\quad & \bzero \leq \bw \leq (\lambda + \eta) \be, \ \ \lambda \geq 0, 
\\ & \ba_i^T \bv - w_i \leq \beta, \quad i = 0, \ldots, n-1.
\end{aligned}
\end{equation} 
}
\end{theorem}

\begin{remark}
The sets $\U_\epsilon^{\chi^2}$, $\U_\epsilon^G$ strongly resemble the uncertainty set for $\CVAR_\epsilon^{\hat{\P}}$ in \cite{bertsimas2009constructing}.  
In fact, as $N \rightarrow \infty$, all three of these sets converge almost surely to the set $\U^{\CVAR_\epsilon^{\P^*}}$ defined by $\delta^*(\bv | \U^{\CVAR_\epsilon^{\P^*}}) = \CVAR^{\P^*}_\epsilon(\bv)$.  The key difference is that for finite $N$, $\U_\epsilon^{\chi^2}$ and $\U_\epsilon^G$ imply a probabilistic guarantee for $\P^*$ at level $\epsilon$, while $\U^{\CVAR_\epsilon^{\hat{\P}}}$ does not.  
\end{remark}
{\blockedit
\begin{remark}
Theorem~\ref{thm:discretesets} exemplifies the distinction drawn in the introduction between uncertainty sets for discrete probability distributions -- such as $\mathcal P^{\chi^2}$ or $\mathcal P^G$ which have been proposed in \cite{ben2013robust} -- and uncertainty sets for general uncertain parameters like $\U^{\chi^2}_\epsilon$ and $\U^G_\epsilon$.  The relationship between these two types of sets is explicit in eqs.~\eqref{eq:ChisqU} and \eqref{eq:GU} because we have known, finite support.  For continuous support and our other sets, the relationship is implicit and must be understood through worst-case value-at-risk in Step~\ref{step:II} of our schema.
\end{remark}

\begin{remark}
When considering $\{(\bv, t) : \delta^*(\bv | \  \U_\epsilon^{\chi^2}) \leq t \}$ or $\{(\bv, t) : \delta^*(\bv | \  \U_\epsilon^{G}) \leq t \}$, we may drop the minimum in the formulation \eqref{eq:SuppFcnChiSq} or \eqref{eq:SuppFcnG}.  Thus, these sets are second-order-cone representable and exponential-cone representable, respectively.  Although theoretically tractable, the exponential cone can be numerically challenging.  
\end{remark}
}

Because of these numerical issues, modeling with $\U_\epsilon^{\chi^2}$ is perhaps preferable to modeling with $\U_\epsilon^{G}$.  Fortunately, for large $N$, the difference between these two sets is negligible:

\begin{proposition}
\label{prop:Taylor}
With arbitrarily high probability, for any $\bp \in \mathcal{P}^G$, $| D(\hat{\bp}, \bp) - \sum_{j=0}^{n-1} \frac{(\hat{p}_j - p_j)^2}{2p_j}| = O(nN^{-3})$.
\end{proposition}
Thus, for large $N$, $\mathcal P^G$ is approximately equal to $\mathcal P^{\chi^2}$, whereby $\U_\epsilon^{G}$ is approximately equal to $\U_\epsilon^{\chi^2}$.  For large $N$, then, $\U_\epsilon^{\chi^2}$ should be preferred for its computational tractability.  

\subsection{A Numerical Example of $\U^{\chi^2}_\epsilon$ and $\U^G_\epsilon$}  Figure~\ref{fig:ChiSq} illustrates the sets $\U^{\chi^2}_\epsilon$ and $\U^G_\epsilon$ with a particular numerical example.  The true distribution is supported on the vertices of the given octagon.  Each vertex is labeled with its true probability.  In the absence of data when the support of $\P^*$ is known, the only uncertainty set $\U$ which implies a probabilistic guarantee for $\P^*$ is the convex hull of these points.  We construct the sets $\U^{\chi^2}_\epsilon$ (grey line) and $\U^G_\epsilon$ (black line) for $\alpha = \epsilon = 10\%$ for various $N$.  For reference, we also plot $\U^{\CVAR_{\epsilon}^{\P^*}}$ (shaded region) which is the limit of both sets as $N\rightarrow \infty$.  
For small $N$, our data-driven sets are equivalent to the convex hull of $\supp(\P^*)$, however, as $N$ increases, our sets shrink considerably.  For large $N$, as predicted by Propostion~\ref{prop:Taylor}, $\U^G_\epsilon$ and  $\U^{\chi^2}_\epsilon$ are very similarly shaped.
\begin{figure}
	\begin{subfigure}{.6 \textwidth}
		\includegraphics[width=\textwidth]{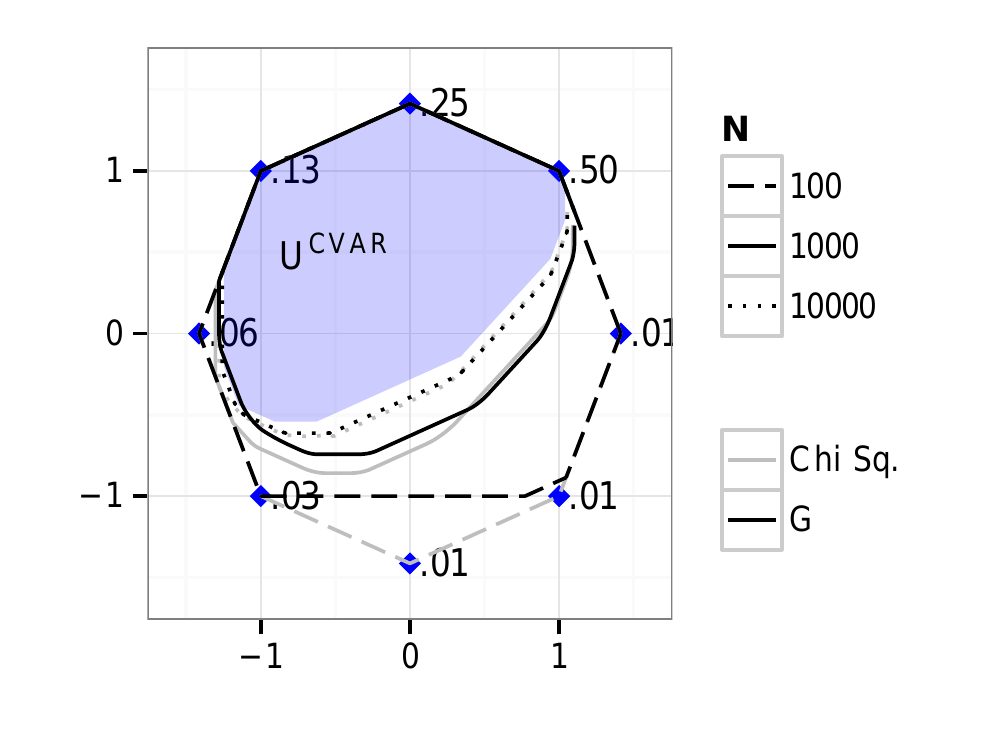}
	\end{subfigure}
	\begin{subfigure}{.4 \textwidth}
		\includegraphics[width= \textwidth]{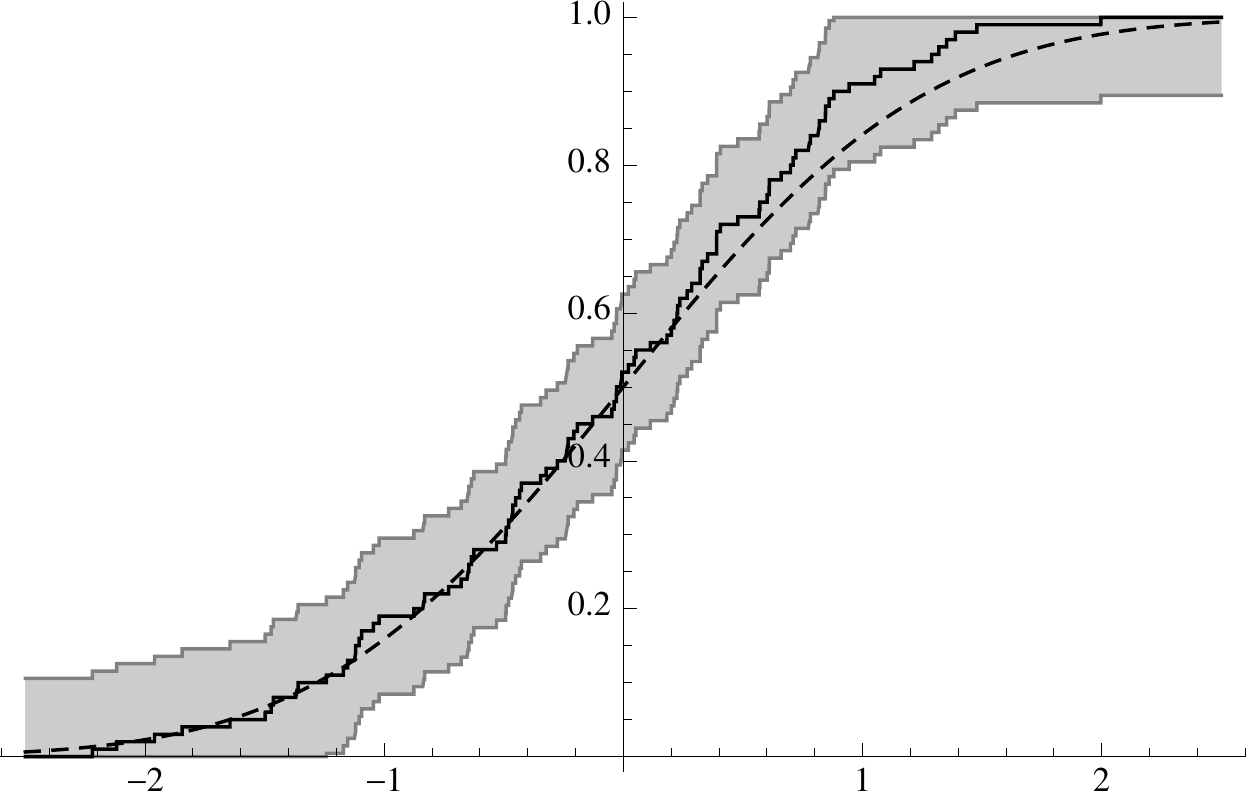}
	\end{subfigure}
\caption{\label{fig:ChiSq} The left panel shows the sets $\U^{\chi^2}_\epsilon$ and $\U^G_\epsilon$, $\alpha=\epsilon=10\%$.  When $N=0$, the smallest set which implies a probabilistic guarantee is $\supp(\P^*)$, the given octagon.  As $N$ increases, both sets shrink to the $\U^{\CVAR_\epsilon^{\P^*}}$ given by the shaded region.  \label{fig:KSFig} The right panel shows the empirical distribution function and confidence region corresponding to the KS test.}
\end{figure}
{\blockedit
\begin{remark} Fig.~\ref{fig:ChiSq} also enables us to contrast our approach to that of \cite{campi2008exact}.  Namely, suppose that $f(\bu, \bx)$ is linear in $\bu$.  In this case, $\bx$ satisfies $f(\hat{\bu}^j, \bx) \leq 0$ for $j=1, \ldots, N$, if and only if $f(\bu, \bx) \leq 0$ for all $\bu \in \text{conv}(\mathcal A)$ where $\mathcal A \equiv \{\ba \in \supp(\P^*) : \exists 1 \leq j \leq N \text{ s.t. } \ba = \hat{\bu}^j \}$.  As $N\rightarrow \infty$, $\mathcal A \rightarrow \supp(\P^*)$ almost surely.  In other words, as $N \rightarrow \infty$, the method of \cite{campi2008exact} in this case is equivalent to using the entire support as an uncertainty set, which is much larger than $\U^{\CVAR_\epsilon^{\P^*}}$ above.  Similar examples can be constructed with continuous distributions or the method of \cite{calafiore2012data}.  In each case, the critical observation is that these methods do not explicitly leverage the concave  (or, in this case, linear) structure of $f(\bu, \bx)$.  
\end{remark}
}

\section{Independent Marginal Distributions}
\label{sec:Independence}
We next consider the case where $\P^*$ may have continuous support, but the marginal distributions $\P^*_i$ are known to be independent.  \edit{ Our strategy is to build up a multivariate test by combining univariate tests for each marginal distribution.}  

\subsection{Uncertainty Sets Built from the Kolmogorov-Smirnov Test}
\edit{For this section, we assume that $\supp(\P^*)$ is contained in a known, finite box $[ \buhat^{(0)}, \buhat^{(N+1)}] \equiv \{ \bu \in \R^d : \uhat_i^{(0)} \leq u_i \leq \uhat_i^{(N+1)}, \ \ i = 1, \ldots, d \}$. }

Given a univariate measure $\P_{0, i}$, the Kolmogorov-Smirnov (KS) goodness-of fit test applied to marginal $i$ considers the null-hypothesis $H_0: \P^*_i = \P_{0, i}$.  It rejects this hypothesis if 
	\begin{equation*}
	\label{eq:KSTest}
	\max_{j = 1, \ldots, N} \max \left( \frac{j}{N} - \P_{0, i}(\urv \leq \uhat_i^{(j)}), \P_{0, i}(\urv < \uhat_i^{(j)}) - \frac{j-1}{N} \right) > \Gamma^{KS}.
	\end{equation*}
where $\uhat_i^{(j)}$ is the $j^\text{th}$ largest element among $\uhat_i^1, \ldots, \uhat_i^N$. 
Tables for the threshold $\Gamma^{KS}$ are widely available \citep{stephens1974edf, thas2010comparing}.

The confidence region of the above test for the $i$-th marginal distribution is
\[
\InfCalP_i^{KS} = \left\{ \P_i \in \Theta[\uhat_i^{(0)}, \uhat_i^{(N+1)}]: 
 	\ \P_i( \urv_i \leq \uhat_i^{(j)} ) \geq \frac{j}{N} - \Gamma^{KS},  \ 
  	\P_i(\urv_i < \uhat_i^{(j)}) \leq \frac{j-1}{N} + \Gamma^{KS}, \ j=1, \ldots, N 
\right\},
\]
where $\Theta[\uhat_i^{(0)}, \uhat_i^{(N+1)}]$ is the set of all Borel probability measures on $[\uhat_i^{(0)}, \uhat_i^{(N+1)}]$.  Unlike $\mathcal{P}^{\chi^2}$ and $\mathcal{P}^G$, this confidence region is infinite dimensional.  

Figure~\ref{fig:KSFig} illustrates an example.  The true distribution is a standard normal whose cumulative distribution function (cdf) is the dotted line.  We draw $N=100$ data points and form the empirical cdf (solid black line). The $80\%$ confidence region of the KS test is the set of measures whose cdfs are more than $\Gamma^{KS}$ above or below this solid line, i.e. the grey region.  

Now consider the multivariate null-hypothesis $H_0: \P^* = \P_0$.  Since $\P^*$ has independent components, the test which rejects if $\P_i$ fails the KS test at level $\alpha^\prime = 1- \sqrt[d]{1-\alpha}$ for any $i$ is a valid test.  Namely, $\P^*_\S( \P^*_i \text{ is accepted by KS at level } \alpha^\prime \text{ for all } i = 1, \ldots ,d ) = \prod_{i=1}^d \sqrt[d]{1-\alpha^\prime} = 1-\alpha$ by independence.  
The confidence region of this multivariate test is
\[
\InfCalP^I = \Big\{ \P \in \Theta [\buhat^{(0)}, \buhat^{(N+1)}] : \P = \prod_{i=1}^d \P_i,
\ \ 
 \P_i \in \InfCalP^{KS}_i  \ \ i=1, \ldots, d
\Big\}.
\]
(``I" in $\mathcal P^I$ is to emphasize independence).  We use this confidence region in Step~\ref{step:I} of our schema.

When the marginals are independent, \cite{nemirovski2006convex} proved
\[
\quant{\epsilon}{\P}(\bv) \leq \inf_{\lambda \geq 0} \left( \lambda \log(1/\epsilon) + \lambda \sum_{i=1}^d \log \E^{\P_i}[e^{v_i\urv_i/\lambda}] \right). 
\]
We use the worst-case value of this bound over $\InfCalP^I$ in Step~\ref{step:II} of our schema.  By passing the supremum through the infimum and logarithm, we obtain
\begin{equation}
\label{eq:IndepConjugate}
\sup_{\P \in \InfCalP^I} \quant{\epsilon}{\P}(\bv) \leq \inf_{\lambda \geq 0} \left( \lambda \log(1/\epsilon) + \lambda \sum_{i=1}^d  \log \sup_{\P_i \in \InfCalP^{KS}_i }\E^{\P_i}[e^{v_i\urv_i/\lambda}] \right).
\end{equation}
Despite the infinite dimensionality, we can solve in the inner-most supremum explicitly by leveraging the simple geometry of $\InfCalP^{KS}_i$.  Intuitively, the worst-case distribution will either be the lefthand boundary or the righthand boundary of the region in Fig.~\ref{fig:KSFig} depending on the sign of $v_i$.  

Specifically, define
\begin{equation}
\label{def:qL}
\begin{aligned}
q^L_j(\Gamma) = \begin{cases} \Gamma & \text{ if } j = 0,
				\\     \frac{1}{N} & \text{ if } 1\leq  j \leq \lfloor N (1- \Gamma) \rfloor,
				\\     1 - \Gamma - \frac{\lfloor N (1- \Gamma) \rfloor}{N} & \text{ if } j =  \lfloor N (1- \Gamma) \rfloor + 1,  
				\\	0	&\text{ otherwise,}
		\end{cases}
\end{aligned} \quad \quad 
\begin{aligned}
q^R_j(\Gamma) = q^L_{N+1-j}(\Gamma),  \ \ j = 0, \ldots, N+1.
\end{aligned}
\end{equation} 
Both $\bq^L(\Gamma), \bq^R(\Gamma) \in \Delta_{N+2}$ so that each vector can be interpreted as a discrete probability distribution on the points $\uhat_i^{(0)}, \ldots, \uhat_i^{(N+1)}$.  One can check that the distributions corresponding to these vectors are precisely the lefthand side and righthand side of the grey region in Fig.~\ref{fig:KSFig}.  
Then, we have
\begin{theorem}
\label{thm:UI}
Suppose $\P^*$ has independent components, with $\supp(\P^*) \subseteq [\buhat^{(0)}, \buhat^{(N+1)}]$.  With probability at least $1-\alpha$ with respect to the sampling,  $\{ \U^I_\epsilon : 0 < \epsilon < 1\}$ simultaneously implies a probabilistic guarantee for $\P^*$, where
\begin{equation} \label{def:UI}
\begin{aligned}
\U^I_\epsilon = \Biggr\{ \bu \in \R^d : &\ \exists \theta_i \in [0, 1], \ \bq^i \in \Delta_{N+2}, \ i = 1\ldots, d,   
\\ &\sum_{j=0}^{N+1} \uhat_i^{(j)} q_j^i = u_i,  \ i = 1, \ldots, d,  
\ \  \sum_{i=1}^d D( \bq_i,  \theta_i \bq^L(\Gamma^{KS}) + (1-\theta_i)\bq^R(\Gamma^{KS}) ) \leq \log( 1/\epsilon)
\Biggr\}.
\end{aligned}
\end{equation}
Moreover, 
\begin{equation} \label{eq:suppUI}
\delta^*(\bv | \ \U^I_\epsilon) = \inf_{\lambda \geq 0} \left\{ \lambda \log(1/\epsilon) + \lambda \sum_{i=1}^d  \log \left[  \max \left( \sum_{j=0}^{N+1} q^L_j(\Gamma^{KS})e^{v_i\uhat_i^{(j)}/\lambda}, 
         \sum_{j=0}^{N+1} q^R_j(\Gamma^{KS})e^{v_i\uhat_i^{(j)}/\lambda} \right) \right] \right\}
\end{equation}
\end{theorem} 
{\blockedit
\begin{remark}
Because $\bq^L(\Gamma)$ (resp. $\bq^R(\Gamma)$) is decreasing (resp. increasing) in its components, the lefthand branch of the innermost maximum in \eqref{eq:suppUI} will be attained when $v_i \leq 0$ and the righthand branch is attained otherwise.  Thus, for fixed $\bv$, the optimization problem in $\lambda$ is convex and differentiable and can be efficiently solved with a line search.
\end{remark}
\begin{remark}\label{rem:UIAlg}
	When representing $\{(\bv, t) : \delta^*(\bv | \ \U^I) \leq t) \}$, we can drop the infimum in \eqref{eq:suppUI}.  Thus, this set is exponential cone representable, which, again, may be numerically challenging.  Using the above line search, however, we can separate over this set:  Given $\bv \in \R^d, t\in \R$ such that $\delta^*(\bv | \ \U^I) > t$, solve \eqref{eq:suppUI} by line search, and let $\lambda^*$ be an optimal solution.  Define
\begin{align*}
\bp^i &= \begin{cases} \bq^L &\text{ if } v_i \leq 0,  \\ \bq^R &\text{otherwise,} \end{cases}
\quad \quad 
q_j^i = \frac{p_j^i e^{v_i \uhat_{i}^{(j)} /\lambda}}{\sum_{j=0}^{N+1} p_j^i e^{v_i \uhat_{i}^{(j)} /\lambda} }, \ \ j = 0, \ldots, N+1, \ \ i = 1, \ldots, d, 
\\
u_i &= \sum_{j=0}^{N+1} q_j^i \uhat_i^{(j)},  \ \ i = 1\ldots, d.
\end{align*}	
Then $\bu \in \U^I_\epsilon$ and $\bu^T\bv \leq t$ is a violated cut for $\{(\bv, t) : \delta^*(\bv | \ \U^I_\epsilon ) \leq t \}$.  That this procedure is valid follows from the proof of Theorem~\ref{thm:UI}, cf. Appendix~\ref{sec:proofsUI}.
\end{remark}

}
\begin{remark}
The KS test is one of many goodness-of-fit tests based on the empirical distribution function (EDF), including the Kuiper (K), Cramer von-Mises (CvM), Watson (W) and Andersen-Darling (AD) tests \citep[][Chapt. 5]{thas2010comparing}.  We can define analogues of $\U^I_\epsilon$ for each of these tests, each having slightly different shape.  \edit{Separating over $\{(\bv, t) : \delta^*(\bv | \ \U) \leq t\}$ is polynomial time tractable for each these sets, but we no longer have a simple algorithm for generating violated cuts.  Thus, these sets are considerably less attractive from a computational point of view.}  
Fortunately, through simulation studies with a variety of different distributions, we have found that the version of $\U^I_\epsilon$ based on the KS test generally performs as well as or better than the other EDF tests.   Consequently, we recommend using the sets $\U^I_\epsilon$ as described.  For completeness, we present the constructions for the analogous tests in Appendix~\ref{sec:OtherEDF}.  
\end{remark}

\subsection{Uncertainty Sets Motivated by Forward and Backward Deviations}
\label{sec:FwdBack}
{ \blockedit
In \cite{chen2007robust}, the authors propose an uncertainty set based on the forward and backward deviations of a distribution.  They focus on a non-data-driven setting, where the mean and support of $\P^*$ are known a priori, and show how to upper bound these deviations to calibrate their set.  In a setting where one has data \emph{and a priori knows the mean of} $\P^*$ \emph{precisely}, they propose a method based on sample average approximation to estimate these deviations.  Unfortunately, the precise statistical behavior of these estimators is not known, so it is not clear that this set calibrated from data implies a probabilistic guarantee with high probability with respect to the sampling.   

In this section, we use our schema to generalize the set of \cite{chen2007robust} to a data-driven setting where \emph{neither the mean of the distribution nor its support are known.}  Our set differs in shape and size from their proposal, and, our construction, unlike their original proposal, will simultaneously imply a probabilistic guarantee for $\P^*$.  
}

We begin by specifying an appropriate multivariate hypothesis test based on combining univariate tests.  Specifically, for a known (univariate) distribution $\P_i$ define its forward and backward deviations by 
\begin{equation} \label{eq:FwdBack}
\sigma_{fi}(\P_i) = \sup_{x > 0 } \sqrt{-\frac{2\mu_i}{x} + \frac{2}{x^2} \log( \E^{\P_i}[ e^{x \urv_i}] )} , 
\quad
\sigma_{bi}(\P_i)  = \sup_{x > 0 } \sqrt{\frac{2\mu_i}{x} + \frac{2}{x^2} \log( \E^{\P_i}[ e^{-x \urv_i}] )}, 
\end{equation}
where $ \E^{\P_i}[\urv_i] = \mu_i$.  Notice the optimizations defining $\sigma_{fi}(\P_i), \sigma_{bi}(\P_i)$ are one dimensional, convex problems which can be solved by a line search.  \edit{A sufficient, but not necessary, condition for $\sigma_{fi}(\P_i), \sigma_{bi}(\P_i)$ to be finite  is that $\P_i$ has bounded support \citep[c.f.][]{chen2007robust}.  To streamline the exposition, we assume throughout this section $\P^*$ has bounded (but potentially unknown) support.}

For a given $\mu_{0,i}, \sigma_{0, fi}, \sigma_{0, bi} \in \R$, consider the following three null-hypotheses:
\begin{equation} \label{eq:NullTriplet}
H_0^1: \E^{\P^*_i}[\urv] = \mu_{0, i}, \ \ H_0^2: \sigma_{fi}(\P^*_i) \leq \sigma_{0, fi}, \ \  H_0^3: \sigma_{bi}(\P^*_i) \leq \sigma_{0, bi}.
\end{equation}
We can test these hypotheses (separately) using $| \hat{\mu}_i - \mu_{0,i} |$, $\sigma_{fi}(\hat{\P}_i)$ and $\sigma_{bi}(\hat{\P}_i)$, respectively, as test statistics.  Since these are not common hypothesis tests in applied statistics, there are no tables for their thresholds.  Instead, 
we compute approximate thresholds $t_i$, $\overline{\sigma}_{fi}$ and $\overline{\sigma}_{bi}$ at the $\alpha/2$, $\alpha/4$ and $\alpha/4$ significance level, respectively, using the bootstrap procedure in Algorithm~\ref{alg:Bootstrap}.  

By the union bound, the univariate test  which rejects if any of these thresholds is exceeded is a valid test at level $\alpha$ for the three hypotheses above to hold simultaneously.  The confidence region of this test is
\[
\mathcal{P}^{FB}_i = \{\P_i \in \Theta(-\infty, \infty) : m_{bi} \leq \E^\P_i[ \urv_i] \leq m_{fi}, \ \ 
\sigma_{fi}(\P_i) \leq \overline{\sigma}_{fi}, \ \ \sigma_{bi}(\P_i) \leq \overline{\sigma}_{bi} \}, 
\]
where $m_{bi} = \hat{\mu}_i - t_i$ and $m_{fi} = \hat{\mu}_i + t_i$.  

Next, consider the multivariate null-hypothesis that all three null-hypotheses in \eqref{eq:NullTriplet} hold simultaneously for all $i=1, \ldots, d$.  As in Sec.~\ref{sec:Independence}, the test which rejects if the above univariate test rejects at level $\alpha^\prime = 1 - \sqrt[d]{1-\alpha}$ for any $i$ is a valid test.  Its confidence region is $\mathcal{P}^{FB} = \{ \P : \P_i \in \mathcal{P}^{FB}_i \ i = 1, \ldots, d \}.$  We will use this confidence region in Step~\ref{step:I} of our schema.  

When the mean and deviations for $\P$ are known and the marginals are independent, \citet{chen2007robust} prove
\begin{equation} \label{eq:ChenBound}
\quant{\epsilon}{\P}(\bv) \leq \sum_{i=1}^d \E^\P[\urv_i] v_i + \sqrt{ 2 \log(1/\epsilon) \left( \sum_{i : v_i < 0 } \sigma_{bi}^2(\P) v_i^2 +  \sum_{i : v_i \geq 0 }  \sigma_{fi}^2(\P) v_i^2 \right) }.
\end{equation}
Computing the worst-case value of this bound over the above confidence region in Step~\ref{step:II} of our schema yields:
%
%
\begin{theorem}
\label{thm:UFB}
Suppose $\P^*$ has independent components and bounded support.  
With probability $1-\alpha$ with respect to the sample, \edit{the family $\{\U^{FB}_\epsilon : 0 < \epsilon < 1 \}$ simultaneously implies a probabilistic guarantee for $\P^*$}, where 
\begin{equation} \label{def:UFB}
\U^{FB}_\epsilon = \left\{ \by_1 + \by_2 - \by_3 :  \by_2, \by_2 \in \R^d_+, 
\ \
\sum_{i=1}^d \frac{y_{2i}^2 }{2 \overline{\sigma}_{fi}^2} + \frac{y_{3i}^2 }{2 \overline{\sigma}_{bi}^2} \leq \log(1/\epsilon),
\ \
m_{bi} \leq y_{1i} \leq m_{fi}, \ \ i = 1, \ldots, d 
\right\}.
\end{equation}
{\blockedit Moreover,
\begin{equation} \label{eq:suppFcnFB}
\delta^*(\bv | \ \U^{FB}_\epsilon) =  \sum_{i : v_i \geq 0} m_{fi} v_i + \sum_{i: v_i < 0 } m_{bi} v_i + \sqrt{ 2 \log(1/\epsilon) \left(\sum_{i: v_i \geq 0} \overline{\sigma}_{fi}^2 v_i^2 + \sum_{i: v_i < 0 } \overline{\sigma}_{bi}^2 v_i^2\right)}
\end{equation} }
\end{theorem}
{\blockedit
\begin{remark} \label{rem:UFBAlg}
From \eqref{eq:suppFcnFB},  $\{(\bv, t) : \delta^*(\bv | \ \U^{FB}_\epsilon) \leq t \}$ is second order cone representable. We can separate over this constraint in closed-form:  Given $\bv, t$, use \eqref{eq:suppFcnFB} to check if $\delta^*(\bv | \ \U^{FB}_\epsilon) > t$.  If so, let
\begin{align*}
\lambda = \sqrt{\frac{\sum_{i: v_i > 0 } v_i^2 \overline{\sigma}_{fi}^2 + \sum_{i: v_i \leq 0 } v_i^2 \overline{\sigma}_{bi}^2}{2 \log (1/\epsilon)} },
\quad \quad 
u_i = \begin{cases} m_{fi} + \frac{v_i \overline{\sigma}_{fi}^2}{\lambda} &\text{ if } v_i > 0
			\\\ m_{bi} + \frac{v_i \overline{\sigma}_{bi}^2}{\lambda}  & \text{ otherwise.}
\end{cases}
\end{align*}
Then, $\bu^T\bv \leq t$ is a violated constraint.  The correctness of this procedure follows from the proof of Theorem~\ref{thm:UFB}.
\end{remark}
\begin{remark}  \label{rem:Support} There is no guarantee that $\U^{FB}_\epsilon \subseteq \supp(\P^*)$.  Consequently, if we have a priori information of the support, we can use this to refine $\U^{FB}_\epsilon$.  Specifically, let $\U_0$ be convex, compact such that $\supp(\P^*) \subseteq \U_0$.  Then, the family $\{\U^{FB}_\epsilon \cap \U_0 : 0 < \epsilon < 1 \}$ simultaneously implies a probabilistic guarantee.  Moreover, for common $\U_0$, optimizing over \eqref{eq:NonlinReform} with $\U^{FB}_\epsilon \cap \U_0$ is computationally similar to optimizing with $\U^{FB}_\epsilon$.  More precisely, from \citep[][Lemma A.4]{ben2012deriving}, $\{(\bv, t) : \delta^*(\bv | \ \U_\epsilon(\S) \cap \U_0) \}$ is equivalent to
\begin{align} \label{eq:ReformIntersect}
\left\{ (\bv, t ) : \exists \bw, \in \R^d, \ \  t_1, t_2 \in \R \text{ s.t. } \delta^*(\bv - \bw | \ \U_\epsilon(\S)) \leq t_1, \  \ \delta^*(\bw | \ \U_0) \leq t_2, \  \ t_1 + t_2 \leq t \right\},
\end{align}
so that \eqref{eq:NonlinReform} with $\U^{FB}_\epsilon \cap \U_0$ will be tractable whenever $\{(\bv, t) : \delta^*(\bv | \ \U_0) \leq t \}$ is tractable, examples of which include when $\U_0$ is a norm-ball, ellipse, or polyhedron (see \cite{ben2012deriving}).
\end{remark}
}

\subsection{Comparing  $\U^I_\epsilon$ and $\U^{FB}_\epsilon$}  
Figure~\ref{fig:UI} illustrates the sets $\U^I_\epsilon$ and $\U^{FB}_\epsilon$ numerically.  The marginal distributions of $\P^*$ are independent and their densities are given in the left panel.  Notice that the first marginal is symmetric while the second is highly skewed.  

{\blockedit In the absence of any data, knowing only $\supp(\P^*)$ and that $\P^*$ has independent components, the smallest uncertainty which implies a probabilistic guarantee is the unit square (dotted line).  With $N=100$ data points from this distribution (blue circles),  however, we can construct both $\U^I_\epsilon$ (dashed black line) and $\U^{FB}_\epsilon$ (solid black line) with  $\epsilon = \alpha = 10\%$, as shown.  We also plot the limiting shape of these two sets as $N \rightarrow \infty$ (corresponding grey lines).  
\begin{figure}
\centering
\begin{subfigure}{0.34\textwidth}
	\includegraphics[width=\textwidth]{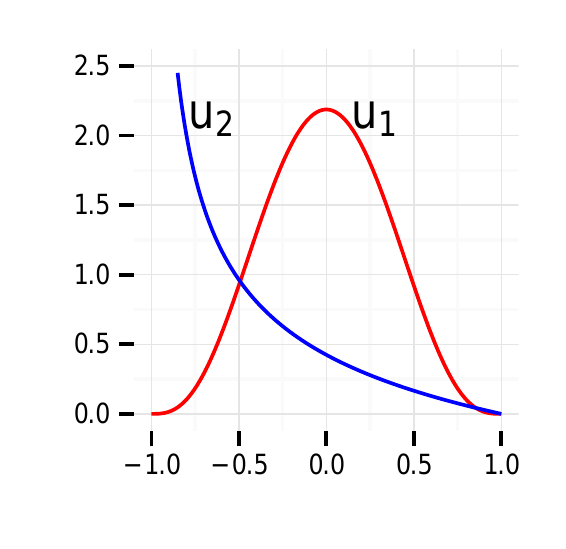}
\end{subfigure}
\begin{subfigure}{.65\textwidth}
	\includegraphics[width=\textwidth]{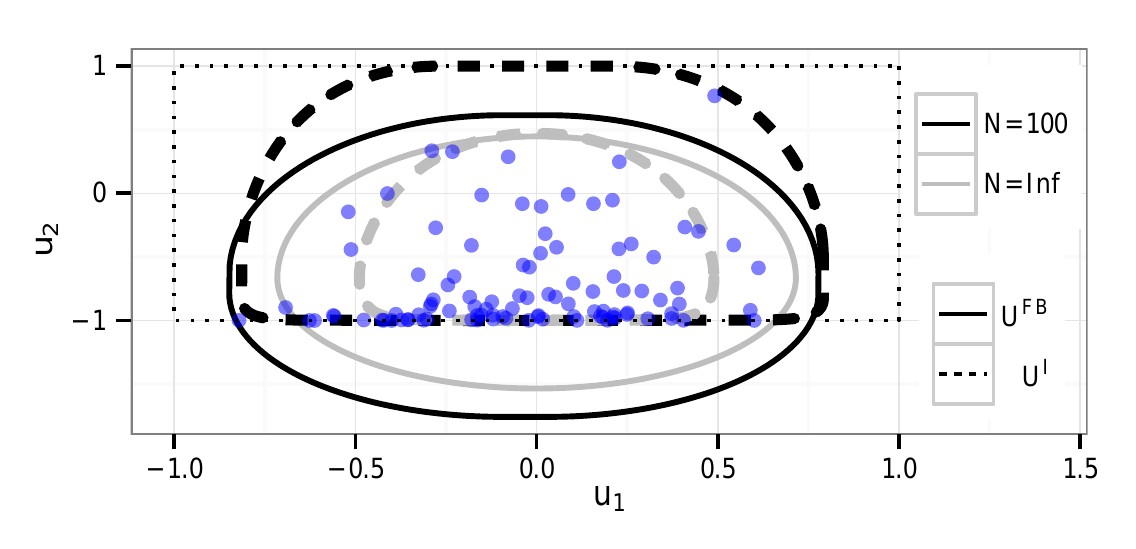}
\end{subfigure}
\caption{\label{fig:UI} The left panel shows the marginal densities. The right panel shows $\U^I_\epsilon$ (dashed black line) and $\U^{FB}_\epsilon$
(solid black line)  built from $N=100$ data points (blue circles) and in the limit as $N\rightarrow \infty$ (corresponding blue lines).}  
\end{figure}

Several features are evident from the plots. First, both sets are able to learn that $\P^*$ is symmetric in its first coordinate (the sets display vertical symmetry) and that $\P^*$ is skewed downwards in its second coordinate (the sets taper more sharply towards the top). Both sets \emph{learn} these features from the data.  Second, although $\U^I_\epsilon$ is a strict subset of $\supp(\P^*)$, $\U^{FB}_\epsilon$ is not (see also Remark~\ref{rem:Support}).
Finally, neither set is a subset of the other, and, although for $N=100$,  $\U^{FB}_\epsilon \cap \supp(\P^*)$ has smaller volume than $\U^I_\epsilon$, the reverse holds for larger $N$.  Consequently, it is not clear which set to prefer in a given application, and the best choice likely depends on $N$.  
}
\section{Uncertainty Sets Built from Marginal Samples}
\label{sec:marginals}
In this section, we observe samples from the marginal distributions of $\P^*$ separately, but do not assume these marginals are independent.  This happens, e.g., when samples are drawn asynchronously, or when there are many missing values.  In these cases, it is impossible to learn the joint distribution of $\P^*$ from the data.  To streamline the exposition, we assume that we observe exactly $N$ samples of each marginal distribution.  The results generalize to the case of different numbers of samples at the expense of more notation.  

In the univariate case, \citet{david1970order} develop a hypothesis test for the $1-\epsilon/d$ quantile, or equivalently $\quant{\epsilon/d}{\P_i}(\be_i)$ of a distribution $\P$.  Namely, 
given $ \overline{q}_{i,0} \in \R$, consider the hypothesis
$H_{0, i}: \quant{\epsilon/d}{\P^*}(\be_i) \geq \overline{q}_{i,0}$. Define the index $s$ by 
\begin{equation}
	\label{def:RS}
	s = \min \left\{ k \in \mathbb{N}: \sum_{j = k}^N \binom{N}{j} (\epsilon/d)^{N-j} (1-\epsilon/d)^{j} \leq \frac{\alpha}{2d} \right\},
\end{equation}
and let $s = N+1$ if the corresponding set is empty.  Then, the test which rejects if $q_{i, 0} > \uhat_i^{(s)}$ is valid at level $\alpha/2d$ \cite[][Sec. 7.1]{david1970order}.  
\cite{david1970order} also prove that $\frac{s}{N} \downarrow(1-\epsilon/ d$).  

The above argument applies symmetrically to the hypothesis $H_{0, i}:  \quant{\epsilon/d}{\P^*}(-\be_i) \geq \underline{q}_{i,0}$ where the rejection threshold now becomes $\uhat_i^{(N-s + 1)}$.  
In the typical case when   $\epsilon/d$ is small, $N - s +1 < s$ so that $\uhat_i^{(N-s + 1)} \leq \uhat_i^{(s)}$.  

Next given $\overline{q}_{i, 0}, \underline{q}_{i, 0} \in \R$ for $i=1, \ldots, d$, consider the multivariate hypothesis:
\begin{align*}
	&H_0: \quant{\epsilon/d}{\P^*}(\be_i) \geq \overline{q}_{i, 0} \text{ and } \quant{\epsilon/d}{\P^*}(-\be_i) \geq \underline{q}_{i, 0} \text{ for all } i = 1, \ldots, d.
	\end{align*}
By the union bound, the test which rejects if  $\uhat_i^{(s)} < \overline{q}_i$ or $-\uhat_i^{(N-s+1)} < \underline{q}_i$, i.e., the above tests fail for the $i$-th component, is valid at level $\alpha$.  
Its confidence region is 
\begin{align*}
\InfCalP^M = \left\{ \P \in \Theta[\buhat^{(0)}, \buhat^{(N+1)}]: \ \ \quant{\epsilon/d}{\P_i} \leq \uhat_i^{(s)}, \ \ \quant{\epsilon/d}{\P_i} \geq \uhat_i^{(N-s+1)}, \ \ i=1, \ldots, d \right\}.
\end{align*} 
Here ``M" is to emphasize ``marginals."  We use this confidence region in Step 1 of our schema.  

When the marginals of $\P$ are known, \citet{embrechts2003Using} proves 
\begin{equation}
\label{eq:boundSum}
\quant{\epsilon}{\P}(\bv) \leq \min_{\blambda : \be^T\blambda = \epsilon} \sum_{i=1}^d \quant{\lambda_i}{\P}(v_i \be_i).
\end{equation}
Since the minimization on the right-hand side can be difficult, we will use the weaker bound
$
\quant{\epsilon}{\P}(\bv) \leq \sum_{i=1}^d\quant{\epsilon/d}{\P}(v_i \be_i)
$
obtained by letting $\lambda_i = \epsilon/d$ for all $i$.  

We compute the worst case value of this bound over $\InfCalP^M$, yielding:
\begin{theorem} \label{thm:UM}
If $s$ defined by Eq.~\eqref{def:RS} satisfies $N-s+1 < s$, then, with probability at least $1-\alpha$ over the sample, the set 
\begin{equation}
\label{def:MarginalU}
\U^M_\epsilon = \left\{ \bu \in \R^d: \uhat_i^{(N-s+1)} \leq u_i \leq \uhat_i^{(s)}  \ \ i =1, \dots, d \right\}.
\end{equation}
 implies a probabilistic guarantee for $\P^*$ at level $\epsilon$.  Moreover, 
 \begin{equation} \label{eq:suppFcnUM}
 \delta^*(\bv | \ \U^M_\epsilon) = \sum_{i=1}^d \max(v_i \uhat_i^{(N-s+1)}, v_i \uhat_i^{(s)}).
\end{equation}
\end{theorem}
\begin{remark} Notice that the family $\{ \U^M_\epsilon : 0 < \epsilon < 1 \}$, may \emph{not} simultaneously imply a probabilistic guarantee for $\P^*$ because the confidence region $\InfCalP^M$ depends on $\epsilon$.
\end{remark}
\begin{remark}
The set $\{(\bv, t) : \delta^*(\bv | \U^M) \leq t \}$ is a simple box, representable by linear inequalities.  We can separate over this set in closed form via \eqref{eq:suppFcnUM}.  
\end{remark}

\section{Uncertainty Sets for Potentially Non-independent Components}
\label{sec:Correlated}
{\blockedit 
In this section, we assume we observe samples drawn from the joint distribution of $\P^*$ which may have unbounded support.  We consider a goodness-of-fit hypothesis test based on linear-convex ordering proposed in \cite{BGKII}.  Specifically, given some multivariate $\P_0$, consider the null-hypothesis $H_0: \P^* = \P_0$.  \citet{BGKII} prove that the test which rejects $H_0$ if 
$\exists (\ba, b) \in \mathcal{B} \equiv \{\ba \in \R^d, b \in \R : \| \ba \|_1 + | b | \leq 1\}$ such that 
\[
\E^{\P_0}[ (\ba^T\burv - b)^+]  - \frac{1}{N} \sum_{j=1}^N (\ba^T\buhat^j - b) ^+ >  \Gamma_{LCX} \ \ \text{ or } \ \ 
\frac{1}{N} \sum_{j=1}^N (\buhat^j)^T \buhat^j -\E^{\P_0}[ \burv^T \burv ] > \Gamma_\sigma
\]
 for appropriate thresholds $\Gamma_{LCX}, \Gamma_\sigma$ is a valid test at level $\alpha$.  The authors provide an explicit bootstrap algorithm to compute $\Gamma_{LCX}, \Gamma_\sigma$.  

The confidence region of this test is
\begin{align} \notag
\InfCalP^{LCX} = \Biggr\{ \P \in \Theta(\R^d): \ &\E^\P[ ( \ba^T\burv - b)^+ ] \leq \frac{1}{N} \sum_{j=1}^N (\ba^T\buhat_j - b)^+ + \Gamma_{LCX} \ \ \forall (\ba, b) \in \mathcal{B},
\\ \label{eq:ConfRegionLCX}
&\sum_{i=1}^d \E^\P[ \| \burv \|^2 ] \geq \frac{1}{N}\sum_{j=1}^N \| \buhat_j \|^2 ] - \Gamma_\sigma
\Biggr\},
\end{align}
We will use this confidence region in Step~\ref{step:I} of our schema.

Combining techniques from semi-infinite optimization with our schema (see electronic companion for proof), we obtain
\begin{theorem}
\label{thm:ULCX}
The family $\{ \U^{LCX}_\epsilon : 0 < \epsilon < 1 \}$ simultaneously implies a probabilistic guarantee for $\P^*$ where
\begin{subequations}\label{eq:ULCX}
\begin{align} 
\U^{LCX}_\epsilon = \Biggr\{& \bu \in \R^d: \ \exists \br \in \R^d, \ 1 \leq z \leq 1/\epsilon,  \ \text{ s.t. } 
\\ \label{eq:ULCX2}
& (\ba^T \br - b (z-1))^+  + (\ba^T \bu -b)^+ \leq \frac{z}{N} \sum_{j=1}^N (\ba^T\buhat_j - b)^+ + \Gamma_{LCX}, \ \forall (\ba, b) \in \mathcal{B} \Biggr\}.
\end{align}
\end{subequations}
Moreover,
\begin{align} \notag
\delta^*(\bv | \ \U^{LCX}_\epsilon) = \sup_{\P \in \InfCalP^{LCX}} \quant{\epsilon}{\P}(\bv) = \min_{\tau, \theta, y_1, y_2, \lambda} \quad & \frac{1}{\epsilon} \tau - \theta - \int_{\mathcal{B}} b dy_1(\ba, b) + \int_{\mathcal B } b dy_2(\ba, b)
\\ \notag
\text{s.t.} \quad & \theta + \int_{\mathcal B} b dy_1(\ba, b) + \int_{\mathcal B} \Gamma(\ba, b) d\lambda(\ba, b) \leq \tau
\\ \label{eq:suppLCX}
&0 \leq dy_1(\ba, b) \leq d\lambda(\ba, b) \quad \forall (\ba, b) \in \mathcal{B}, 
\\ \notag
&0 \leq dy_2(\ba, b) \leq d\lambda(\ba, b) \quad \forall (\ba, b) \in \mathcal{B},
\\ \notag
&\int_{\mathcal{B}} \ba \  dy_1(\ba, b) = 0,
\ \
\bv = \int_{\mathcal{B}} \ba \ d y_2(\ba, b),
\\ \notag
& \theta, \tau \geq 0.
\end{align}

\end{theorem}  
\begin{remark} As the intersection of convex constraints, $\U^{LCX}_\epsilon$ is convex.
\end{remark}
\begin{remark}
It is possible to separate over \eqref{eq:ULCX2} efficiently.  Specifically, fix $\bu, \br \in \R^d$ and $1\leq z \leq 1/\epsilon$.  We identify the worst-case $(\ba, b) \in \mathcal{B}$ in \eqref{eq:ULCX2} by solving three auxiliary optimization problems:
\begin{align*} \notag
\xi_1 = \max_{(\ba, b) \in \mathcal{B}, \bt \geq \bzero} \quad &\ba^T\br - b ( z-1)) + (\ba^T \bu - b) - \frac{z}{N}\sum_{j=1}^N t_j
\\ \notag
\text{s.t.} \quad & t_j \geq \ba^T\hat{\bu}_j -b,
\ \ 
 \ba^T \bu - b \geq 0,
\ \ 
\ba^T \br - b(z-1) \geq 0, 
\\
\xi_2 = \max_{(\ba, b) \in \mathcal{B}, \bt \geq \bzero} \quad &\ba^T\br - b ( z-1)) - \frac{z}{N}\sum_{j=1}^N t_j
\\ \notag
\text{s.t.} \quad & t_j \geq \ba^T\hat{\bu}_j -b,
\ \  
 \ba^T \bu - b \leq 0,
\ \ 
\ba^T \br - b(z-1) \geq 0, 
\\
\xi_3 = \max_{(\ba, b) \in \mathcal{B}, \bt \geq \bzero} \quad & (\ba^T \bu - b) - \frac{z}{N}\sum_{j=1}^N t_j
\\ \notag
\text{s.t.} \quad & t_j \geq \ba^T\hat{\bu}_j -b,
\ \  
 \ba^T \bu - b \geq 0,
\ \ 
\ba^T \br - b(z-1) \leq 0, 
\end{align*}
corresponding to the potential signs of $\ba^T\br - b(z-1)$ and $\ba^T\bu - b$ at the worst-case value.  (The fourth case, where both terms are negative is trivial since $\Gamma_{LCX}  > 0$.)  Each of these optimization problems can be written as linear optimizations.  If $\max(\xi_1, \xi_2, \xi_3) \leq \Gamma_{LCX}$, then $\bu, \br$ and $z$ are feasible in \eqref{eq:ULCX2}.  Otherwise, the optimal $\ba, b$ in the maximizing subproblem yields a violated cut.  
\end{remark}

\begin{remark} \label{rem:SeparateLCX} The representation of $\delta^*(\bv | \ \U^{LCX})$ is not particularly convenient.  Nonetheless, we can 
 separate over $\{(\bv, t) :  \delta^*(\bv | \ \U^{LCX}) \leq t \}$ in polynomial time by using the above separation routine with the ellipsoid algorithm to solve $\max_{\bu \in \U^{LCX}} \bv^T\bu$.  Alternatively, combining the above separation routine with the dual-simplex algorithm yields a practically efficient algorithm for large-scale instances
\end{remark}
}

\section{Hypothesis Testing: A Unifying Perspective}
\label{sec:HypTestPerspective}
{\blockedit Several data-driven methods in the literature create families of measures $\mathcal{P}(\S)$ that contain $\P^*$ with high probability.  These methods do not explicitly reference hypothesis testing.  In this section, we provide a hypothesis testing interpretation of two such methods \citep{shawe2003estimating, delage2010distributionally}.  Leveraging this new perspective, we show how standard techniques for hypothesis testing, such as the bootstrap, can be used to improve upon these methods.  Finally, we illustrate how our schema can be applied to these improved family of measures to generate new uncertainty sets.  To the best of our knowledge, generating uncertainty sets for \eqref{eq:NonlinearRobust} is a new application of both \citep{shawe2003estimating, delage2010distributionally}.  }

The key idea in both cases is to recast $\mathcal{P}(\S)$ as the confidence region of a  hypothesis test.  This correspondence is not unique to these methods.  There is a one-to-one correspondence between families of measures which contain $\P^*$ with probability at least $1-\alpha$ with respect to the sampling and the confidence regions of hypothesis tests.  This correspondence is sometimes called the ``duality between confidence regions and hypothesis testing" in the statistical literature \citep{rice2007mathematical}.  It implies that any data-driven method predicated on a family of measures that contain $\P^*$ with probability $1-\alpha$ can be interpreted in the light of hypothesis testing. 

This observation is interesting for two reasons.  First, it provides a unified framework to compare distinct methods in the literature and ties them to the well-established theory of hypothesis testing in statistics. Secondly, there is a wealth of practical experience with hypothesis testing.  In particular, we know empirically which tests are best suited to various applications and which tests perform well even when the underlying assumptions on $\P^*$ that motivated the test may be violated.  \edit{In the next section, we leverage some of this practical experience with hypothesis testing to strengthen these methods, and then derive uncertainty sets corresponding to these hypothesis tests to facilitate comparison between the approaches. }

\subsection{Uncertainty Set Motivated by Cristianini and Shawe-Taylor, 2003}
\label{sec:ElGhaouiSet}
Let $\| \cdot \|_F$ denote the Frobenius norm of matrices.  As part of a particular machine learning application, 
\citet{shawe2003estimating} prove
\begin{theorem}[Cristianini and Shawe-Taylor, 2003]
\label{thm:CS}
Suppose that $\supp(\P^*)$ is contained within the ball of radius $R$ and that $N > (2 + 2 \log(2/\alpha))^2.$  Then, with probability at least $1-\alpha$ with respect to the sampling, 
\[
\mathcal{P}^{CS} = \{ \P \in \Theta(R) : \| \E^\P[\burv] - \hat{\bmu} \|_2 \leq \Gamma_1(\alpha/2, N)
 \text{ and }  
  \| \E^\P[ \burv \burv^T] - \E^\P[\burv]\E^\P[\burv^T] - \hat{\bSigma} \|_F \leq \Gamma_2(\alpha/2, N),
\]	
where 
$\hat{\bmu}, \hat{\bSigma}$ denote the sample mean and covariance, 
$\Gamma_1(\alpha, N) = \frac{R}{\sqrt{N}} \left(2 + \sqrt{2 \log 1/\alpha} \right)$,  
$\Gamma_2(\alpha, N) = \frac{2R^2}{\sqrt{N}} \left(2 + \sqrt{2 \log 2/\alpha} \right)$, and $\Theta(R)$ denotes the set of Borel probability measures supported on the  ball of radius $R$.  
\end{theorem}
The key idea of their proof is to use a general purpose concentration inequality (McDiarmid's inequality) to compute $\Gamma_1(\alpha, N)$, $\Gamma_2(\alpha, N)$.  

We observe that $\mathcal{P}^{CS}$ is the $1-\alpha$ confidence region of a hypothesis test for the mean and covariance of $\P^*$.  Namely,  the test considers
\begin{equation} \label{eq:MeanAndCovar}
H_0 : \E^{\P^*}[ \burv] = \bmu_0 \text{ and } \E^{\P^*}[ \burv \burv^T] - \E^{\P^*}[\burv] \E^{\P^*}[\burv^T] = \bSigma_0,
\end{equation}
using statistics $\| \hat{\bmu} - \bmu_0 \|$ and $\| \hat{\bSigma} - \bSigma_0  \|$ and thresholds $\Gamma_1(\alpha/2, N), \Gamma_2(\alpha/2, N)$.  

{\blockedit Practical experience in applied statistics suggests, however, that tests whose thresholds are computed as above using general purpose concentration inequalities, while valid, are typically very conservative for reasonable values of $\alpha$, $N$.  They reject $H_0$ when it is false only when $N$ is very large.  The standard remedy is to use the bootstrap (Algorithm~\ref{alg:Bootstrap}) to calculate alternate thresholds $\Gamma_1^{B}, \Gamma_2^{B}$.  These bootstrapped thresholds are typically much smaller, but still (approximately) valid at level $1-\alpha $.  The first five columns of Table~\ref{tab:BootstrapThreshCS} illustrates the magnitude of the difference with a  particular example.  
Entries of $\infty$ indicate that the threshold as derived in \cite{shawe2003estimating} does not apply for this value of $N$.  
The data 
are drawn from a standard normal distribution with $d=2$ truncated to live in a ball of radius $9.2$.  We take $\alpha = 10\%$, $N_B = 10,000$.  We can see that the reduction can be a full-order of magnitude, or more.

Reducing the thresholds $\Gamma_1^B, \Gamma_2^B$ shrinks $\mathcal{P}^{CS}$, in turn reducing the ambiguity in $\P^*$.  This reduction ameliorates the potential over-conservativeness of any method using $\mathcal{P}^{CS}$, including the original machine learning application of \cite{shawe2003estimating} and our own schema for developing uncertainty sets.  }
\begin{table}
\TABLE
{\label{tab:BootstrapThreshCS} Comparing Thresholds with and without bootstrap using $N_B = 10,000$ replications, $\alpha=10
\%$.}
{	\begin{tabular}{rp{.15cm}ccccp{.15cm}cccc}
\toprule
&& \multicolumn{4}{c}{\pbox{20cm}{{ \footnotesize Shawe-Taylor \& Cristianini (2003)}}}    & &
\multicolumn{4}{c}{\pbox{20cm}{ \footnotesize Delage \& Ye (2010) }}
 \\[4pt]
N   & & $\Gamma_1$ & $\Gamma_2$  & $ \Gamma_1^B$   & $\Gamma_2^B$  && $\gamma_1$ & $\gamma_2$  & $ \gamma_1^B$   & $\gamma_2^B$ \\
\midrule
10 && $\infty$ & $\infty$ & 0.805 & 1.161 && $\infty$ & $\infty$ & 0.526 & 5.372
\\
50 && $\infty$ & $\infty$ & 0.382 & 0.585 && $\infty$ & $\infty$ & 0.118 & 1.684
\\
100 && 3.814 & 75.291 & 0.262 & 0.427 && $\infty$ & $\infty$ & 0.061 & 1.452
\\
500 && 1.706 & 33.671 & 0.105 & 0.157 && $\infty$ & $\infty$ & 0.012 & 1.154
\\
50000 && 0.171 & 3.367 & 0.011 & 0.018 && $\infty$ & $\infty$ & 1e-4 & 1.015
\\
100000 && 0.121 & 2.381 & 0.008 & 0.013 && 0.083 & 5.044 & 6e-5 & 1.010
\\
\bottomrule
\end{tabular}
 }
{}
\end{table}

We next use $\mathcal{P}^{CS}$ in Step~\ref{step:I} of our schema to construct an uncertainty set.  Bounding Value at Risk for regions like $\mathcal{P}^{CS}$ was studied by \cite{calafiore2006distributionally}.  Their results imply
\begin{equation}
\label{eq:suppCS}
\sup_{\P \in \mathcal{P}^{CS}} \quant{\epsilon}{\P}(\bv) = \hat{\bmu}^T\bv + \Gamma_1 \| \bv \|_2 + \sqrt{\frac{1-\epsilon}{\epsilon}}
\sqrt{ \bv^T ( \hat{\bSigma} + \Gamma_2 \bI) \bv }.
\end{equation}
We translate this bound into an uncertainty set.
\begin{theorem}
\label{thm:CSSet}
With probability at least $1-\alpha$ with respect to the sampling, the family $\{ \U^{CS}_\epsilon : 0 < \epsilon < 1\}$ simultaneously implies a probabilistic guarantee for $\P^*$, where 
\begin{equation} \label{def:UCS}
\U_\epsilon^{CS} = \left\{ \hat{\bmu} + \by + \bC^T\bw : \exists \by, \bw \in \R^d
\text{ s.t. }   
 \| \by \| \leq \Gamma_1^{B}, \ \ \| \bw \| \leq \sqrt{\frac{1}{\epsilon} -1} \right\},
\end{equation}  
where $\bC^T\bC = \hat{\bSigma} + \Gamma_2^B \bI$ is a cholesky decomposition.
Moreover, $\delta^*(\bv | \ \U^{CS}_\epsilon)$ is given explicitly by the right-hand side of Eq.~\eqref{eq:suppCS} with $(\Gamma_1, \Gamma_2)$ replaced by the bootstrapped thresholds $\Gamma_1^B, \Gamma_2^B$.  
\end{theorem}
{\blockedit
\begin{remark} \label{rem:CSequiv}
Notice that \eqref{eq:suppCS} is written with an \emph{equality}.  The robust constraint $\max_{\bu \in \U^{CS}_\epsilon} \bv^T\bx \leq 0$ is exactly equivalent to the ambiguous chance-constraint $\supp_{\P \in \mathcal{P}^{CS}} \quant{\epsilon}{\P}(\bv) \leq 0$ where $\mathcal{P}^{CS}$ is defined with the smaller (bootstrapped) thresholds.  
\end{remark}

\begin{remark} From \eqref{eq:suppCS}, $\{ (\bv, t) : \delta^*(\bv | \ \U^{CS}_\epsilon) \leq t \}$ is second order cone representable.  Moreover, we can separate over this constraint in closed-form.  Given $\bv, t$ such that $\delta^*(\bv | \ \U^{CS}_\epsilon) > t$, let $\bu = \bmu + \frac{\Gamma_1^B}{ \| \bv \|} \bv  + \sqrt{\frac{1}{\epsilon}-1} \frac{\bC \bv}{\| \bC \bv \|} $.  Then $\bu \in \U^{CS}_\epsilon$ and $\bu^T\bv \leq t$ is a violated inequality (cf. Proof of Theorem~\ref{thm:CSSet}.)
\end{remark}
\begin{remark}  \label{rem:CSSupp} Like $\U^{FB}_\epsilon$, there is no guarantee that $\U^{CS}_\epsilon \subseteq \supp(\P^*)$.  Consequently, when a priori knowledge of the support is available, we can refine this set as in Remark~\ref{rem:Support}.
\end{remark}
}
To emphasize the benefits of bootstrapping when constructing uncertainty sets, Fig.~\ref{fig:UCSBoot} in the electronic companion illustrates the set $\U^{CS}_\epsilon$ for the example considered in Fig.~\ref{fig:UI} with thresholds computed with and without the bootstrap.

\subsection{Uncertainty Set Motivated by Delage and Ye, 2010}
\label{sec:DY}
\citet{delage2010distributionally} propose a data-driven approach for solving distributionally robust optimization problems.  Their method relies on a slightly more general version of the following:\footnote{Specifically, since $R$ is typically unknown, the authors describe an estimation procedure for $R$ and prove a modified version of the Theorem~\ref{thm:DYFamily} using this estimate and different constants.  We treat the simpler case where $R$ is known here.  Extensions to the other case are straightforward.}
\begin{theorem}[Delage and Ye, 2010]
\label{thm:DYFamily}
Let $R$ be such that $\P^*( (\burv - \bmu)^T \bSigma^{-1} (\burv-\bmu)  \leq R^2 ) = 1$ where $\bmu, \bSigma$ are the true mean and covariance of $\burv$ under $\P^*$.  Let, 
$
\gamma_1 \equiv \frac{ \beta_2}{1 - \beta_1 - \beta_2}$, \  $\gamma_2 \equiv \frac{1+ \beta_2}{1 - \beta_1 - \beta_2}$, \  $\beta_2 \equiv \frac{R^2}{N}\left( 2 + \sqrt{2 \log (2/\alpha) } \right)^2$,
$\beta_1 \equiv \frac{R^2}{\sqrt{N}} \left( \sqrt{1- \frac{d}{R^4}} + \sqrt{ \log(4/\alpha)} \right)$, and suppose
also that $N$ is large enough so that $1-\beta_1 -\beta_2 > 0$.  Finally suppose $\supp(\P^*) \subseteq  [\buhat^{(0)}, \buhat^{(N+1)}] $.  Then with probability at least $1-\alpha$ with respect to the sampling, $\P^* \in \mathcal{P}^{DY}$ where
\begin{align*}
\mathcal{P}^{DY} \equiv \left\{ \P \in \Theta [\buhat^{(0)}, \buhat^{(N+1)}]  : (\E^\P[ \burv] - \hat{\bmu} )^T \hat{\bSigma}^{-1} (\E^\P[ \burv] - \hat{\bmu} ) \leq \gamma_1, \ \ 
\E^\P[ (\burv - \hat{\bmu}) (\burv - \hat{\bmu})^T ] \preceq \gamma_2 \hat{\bSigma}
\right \}.
\end{align*}
\end{theorem}
The key idea is again to compute the thresholds using a general purpose concentration inequality.  The condition on $N$ is required for the confidence region to be well-defined.

We again observe that $\mathcal{P}^{DY}$ is the $1-\alpha$ confidence region of a hypothesis test.  Specifically, it considers the hypothesis \eqref{eq:MeanAndCovar} using the statistics $(\hat{\bmu} - \bmu_0)^T \hat{\bSigma}^{-1}(\hat{\bmu} - \bmu_0)$ and $\max_{\blambda} \frac{\blambda^T(\bSigma_0 + (\bmu_0 - \hat{\bmu})(\bmu_0 - \hat{\bmu})^T)\blambda}{\blambda^T \hat{\bSigma} \blambda}$ with thresholds $\gamma_1, \gamma_2$.  

Since the thresholds are, again, potentially overly conservative, we approximate new thresholds using the bootstrap.  Table~\ref{tab:BootstrapThreshCS} shows the reduction in magnitude.  Observe that the bootstrap thresholds exist for all $N$, not just $N$ sufficiently large.  Moreover, they are significantly smaller.  This reduction translates to a reduction in the potential over conservatism of any method using $\mathcal P^{DY}$, including those presented within \cite{delage2010distributionally} while retaining the same probabilistic guarantee.  

We next consider using $\mathcal{P}^{DY}$ in Step~\ref{step:I} of our schema to generate an uncertainty set $\U$ that ``corresponds" to this method.
\begin{theorem}
\label{thm:DY} 
Suppose $\supp(\P^*) \subset [\buhat^{(0)}, \buhat^{(N+1)}]$.  Then, with probability at least $1-\alpha$ with respect to the sampling, the family $\{ \U^{DY}_\epsilon : 0 < \epsilon <  1 \}$ simultaneously implies a probabilistic guarantee for $\P^*$, where 
{\small
\begin{align}  \notag
\U^{DY}_\epsilon = \Big\{ \bu \in [\buhat^{(0)}, \buhat^{(N+1)}] : &\exists \lambda \in \R, \ \bw, \bm \in \R^d, \ \bA, \hat{\bA} \succeq \bzero  \text{ s.t. }
\\ \notag
& \lambda \leq \frac{1}{\epsilon}, 
\ \ \ 
(\lambda -1) \buhat^{(0)} \leq \bm \leq (\lambda -1) \buhat^{(N+1)}, 
\\ \label{def:UDY}
&
\begin{pmatrix}
	\lambda -1 & \bm^T
	\\
	\bm & \bA
\end{pmatrix}
\succeq \bzero, 
\ \ \
\begin{pmatrix}
	1 & \bu^T
	\\
	\bu & \hat{\bA}
\end{pmatrix}
\succeq \bzero, 
\\ \notag
&\lambda \hat{\bmu} = \bm + \bu + \bw, \ \ \ \| \bC \bw \| \leq \lambda \sqrt{\gamma_1^B},
\\ \notag
&\lambda (\gamma_2^B \hat{\bSigma} + \hat{\bmu} \hat{\bmu}^T) - \bA - \hat{\bA} - \bw \hat{\bmu}^T - \hat{\bmu} \bw^T \succeq \bzero
 \Big\}, 
\end{align} 
}
$C^T C = \hat{\bSigma}^{-1}$ is a Cholesky-decomposition, and $\gamma_1^B, \gamma_2^B$ are computed by bootstrap.  
Moreover, 
\begin{align*}
\delta^*(\bv | \ \U^{DY}_\epsilon ) = \sup_{\P \in \mathcal{P}^{DY}} \quant{\epsilon}{\P}(\bv) = 
\inf \quad & t
\\ \nonumber
\text{s.t.} \quad &r + s \leq \theta \epsilon,
\\ \nonumber
& 
\begin{pmatrix}
	r  + \by_1^{+T} \buhat^{(0)} - \by_1^{-T} \buhat^{(N+1)} & \frac{1}{2} (\bq - \by_1)^T, 
\\ \notag
	\frac{1}{2} (\bq - \by_1) & \bZ 
\end{pmatrix} \succeq \bzero,
\\ 
& \begin{pmatrix}
	r  + \by_2^{+T} \buhat^{(0)} - \by_2^{-T} \buhat^{(N+1)} + t - \theta & \frac{1}{2} (\bq - \by_2 - \bv)^T, 
\\ \notag
	\frac{1}{2} (\bq - \by_2 - \bv) & \bZ 
\end{pmatrix} \succeq \bzero,
\\ \nonumber
& s \geq (\gamma^B_2 \hat{\bSigma} + \hat{\bmu} \hat{\bmu}^T) \circ \bZ + \hat{\bmu}^T\bq + \sqrt{\gamma^B_1} \| \bq + 2 \bZ \hat{\bmu} \|_{\hat{\bSigma}^{-1}},
\\ \nonumber
&\by_1 = \by_1^+ - \by_1^-,  \ \ \by_2 = \by_2^+ - \by_2^-,
\ \ 
 \by_1^+, \by_1^-, \by_2^+,\by_2^-, \theta \geq \bzero.
\end{align*}
\end{theorem}
\begin{remark} \label{rem:DYequiv}
Similar to $\U^{CS}_\epsilon$, the robust constraint $\max_{\bu\in \U^{DY}_\epsilon} \bv^T\bu \leq 0$ is equivalent to the ambiguous chance constraint $\sup_{\P \in \mathcal{P}^{DY}} \quant{\epsilon}{\P}(\bv) \leq 0$.  
\end{remark}
\begin{remark}
The set $\{(\bv, t) : \delta^*(\bv | \ \U^{DY} ) \leq t \}$ is representable as a linear matrix inequality.  At time of writing, solvers for linear matrix inequalities are not as developed as those for second order cone programs.  Consequently, one may prefer $\U^{CS}_\epsilon$ to $\U^{DY}_\epsilon$ in practice for its simplicity.  
\end{remark}

{\blockedit
\subsection{Comparing $\U^M_\epsilon$, $\U^{LCX}_\epsilon$, $\U^{CS}_\epsilon$ and $\U^{DY}_\epsilon$}
One of the benefits of deriving uncertainty sets corresponding to the methods of \cite{shawe2003estimating} and \cite{delage2010distributionally} is that it facilitates comparisons between these methods and our own proposals.  In Fig.~\ref{fig:Comparisons}, we illustrate the sets $\U^M_\epsilon$, $\U^{LCX}_\epsilon$, $\U^{CS}_\epsilon$ and $\U^{DY}_\epsilon$ for the same numerical example from Fig.~\ref{fig:UI}.  
Because $\U^M$ does not leverage the joint distribution $\P^*$, it does not learn that its marginals are independent. Consequently, $\U^M$ has pointed corners permitting extreme values of both coordinates simultaneously.  The remaining sets do learn the marginal independence from the data and, hence, have rounded corners.

The set $\U^{CS}_\epsilon$ is not contained in $\supp(\P^*)$.  Interestingly, the intersection $\U^{CS}_\epsilon \cap \supp(\P^*)$ is very similar to $\U^{DY}_\epsilon$ for this example (indistinguishable in picture).  Since $\U^{CS}$ and $\U^{DY}$ only depend on the first two moments of $\P^*$, neither is able to capture the skewness in the second coordinate.  
Finally, $\U^{LCX}$ is contained within $\supp(\P^*)$ and displays symmetry in the first coordinate and skewness in the second.   In this example it is also the smallest set (in terms of volume).
All sets shrink as $N$ increases.  

\begin{figure}
\centering
\begin{subfigure}{0.49\textwidth}
	\includegraphics[width=\textwidth]{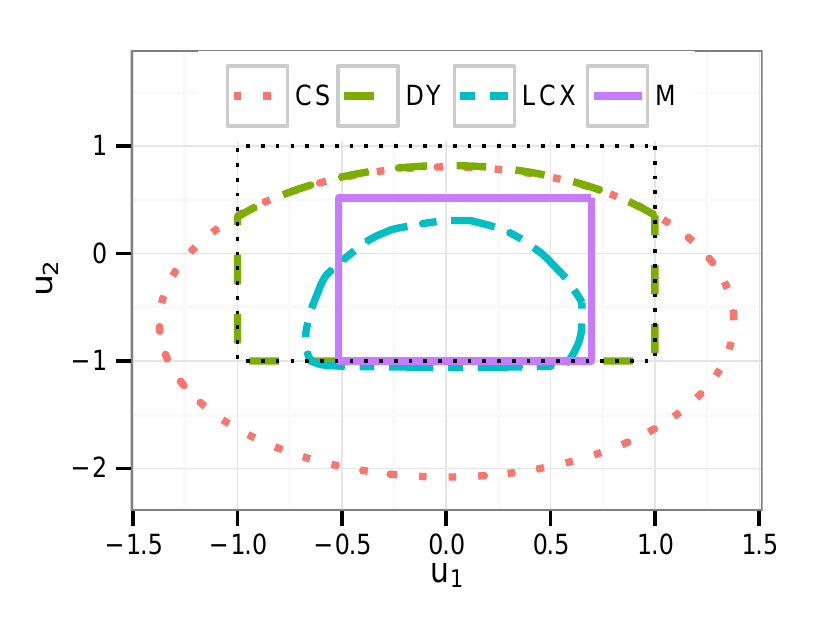}
\end{subfigure}
\begin{subfigure}{0.49\textwidth}
	\includegraphics[width=\textwidth]{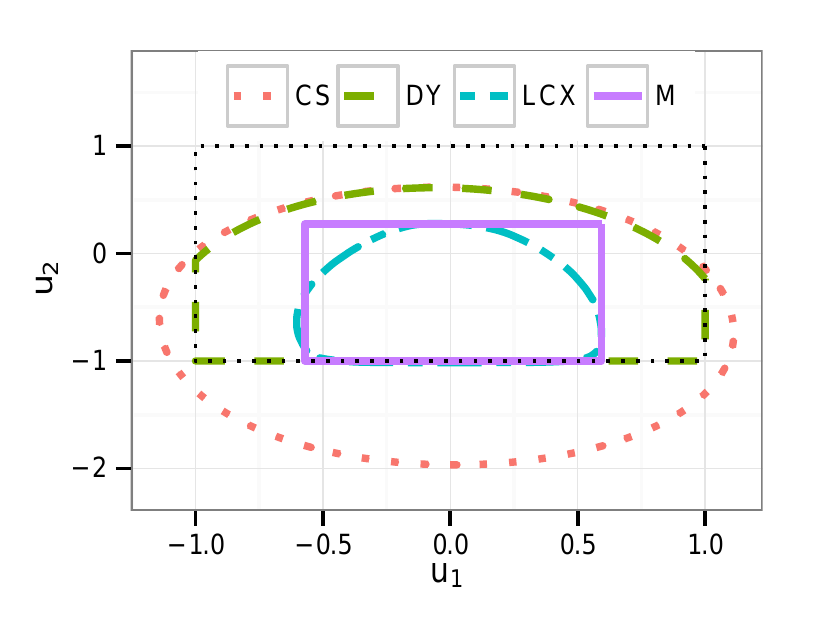}
\end{subfigure}
\caption{\label{fig:Comparisons} Comparing $\U^M_\epsilon$, $\U^{LCX}_\epsilon$, $\U^{CS}_\epsilon$ and $\U^{DY}_\epsilon$ for the example from Fig.~\ref{fig:UI}, $\epsilon = 10\%$, $\alpha = 20\%$.  The left panel uses $N=100$ data points, while the right panel uses $N=1,000$ data points.}
\end{figure}
}


\subsection{Refining $\U^{FB}_\epsilon$}
Another common approach to hypothesis testing in applied statistics is to use tests designed for Gaussian data that are ``robust to departures from normality."  The best known example of this approach is the $t$-test from Sec.~\ref{sec:BackgroundHypTests}, for which there is a great deal of experimental evidence to suggest that the test is still approximately valid when the underlying data is non-Gaussian \citep[][Chapt. 11.3]{romano2005testing}.  
Moreover, certain nonparametric tests of the mean for non-Gaussian data are asymptotically equivalent to the $t$-test, so that the $t$-test, itself, is asymptotically valid for non-Gaussian data \citep[][p. 180]{romano2005testing}.  Consequently, the $t$-test is routinely used in practice, even when the Gaussian assumption may be invalid.  

We next use the $t$-test in combination with bootstrapping to refine $\U^{FB}_\epsilon$.  We replace $m_{fi}, m_{bi}$ in Eq.~\eqref{def:UFB}, with the upper and lower thresholds of a $t$-test at level $\alpha^\prime/2$.  We expect these new thresholds to correctly bound the true mean $\mu_i$ with probability approximately $1-\alpha^\prime/2$ with respect to the data.  We then use the bootstrap to calculate bounds on the forward and backward deviations $\overline{\sigma}_{fi}, \overline{\sigma}_{bi}$.  

We stress not all tests designed for Gaussian data are robust to departures from normality.  Applying Gaussian tests that lack this robustness will likely yield poor performance.  Consequently, some care must be taken when choosing an appropriate test.  

{\blockedit
\section{Optimizing over Multiple Constraints}
\label{sec:Multiple}
In this section, we propose an approach for solving \eqref{eq:optMultRob}.  The key observation is
\begin{theorem} \label{thm:biconvex} \hfill
\begin{enumerate}[label=\alph*)]
\item The constraint $\delta^*(\bv |\  \U^{CS}_\epsilon) \leq t$ is bi-convex in $(\bv, t)$ and $\epsilon$, for $0 < \epsilon < .75$.  
\item The constraint $\delta^*(\bv |\  \U^{FB}_\epsilon) \leq t$ is bi-convex in $(\bv, t)$ and $\epsilon$, for $0 < \epsilon < 1/\sqrt{e}$.
\item The constraint $\delta^*(\bv | \ \U_\epsilon) \leq t$ is bi-convex in $(\bv, t)$ and $\epsilon$, for $0 < \epsilon < 1$, and $\U_\epsilon \in \{ \U^{\chi^2}_\epsilon, \U^{G}_\epsilon, \U^I_\epsilon, \U^{LCX}_\epsilon, \U^{DY}_\epsilon\}$.
\end{enumerate}
\end{theorem}  

This observations suggests a heuristic: Fix the values of $\epsilon_j$, and solve the robust optimization problem in the original decision variables.  Then fix this solution and optimize over the $\epsilon_j$.  Repeat until some stopping criteria is met or no further improvement occurs.  \citet{chen2010cvar} suggested a similar heuristic for multiple chance-constraints in a different context.  In Appendix~\ref{sec:EpsilonOpt} we propose a refinement of this approach that solves a linear optimization problem to obtain the next iterates for $\epsilon_j$, incorporating dual information from the overall optimization and other constraints.  Our proposal ensures the optimization value is non-increasing between iterations and that the procedure is finitely convergent.
}

{\blockedit 
\section{Choosing the ``Right" Set and Tuning $\alpha$, $\epsilon$}
\label{sec:Choose}

Often several of our data-driven sets may be consistent with the a priori knowledge of $\P^*$.  Choosing an appropriate set from amongst our proposals is a non-trivial task that depends on the application and the data.  One may be tempted to use the intersection of all eligible sets.  We caution that the intersection of two sets which imply a probabilistic guarantee at level $\epsilon$ need not imply a probabilistic guarantee at level $\epsilon$. 
Similarly, one may be tempted to solve the robust optimization model for each eligible set separately and select the set and solution with best objective value.  We caution that a set chosen in this way will suffer from an in-sample bias.  Specifically, the probability with respect to the sampling that this set does not imply a probabilistic guarantee at level $\epsilon$ may be much larger than $\alpha$.  

Drawing an analogy to model selection in machine learning, we propose a different approach to set selection.  Specifically, split the data into two parts, a training set and a hold-out set.  Use the training set to construct each potential uncertainty set, in turn, and solve the robust optimization problem.  Test each of the corresponding solutions out-of-sample on the hold-out set, and select the best solution and corresponding uncertainty set.  Since the two halves of the data are independent, it follows that with probability at least $1-\alpha$ with respect to the sampling, the set so selected will correctly imply a probabilistic guarantee at level $\epsilon$.

The drawback of this approach is that only half the data is used to calibrate the uncertainty set.  When $N$ is only moderately large, this may be impractical.  In these cases, $k$-fold cross-validation can be used to select a set.  (See \citet{hastie2009elements} for a review of cross-validation.)  Unlike the above procedure, we cannot prove that the set chosen by $k$-fold cross-validation satisfies the appropriate guarantee.  Nevertheless, experience in model selection suggests that this procedure frequently identifies a good model, and, thus, we expect it will identify a good set.  We use $5$-fold cross-validation in our numerical experiments.

In applications where there is not a natural choice for $\alpha$ or $\epsilon$, we suggest tuning these parameters in an entirely analogous way.  Namely, we propose selecting a grid of potential values for $\alpha$ and/or $\epsilon$ and then selecting the best value either using a hold-out set or cross-validation. Since the optimal value likely depends on the choice of uncertainty set, we suggest choosing them jointly.
}

\section{Applications}
\label{sec:computational}
We demonstrate how our new sets may be used in two applications: portfolio management and queueing theory.  Our goals are to, first, illustrate their application and, second, to compare them to one another.  We summarize our major insights:  
{\blockedit 
\begin{itemize}
	\item In these two applications, our data-driven sets outperform traditional, non-data driven uncertainty sets, and, moreover, robust models built with our sets perform as well or better than other data-driven approaches.  
	\item Although our data-driven sets all shrink as $N\rightarrow \infty$, they learn different features of $\P^*$, such as correlation structure and skewness.  Consequently, different sets may be better suited to different applications, and the right choice of set may depend on $N$.  Cross-validation and other model selection techniques effectively identify the best set.
	\item Optimizing the $\epsilon_j$'s in the case of multiple constraints can significantly improve performance.  
\end{itemize}
}
\subsection{Portfolio Management}
\label{sec:portfolio}
Portfolio management has been well-studied in the robust optimization literature \citep[e.g., ][]{goldfarb2003robust, natarajan2008incorporating, calafiore2012data}.  For simplicity, we will consider the one period allocation problem:
\begin{align} \label{eq:AssetAlloc}
\max_{\bx} \left\{  \min_{\mathbf{r} \in \U} \ \ \mathbf{r}^T \bx : \ \
\be^T \bx = 1, \ \ \bx \geq \bzero \right\}, 
\end{align}
which seeks the portfolio $\bx$ with maximal worst-case return over the set $\U$.  If $\U$ implies a probabilistic guarantee for $\P^*$ at level $\epsilon$, then the optimal value $z^*$ of this optimization is a conservative bound on the $\epsilon$-worst case return for the optimal solution $\bx^*$.  

{\blockedit We consider a synthetic market with $d = 10$ assets.  Returns are generated according to the following model from \cite{natarajan2008incorporating}:
\begin{equation} \label{eq:FactorMarket}
 \tilde{r}_i = \begin{cases} \frac{\sqrt{(1-\beta_i)\beta_i}}{\beta_i} & \text{with probability } \beta_i
 						\\ -\frac{\sqrt{(1-\beta_i)\beta_i}}{1-\beta_i} & \text{with probability } 1-\beta_i
			\end{cases}, 
\quad 
\beta_i = \frac{1}{2}\left(1 + \frac{i}{11}\right), \ \ i = 1, \ldots, 10.
\end{equation}
In this model, all assets have the same mean return (0\%), the same standard deviation ($1.00\%$), but have different skew and support.  Higher indexed assets are highly skewed; they have a small probability of achieving a very negative return.}  Returns for different assets are independent.  We simulate $N=500$ returns to use as data.


We will utilize our sets $\U^M_\epsilon$ and $\U^{LCX}_\epsilon$ in this application.  We do not consider the sets $\U^I_\epsilon$ or $\U^{FB}_\epsilon$ since we do not know a priori that the returns are independent.  To contrast to the methods of \citep{shawe2003estimating} and \citep{delage2010distributionally} we also construct the sets $\U^{CS}_\epsilon$ and  $\U^{DY}_\epsilon$.  \edit{Recall from Remarks~\ref{rem:CSequiv} and \ref{rem:DYequiv} that robust linear constraints over these sets are equivalent to ambiguous chance-constraints in the original methods, but with improved thresholds.  As discussed in Remark~\ref{rem:CSSupp}, we also construct $\U^{CS}_\epsilon \cap \supp(\P^*)$ for comparison.  We use $\alpha = \epsilon = 10\%$ in all of our sets.
Finally, we will also compare to the method of \cite{calafiore2012data} (denoted ``CM" in our plots), which is not an uncertainty set based method.  We calibrate this method to also provide a bound on the $10\%$ worst-case return that holds with at least $90\%$ with respect to the sampling so as to provide a fair comparison.}

{\blockedit We first consider the problem of selecting an appropriate set via $5$-fold cross-validation.  The top left panel in Fig.~\ref{fig:PortExp} shows the out-of-sample 10\% worst-case return for each of the $5$ runs (blue dots), as well as the average performance on the $5$ runs for each set (black square).  Sets $\U^M_\epsilon$, $\U^{CS}_\epsilon \cap \supp(\P^*)$ and $\U^{DY}_\epsilon$ yield identical portfolios (investing everything in the first asset) so we only include $\U^M$ in our graphs.  The average performance is also shown in Table~\ref{tab:PortSummary} under column CV (for ``cross-validation.")  The optimal objective value of \eqref{eq:AssetAlloc} for each of our sets (trained with the entire data set) is shown in column $z_{In}$.

\begin{figure}
\centering
\begin{subfigure}{.49 \textwidth}
	\includegraphics[width=\textwidth]{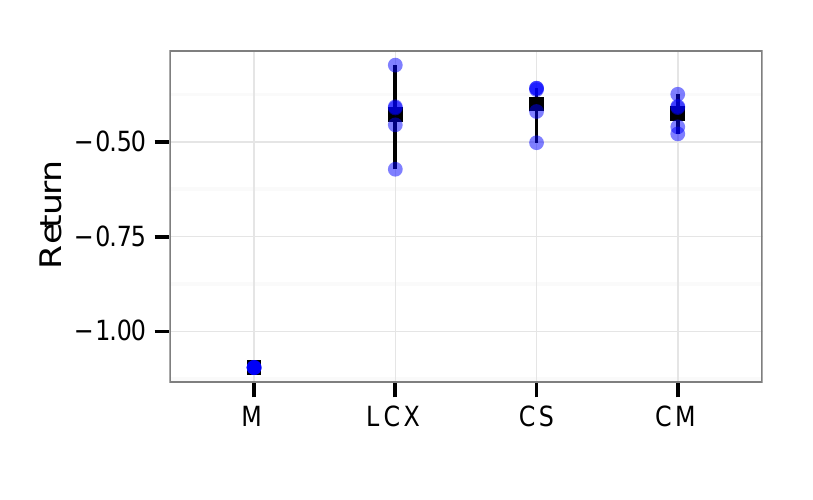}
\end{subfigure}	
\begin{subfigure}{.49 \textwidth}
	\includegraphics[width=\textwidth]{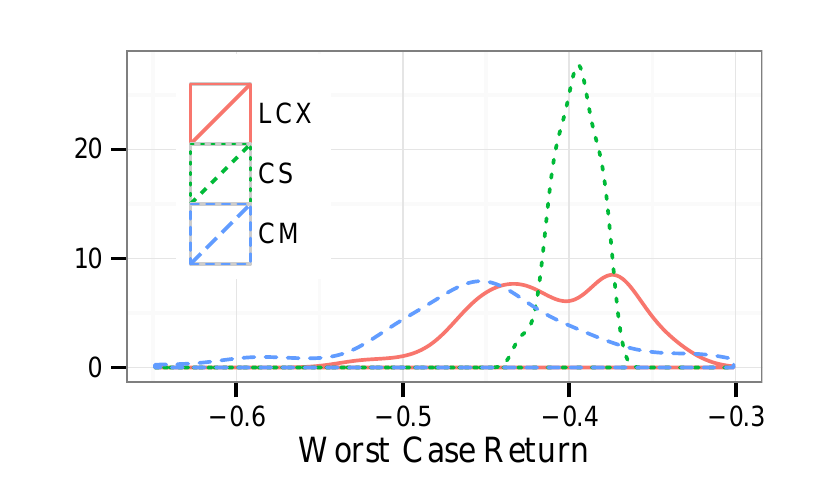}
\end{subfigure}	
\\
\begin{subfigure}{.49 \textwidth}
	\includegraphics[width=\textwidth]{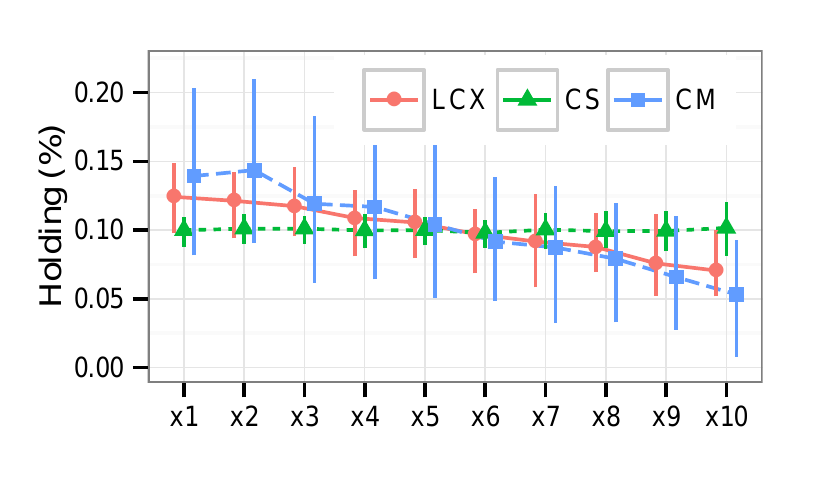}
\end{subfigure}	
\begin{subfigure}{.49 \textwidth}
	\includegraphics[width=\textwidth]{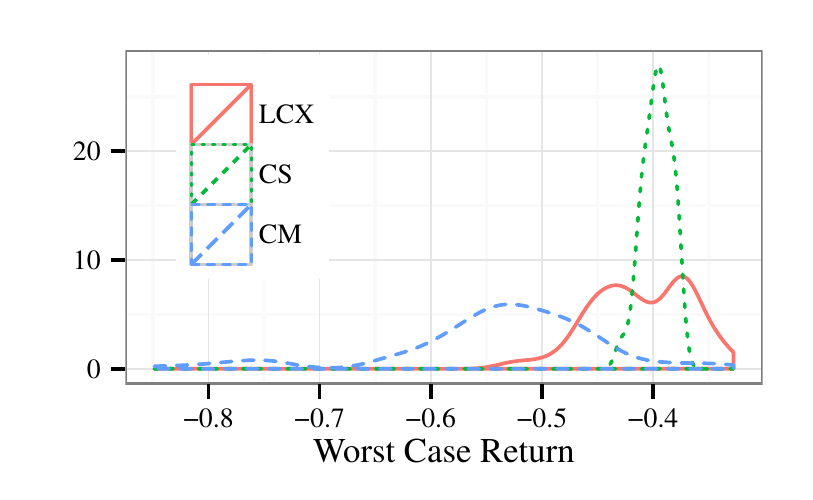}
\end{subfigure}	

\caption{\label{fig:PortExp} Portolio performance by method: $\alpha = \epsilon = 10\%$.  Top left: Cross-validation results.  Top right: Out-of-sample distribution of the 10\% worst-case return over 100 runs.  Bottom left: Average portfolio holdings by method.  Bottom right:  Out-of-sample distribution of the 10\% worst-case return over 100 runs.   The bottom right panel uses $N=2000$.  The remainder use $N=500$.}
\end{figure}

\begin{table}
\centering
\caption{\label{tab:PortSummary} Portfolio statistics for each of our methods.  $\U^{DY}_\epsilon$ and $\U^{CS}_\epsilon \cap \supp(\P^*)$ perform identically to $\U^{M}_\epsilon$.  ``CM" refers to the method of \cite{calafiore2012data}.}
\begin{tabular}{@{}rcccccccc@{}}
\toprule
\multicolumn{1}{l}{} & \multicolumn{4}{c}{$N=500$}               & \multicolumn{4}{c}{$N=2000$}              \\ \cmidrule(l){2-5} \cmidrule(l){6-9} 
              & $z_{In}$ & CV     & $z_{Out}$ & $z_{Avg}$ & $z_{In}$ & CV     & $z_{Out}$ & $z_{Avg}$ \\ \midrule
M                    & -1.095   & -1.095 & -1.095    & -1.095    & -1.095   & -1.095 & -1.095    & -1.095    \\
LCX                  & -0.699   & -0.373 & -0.373    & -0.411    & -0.89    & -0.428 & -0.395    & -0.411    \\
CS                   & -1.125   & -0.403 & -0.416    & -0.397    & -1.306   & -0.400   & -0.417    & -0.396    \\
CM                   & -0.653   & -0.495 & -0.425    & -0.539    & -0.739   & -0.426 & -0.549    & -0.451    \\ \bottomrule
\end{tabular}
\end{table}

Based on the top left panel of Fig.~\ref{fig:PortExp}, it is clear that $\U^{LCX}_\epsilon$ and $\U^{CS}_\epsilon$ significantly outperform the remaining sets.  They seem to perform similarly to the CM method. Consequently, we would choose one of these two sets in practice.  

We can assess the quality of this choice by using the ground-truth model \eqref{eq:FactorMarket} to calculate the true 10\% worst-case return for each of the portfolios.  These are shown in Table~\ref{tab:PortSummary} under column $z_{Out}$.  Indeed, these sets perform better than the alternatives, and, as expected, the cross-validation estimates are reasonably close to the true out-of-sample performance.  By contrast, the in-sample objective value $z_{In}$ is a loose bound.  We caution against using this in-sample value to select the best set.  

Interestingly, we point out that while $\U^{CS}_\epsilon \cap \supp(\P^*)$ is potentially smaller (with respect to subset containment) than $\U^{CS}_\epsilon$, it performs much worse out-of-sample (it performs identically to $\U^M_\epsilon$).  This experiment highlights the fact that size calculations alone cannot predict performance.  Cross-validation or similar techniques are required.

One might ask if these results are specific to the particular draw of $500$ data points we use.  We repeat the above procedure $100$ times.   The resulting distribution of 10\% worst-case return is shown in the top right panel of Fig.~\ref{fig:PortExp} and the average of these runs is shown Table~\ref{tab:PortSummary} under column $z_{Avg}$.  As might have been guessed from the cross-validation results, $\U^{CS}_\epsilon$ delivers more stable and better performance than either $\U^{LCX}_\epsilon$ or CM.  $\U^{LCX}_\epsilon$ slightly outperforms CM, and its distribution is shifted right.  

We next look at the distribution of actual holdings between these methods.  We show the average holding across these $100$ runs as well as $10\%$ and $90\%$ quantiles for each asset in the bottom left panel of Fig.~\ref{fig:PortExp}.  
Since $\U^M_\epsilon$ does not use the joint distribution, it sees no benefit to diversification.
Portfolios built from $\U^M_\epsilon$ consistently holds all their wealth in the first asset over all the runs, hence, omitted from graphs.   The set $\U^{CS}_\epsilon$ depends only on the first two moments of the data, and, consequently, cannot distinguish between the assets.  It holds a very stable portfolio of approximately the same amount in each asset.  By contrast, $\U^{LCX}$ is able to learn the asymmetry in the distributions, and holds slightly less of the higher indexed (toxic) assets.  CM is similar to $\U^{LCX}$, but demonstrates more variability in the holdings.  

We point out that the performance of each method depends slightly on $N$.  We repeat the above experiments with $N=2000$.  Results are summarized in Table~\ref{tab:PortSummary}.  The bottom right panel of Fig.~\ref{fig:PortExp} shows the distribution of the $10\%$ worst-case return.  (Additional plots are also available in Appendix~\ref{sec:PortAppendix}.)  Both $\U^{LCX}$ and CM perform noticeably better with the extra data, but $\U^{LCX}$ now noticeably outperforms CM and its distribution is shifted significantly to the right.  
}

\subsection{Queueing Analysis}
\label{sec:queue}
{\blockedit 
One of the strengths of our approach is the ability to retrofit existing robust optimization models by replacing their uncertainty sets with our proposed sets, thereby creating new data-driven models that satisfy strong guarantees. In this section, we illustrate this idea with a robust queueing model as in \cite{bertsimas2011performance} and \cite{Bandi}.  \cite{Bandi} use robust optimization to generate \emph{approximations} to a performance metric of a queuing network.  We will combine their method with our new sets to generate \emph{probabilistic upper bounds} to these metrics.}  For concreteness, we focus on the waiting time in a G/G/1 queue.  Extending our analysis to more complex queueing networks can likely be accomplished similarly.  \edit{We stress that we do not claim that our new bounds are the best possible -- indeed there exist extremely accurate, specialized techniques for the G/G/1 queue -- but, rather, that the retrofitting procedure is general purpose and yields reasonably good results.  These features suggest that a host of other robust optimization applications in information theory \citep{bandi2012tractable}, supply-chain management \citep{ben2005retailer} and revenue management \citep{rusmevichientong2012robust} might benefit from this retrofitting.}

Let $\burv_i = (\tilde{x}_i, \tilde{t}_i)$ for $i = 1, \ldots, n$ denote the uncertain service times and interarrival times of the first $n$ customers in a queue. We assume that $\burv_i$ is i.i.d. for all $i$ and has independent components, and that there exists $\buhat^{(N+1)} \equiv (\overline{x}, \overline{t})$ such that $0 \leq \tilde{x}_i \leq \overline{x}$ and $0 \leq \tilde{t}_i \leq \overline{t}$ almost surely.  
%
%

From Lindley's recursion \citep{lindley1952theory}, the waiting time of the $n^\text{th}$ customer is
\begin{equation} \label{eq:SingleServerQueue}
\tilde{W}_n = \max_{1 \leq j \leq n} \left(\max \left(\sum_{l=j}^{n-1} \tilde{x}_l - \sum_{l=j+1}^n \tilde{t}_l, 0\right)\right) = 
\max\left(0, \max_{1 \leq j \leq n} \left(\sum_{l=j}^{n-1} \tilde{x}_l - \sum_{l=j+1}^n \tilde{t}_l\right)\right).
\end{equation}
Motivated by \citet{Bandi}, we consider a worst-case realization of a Lindley recursion
\begin{equation} \label{eq:SingleServerQueueRob}
\max \left(0, \max_{1 \leq j \leq n} \max_{(\mathbf{x}, \mathbf{t} ) \in \U} \left(\sum_{l=j}^{n-1} \tilde{x}_l - \sum_{l=j+1}^{n} \tilde{t}_l  \right) \right).
\end{equation}
Taking $\U = \U^{FB}_{\overline{\epsilon}/n}$ and applying Theorem~\ref{thm:UFB} to the inner-most optimization  yields
\begin{align} \label{eq:FB1}
\max_{1 \leq j \leq n} (m_{f1} - m_{b2}) (n-j) + \sqrt{ 2 \log(n/\overline{\epsilon}) ( \sigma_{f1}^2 + \sigma_{b2}^2 )} \sqrt{n-j}
  \end{align}
Relaxing the integrality on $j$, this optimization can be solved closed-form yielding
\begin{align} \label{eq:QueueMedian}
W_n^{1, FB} &\equiv \begin{cases}
				(m_{f1} - m_{b2}) n + \sqrt{ 2 \log(\frac{n}{\overline{\epsilon}}) ( \sigma_{f1}^2 + \sigma_{b2}^2 )} \sqrt{n} & \text{ if } n < \frac{\log(\frac{n}{\overline{\epsilon}}) ( \sigma_{f1}^2 + \sigma_{b2}^2 )}{2  (m_{b2} - m_{f1})^2}  \text{ or } m_{f1} > m_{b2},
\\
				\frac{\log(\frac{n}{\overline{\epsilon}}) (\sigma_{f1}^2 + \sigma_{b2}^2)}{ 2( m_{b2} - m_{f1})}
				& \text{ otherwise. }
			\end{cases}
\end{align}
From \eqref{eq:SingleServerQueueRob}, with probability at least $1-\alpha$ with respect to the sampling, each of the inner-most optimizations upper bound their corresponding random quantity with probability $1-\overline{\epsilon}/n$ with respect to $\P^*$.  Thus, by union bound, $\P^*(\tilde{W}_n \leq W_n^{1, FB} ) \geq 1-\overline{\epsilon}$.

{\blockedit 
On the other hand, since $\{ \U^{FB}_\epsilon : 0 < \epsilon < 1 \}$ simultaneously implies a probabilistic guarantee, we can also optimize the choice of $\epsilon_j$ in \eqref{eq:FB1}, yielding
\begin{align} \notag
W^{2, FB}_n &\equiv \min_{w, \bepsilon} \quad w
\\ \label{eq:tightconstraint}
\text{s.t.} \quad & w \geq (m_{f1} - m_{b2}) (n-j) + \sqrt{ 2 \log(1/\epsilon_j) ( \sigma_{f1}^2 + \sigma_{b2}^2 )} \sqrt{n-j}, \ \ j=1, \ldots, n-1, 
\\ \notag
&w \geq 0, \ \ \bepsilon \geq \bzero, \ \ \sum_{j=1}^{n-1} \epsilon_j \leq \overline{\epsilon}.
\end{align}
From the KKT conditions, the constraint \eqref{eq:tightconstraint} will be tight for all $j$, so that $W^{2, FB}_n$ satisfies
\begin{equation} \label{eq:SumExpr}
\sum_{j=1}^{n-1} \exp \left( - \frac{(W_n^{2, FB} - (m_{f1} - m_{b2}) )^2 }{2 (n-j) ( \sigma_{f1}^2 + \sigma_{b2}^2 )^2} \right) = \overline{\epsilon},
\end{equation}
which can be solved by line search.  Again, with probability $1-\alpha$ with respect to the sampling, $\P^*(\tilde{W}_n \leq W_n^{2, FB}) \geq 1-\overline{\epsilon}$, and $W_n^{2, FB}  \leq W_n^{1, FB}$ by construction.

We can further refine our bound by truncating the recursion \eqref{eq:SingleServerQueue} at customer $\min (n, n^{(k)} )$ where, with high probability, $\tilde{n} \leq n^{(k)}$.  A formal derivation of the resulting bound, which we denote $W_n^{3, FB}$, can be found in Appendix~\ref{sec:QueueAppendix}.  Therein we also prove that with probability at least $1-\alpha$ with respect to the sampling, $\P^*(\tilde{W}_n \leq W^{3, FB}_n) \geq 1-\overline{\epsilon}$.  

Finally, our choice of $\U^{FB}_\epsilon$ was somewhat arbitrary.  Similar analysis can be performed for many of our sets.  To illustrate, Appendix~\ref{sec:QueueAppendix} also contains corresponding bounds for the set $\U^{CS}_\epsilon$.  
}

%
\begin{figure}
\begin{subfigure}{0.45\textwidth}
	\includegraphics{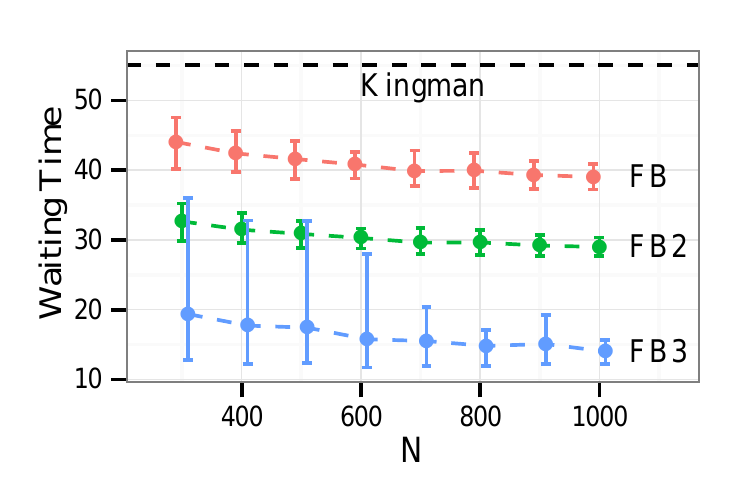}
\end{subfigure}
\begin{subfigure}{0.45\textwidth}
	\includegraphics{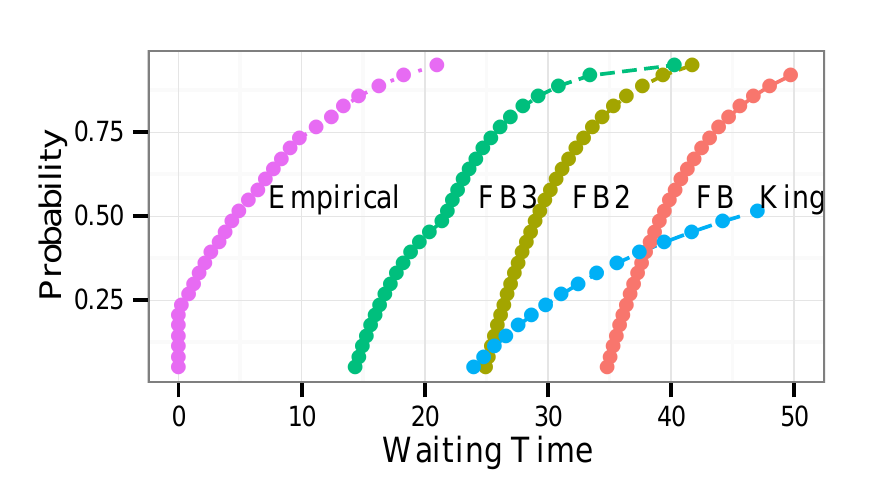}
\end{subfigure}
\caption{\label{fig:Queue} 
The left panel shows various bounds on the median waiting time ($\epsilon=.5$) for $n=10$ and various values of $N$.  The right panel bounds the entire cumulative distribution of the waiting time for $n=10$ and $N=1000$. using $W_n^{FB, 3}$.  In both cases, $\alpha = 20\%$.}   
\end{figure}

We illustrate these ideas numerically.  Let service times follow a Pareto distribution with parameter $1.1$ truncated at $15$, and the interarrival times follow an exponential distribution with rate $3.05$ truncated at 15.25.  The resulting truncated distributions have means of approximately $3.029$ and $3.372$, respectively, yielding an approximate 90\% utilization.

{\blockedit 
As a first experiment, we bound the median waiting time ($\epsilon = 50\%$) for the $n=10$ customer, using each of our bounds with differing amounts of data.  We repeat this procedure $100$ times to study the variability of our bounds with respect to the data.  The left panel of Fig.~\ref{fig:Queue} shows the average value of the bound and error bars for the 10\% and 90\% quantiles.  As can be seen, all of the bounds improve as we add more data.  Moreover, optimizing the $\epsilon_j$'s (the difference between $W_n^{FB, 1}$ and $W_n^{FB, 2}$ is significant.  

For comparison purposes, we include a sample analogue of Kingman's bound \citep{kingman1962some} on the $1-\epsilon$ quantile of the waiting time, namely, 
\[
W^{King} \equiv \frac{ \hat{\mu}_x (\hat{\sigma}_a^2 \hat{\mu}_x^2 + \hat{\sigma}_x^2 \hat{\mu}_t^2)} {2\overline{\epsilon} \hat{\mu}_t^2 (\hat{\mu}_t - \hat{\mu}_x)},
\]
where $\hat{\mu}_t, \hat{\sigma}^2_t$ are the sample mean and sample variance of the arrivals, $\hat{\mu}_x, \hat{\sigma}^2_x$ are the sample mean and sample variance of the service times, and we have applied Markov's inequality.   Unfortunately, this bound is extremely unstable, even for large $N$.  The dotted line in the left-panel of Fig.~\ref{fig:Queue} is the average value over the $100$ runs of this bound for $N=10,000$ data points (the error-bars do not fit on graph.)  Sample statistics for this bound and our bounds can also be seen in Table~\ref{tab:queue}.  As shown, our bounds are both significantly better (with less data), and exhibit less variability.  
\begin{table} \centering
\caption{\label{tab:queue} Summary statistics for various bounds on median waiting time.  $N=10,000$, $n=10$, $\alpha = 10\%$.  The last two columns refer to upper and lower quantiles over the simulation.}   
\begin{tabular}{@{}rrrrr@{}}
\toprule
\multicolumn{1}{l}{} & Mean & St. Dev &  10\% & 90\% \\ \midrule
$W_n^{FB, 1}$                   & 34.6 & 0.4     &  34.0  & 35.2  \\
$W_n^{FB, 2}$                  & 25.8 & 0.3     &  25.4 & 26.2\\
$W_n^{FB, 3}$                & 14.4 & 1.2     &  13.5  & 15.5\\
$W^{King}$              & 55.1 & 8.7     &  46.0 & 67.4  \\ \bottomrule
\end{tabular}
\end{table}

As a second experiment, we use our bounds to calculate a probabilistic upper bound on the entire CDF of $\tilde{W}_n$ for $n=10$ with $N=1,000$, $\alpha=20\%$.  Results can be seen in the right panel of Fig.~\ref{fig:Queue}.  We have included the empirical CDF of the waiting time and the sampled version of the Kingman bound comparison.  As seen, our bounds significantly improve upon the sampled Kingman bound, and the benefit of optimizing the $\epsilon_j$'s is again, significant.  We remark that the ability to simultaneously bound the entire CDF for any $n$, whether transient or steady-state, is an important strength of this type of analysis.  
}

\section{Conclusions}
\label{sec:Conclusion}
The prevalence of high quality data is reshaping operations research.  Indeed, a new data-centered paradigm is emerging.  
In this work, we took a first step towards adapting traditional robust optimization techniques to this new paradigm.  Specifically, we proposed a novel schema for designing uncertainty sets for robust optimization from data using hypothesis tests.  Sets designed using our schema imply a probabilistic guarantee and are typically much smaller than corresponding data poor variants.  Models built from these sets are thus less conservative than conventional robust approaches, yet retain the same robustness guarantees.  

\theendnotes

\ACKNOWLEDGMENT{Part of this work was supported by the National Science Foundation Graduate Research Fellowship under Grant No. 1122374.  \edit{We would also like to thank two anonymous reviewers and the Associate Editor for their insightful and constructive comments.  They greatly helped to improve the quality of the paper.}}

\SingleSpacedXI
\bibliographystyle{ormsv080} 
\bibliography{DataDrivenUSets.bib} 
\OneAndAHalfSpacedXI

\ECSwitch


\ECHead{Appendices}

\section{Omitted Proofs}
\label{sec:Proofs}
\subsection{Proof of Theorem~\ref{thm:support}}
{\blockedit
\proof{Proof}
For the first part, let $\bx^*$ be robust feasible in \eqref{def:Guarantee} and consider the closed, convex set $\{ \bu \in \R^d : f(\bu, \bx^*) \geq t \}$ where $t > 0$.  That $\bx^*$ is robust feasible implies $\max_{\bu \in \U} f(\bu, \bx^*) \leq 0$ which implies that $\U$ and $\{ \bu \in \R^d : f(\bu, \bx^*) \geq t \}$ are disjoint.  From the separating hyperplane theorem, there exists a strict separating hyperplane $\bv^T\bu = v_0$ such that
$v_0 > \bv^T \bu$ for all $\bu \in \U$ and $\bv^T\bu < v_0$ for all $\bu \in \{ \bu \in \R^d: f(\bu, \bx^*) \geq t \}$ .  Observe
\[
v_0 > \max_{\bu \in \U} \bv^T\bu = \delta^*(\bv | \ \U) \geq \quant{\epsilon}{\P}(\bv),
\]
and 
\[
\P( f(\burv, \bx^*) \geq t ) \leq \P(\bv^T\burv > v_0 ) \leq \P(\bv^T\burv > \quant{\epsilon}{\P}(\bv)) \leq \epsilon. 
\]
Taking the limit as $t \downarrow 0$ and using the continuity of probability proves $\P( f(\burv, \bx^*) > 0 ) \leq \epsilon$ and that \eqref{def:Guarantee} is satisfied.

For the second part of the theorem, let $t > 0$ be such that $\delta^*( \bv | \ \U) \leq \quant{\epsilon}{\P}(\bv) - t$.  Define $f(\bu, x) \equiv \bv^T\bu - x$.  Then $x^* = \delta(\bv | \ \U)$ is robust feasible in \eqref{def:Guarantee}, but 
\[
\P(f(\burv, \bx) > 0 ) = \P(\burv^T\bv  > \delta(\bv | \ \U))   \geq \P(\burv^T\bv \geq \quant{\epsilon}{\P}(\bv) - t) > \epsilon
\]
by \eqref{def:quantile}.

%
\hfill \Halmos \endproof
}

\subsection{Proofs of Theorems~\ref{thm:schema} and \ref{thm:uniform}}
{ \blockedit
\begin{proof}{Proof of Theorem~\ref{thm:schema}.}
\begin{align*}
\P^*_\S( \U(\S, \epsilon, \alpha) &\text{ implies a probabilistic guarantee at level $\epsilon$ for $\P^*$})
\\
&= \P^*_\S( \delta^*(\bv | \ \U(\S, \epsilon, \alpha) ) \geq \quant{\epsilon}{\P^*}(\bv) \ \forall \bv \in \R^d) && \text{(Theorem~\ref{thm:support})}
\\
&\geq \P^*_\S( \P^* \in \mathcal{P}(\S, \epsilon, \alpha) ) && \text{(Step~\ref{step:II} of schema)}
\\
&\geq 1-\alpha && \text{(Confidence region).}
\end{align*}
\end{proof}

\begin{proof}{Proof of Theorem~\ref{thm:uniform}.}
For the first part,
\begin{align*} 
\P^*_\S( \{ \U(\S, &\epsilon, \alpha): \ 0 < \epsilon < 1 \} \text{ simultaneously implies a probabilistic guarantee})
\\
&\begin{aligned}
&= \P^*_\S(\delta^*(\bv | \ \U(\S, \epsilon, \alpha) ) \geq \quant{\epsilon}{\P^*}(\bv) \ \forall \bv \in \R^d, \ 0 < \epsilon < 1 ) && \text{(Theorem~\ref{thm:support})}
\\
&\geq \P^*_\S( \P^* \in \bigcap_{\epsilon : 0 \leq \epsilon \leq 1} \mathcal{P}(\S, \epsilon, \alpha) ) && \text{(Step~\ref{step:II} of schema)}
\\
&= \P^*_\S( \P^* \in \mathcal{P}(\S, \alpha) ) && \text{$(\mathcal{P}(\S, \alpha))$ is independent of $\epsilon$)}
\\
&\geq 1-\alpha && \text{(Confidence region).}
\end{aligned}
\end{align*}
For the second part, let $\epsilon_1, \ldots, \epsilon_m$ denote any feasible $\epsilon_j$'s in \eqref{eq:optMultRob}.
\begin{align*}
1-\alpha &\leq \P^*_\S( \{ \U(\S, \epsilon, \alpha): \ 0 < \epsilon < 1 \} \text{ simultaneously implies a probabilistic guarantee})
\\
& \leq  \P^*_\S( \U(\S, \epsilon_j, \alpha) \text{ implies a probabilistic guarantee at level } \epsilon_j, j=1, \ldots, m).
\end{align*}
Applying the union-bound and Theorem~\ref{thm:schema} yields the result.  
\end{proof}
}

\subsection{Proof of Theorem~\ref{thm:discretesets} and Proposition~\ref{prop:Taylor}}
\label{sec:ProofTaylor}

We require the following well-known result.
\begin{theorem}[Rockafellar and Ursayev, 2000]
\label{thm:CVARRep}
Suppose $\supp(\P) \subseteq \{\ba_0, \ldots, \ba_{n-1} \}$ and let $\P(\burv = \ba_j) = p_j$.  Let
\begin{equation}
\label{def:UCVAR}
\U^{\CVAR_\epsilon^\P} = \left\{ \bu \in \R^d : \bu = \sum_{j=0}^{n-1} q_j \ba_j, \ \bq \in \Delta_n, \ \bq \leq \frac{1}{\epsilon} \bp \right\}.
\end{equation}
Then,
$
\delta^*(\bv | \ \U^{\CVAR_\epsilon^{\P}})
 = \CVAR^{\P}(\bv)
$.
\end{theorem} 
We now prove the theorem.
\proof{Proof of Theorem~\ref{thm:discretesets}:}
We prove the theorem for $\U^{\chi^2}_\epsilon$.  The proof for $\U^G_\epsilon$ is similar.  From Thm.~\ref{thm:schema}, it suffices to show that  $\delta^*( \bv | \ \U^{\chi^2}_\epsilon)$ is an upper bound to $\sup_{\P \in \mathcal{P}^{\chi^2}} \quant{\epsilon}{\P}(\bv)$:
\begin{align*}
\sup_{\P \in \mathcal{P}^{\chi^2}} \quant{\epsilon}{\P}(\bv) &\leq \sup_{\P \in \mathcal{P}^{\chi^2}} \CVAR_\epsilon^\P(\bv) && (\CVAR \text{ is an upper bound to } \quant{}{} )
\\
&= \sup_{\P \in \mathcal{P}^{\chi^2}} \max_{\bu \in \U^{\CVAR_\epsilon^\P}} \bu^T\bv
&& \text{(Thm.~\ref{thm:CVARRep})} 
\\
& = \max_{\bu \in \U^{\chi^2}_\epsilon} \bu^T\bv &&\text{(Combining Eqs.~\eqref{eq:ChisqU} and \eqref{def:ChisqFamily}).} 
\end{align*}

{\blockedit
To obtain the expression for $\delta^*(\bv | \ \U^{\chi^2}_\epsilon)$ observe, 
\begin{align*}
\delta^*(\bv | \ \U^{\chi^2}_\epsilon) 
&= \inf_{\bw \geq 0} \left\{ \max_{\bq \in \Delta_n} \sum_{i=0}^{n-1} q_i (\ba_i^T\bv - w_i) + 
\frac{1}{\epsilon} \max_{\bp \in \mathcal{P}^{\chi^2} } \bw^T\bp \right\}, 
\end{align*}
from Lagrangian duality.  The optimal value of the first maximization is $\beta = \max_i \ba_i^T\bv - w_i$. The second maximization is of the form studied in \cite[Corollary 1]{ben2013robust} and has optimal value 
\[
\eta + \frac{\lambda \chi^2_{n-1, 1-\alpha}}{N} + 2\lambda - 2 \sum_{i=0}^{n-1} \hat{p}_i \sqrt{\lambda}\sqrt{\lambda  + \eta - w_i}.
\]
Using the second-order cone representation of the hyperbolic constraint $s_i^2 \leq \lambda \cdot (\lambda + \eta - w_i)$ \citep{lobo1998applications} and simplifying we obtain the result.  
}
\hfill \Halmos
\endproof

\proof{Proof of Proposition~\ref{prop:Taylor}.}
Let $\Delta_j \equiv \frac{\hat{p}_j - p_j}{p_j}$.  Then, 
$D( \hat{\bp}, \bp) = \sum_{j=0}^{n-1} \hat{p}_j \log( \hat{p}_j / p_j ) = \sum_{j=0}^{n-1} p_j (\Delta_j + 1) \log( \Delta_j + 1)$.  Using a Taylor expansion of $x \log x $ around $x=1$ yields, 
\begin{equation}\label{eq:GandChisqEquiv}
D( \hat{\bp}, \bp) =  \sum_{j=0}^{n-1} p_j \left(\Delta_j + \frac{\Delta_j^2}{2} + O( \Delta_j^3) \right) = 
\sum_{j=0}^{n-1} \frac{(\hat{p}_j - p_j)^2}{2p_j} + \sum_{j=0}^{n-1} O(\Delta_j^3),
\end{equation}
where the last equality follows by expanding out terms and observing that $\sum_{j=0}^{n-1} \hat{p}_j  = \sum_{j=0}^{n-1} p_j = 1$.
Next, note
$
\bp \in \mathcal{P}^G \implies \hat{p}_j / p_j \leq \exp( \frac{\chi^2_{n-1, 1-\alpha}}{2N\hat{p}_j} ). 
$
From the Strong Law of Large Numbers, for any $0 < \alpha^\prime < 1$, there exists $M$ such that $\hat{p}_j \geq p^*_j / 2$ with probability at least $1 - \alpha^\prime$ for all $j = 0, \ldots, n-1$, simultaneously.  It follows that for $N$ sufficiently large, with probability $1-\alpha^\prime$,
$\bp \in \mathcal{P}^G \implies \hat{p}_j / p_j \leq \exp( \frac{\chi^2_{n-1, 1-\alpha}}{Np^*_j} )$ which implies that $| \Delta_j | \leq \exp( \frac{\chi^2_{n-1, 1-\alpha}}{Np^*_j} ) -1 = O(N^{-1})$.  
Substituting into \eqref{eq:GandChisqEquiv} completes the proof.  
\hfill \Halmos \endproof

\subsection{Proof of Theorems~\ref{thm:UI} and \ref{thm:UFB}}
\label{sec:proofsUI}
We first prove the following auxiliary result that will allow us to evaluate the inner supremum in \eqref{eq:IndepConjugate}.  
\begin{theorem} \label{thm:KSFinite} 
Suppose $g(u)$ is monotonic. Then,
\begin{equation}
\label{eq:DisjKS}
\sup_{\P_i \in \InfCalP^{KS}_i} \E^{\P_i}[ g(\urv_i) ] = \max\left( \sum_{j=0}^{N+1} q^L_j(\Gamma^{KS}) g(\uhat_i^{(j)}), \sum_{j=0}^{N+1} q^R_j(\Gamma^{KS}) g(\uhat_i^{(j)}) \right)
\end{equation}
\end{theorem}

\proof{Proof.}
Observe that the discrete distribution which assigns mass $q^L_j(\Gamma^{KS})$ (resp. $q^R_j(\Gamma^{KS})$) to the point $\uhat^{(j)}$ for $j=0, \ldots, N+1$ is an element of $\InfCalP^{KS}_i$.  Thus, Eq.~\eqref{eq:DisjKS} holds with ``$=$" replaced by ``$\geq$".  

For the reverse inequality, we have two cases.  Suppose first that $g(u_i)$ is non-decreasing.  Given $\P_i \in \InfCalP_i^{KS}$, consider the measure $\Q$ defined by 
\begin{align} \label{eq:nonDecrQ}
&\Q(\urv_i = \uhat_i^{(0)}) \equiv 0, 
\ \  
\Q(\urv_i = \uhat_i^{(1)}) \equiv \P_i(\uhat_i^{(0)} \leq \urv_i \leq \uhat_i^{(1)}), 
\\ \notag
&\Q(\urv_i = \uhat_i^{(j)}) \equiv \P_i(\uhat_i^{(j-1)} < \urv_i \leq \uhat_i^{(j)}), \ \ j = 2, \ldots, N+1.
\end{align}
Then, $\Q \in \InfCalP^{KS}$, and since $g(u_i)$ is non-decreasing, $\E^{\P_i}[g(\urv_i)] \leq \E^\Q[g(\urv_i)]$.  Thus, the measure attaining the supremum on the left-hand side of Eq.~\eqref{eq:DisjKS} has discrete support $\{\uhat_i^{(0)}, \ldots, \uhat_i^{(N+1)} \}$, and the supremum is equivalent to the linear optimization problem:
{ \small
\begin{align} \notag
\max_\bp \quad & \sum_{j=0}^{N+1} p_j g(\uhat^{(j)})
\\ \label{eq:LPForKS}
\text{s.t.} \quad & 
\bp \geq \bzero, \ \ \be^T\bp = 1, 
\\ \notag
& \sum_{k=0}^j p_k \geq \frac{j}{N} - \Gamma^{KS},  \quad \sum_{k=j}^{N+1} p_k \geq \frac{N-j +1}{N} - \Gamma^{KS}, \quad j=1, \ldots, N, 
\end{align}
}
(We have used the fact that $\P_i(\urv_i < \uhat_i^{(j)}) = 1 - \P_i(\urv \geq \uhat_i^{(j)})$.) Its dual is:
{\small
\begin{align*}
\min_{\bx, \by, t} \quad &\sum_{j=1}^N x_j \left(\Gamma^{KS}- \frac{j}{N}\right) + \sum_{j=1}^N y_j\left(\Gamma^{KS} - \frac{N-j+1}{N}\right) + t
\\
\text{s.t.} \quad &t -\sum_{k \leq j \leq N} x_j -\sum_{1 \leq j \leq k} y_j \geq g(\uhat^{(k)}), \ \ k = 0, \ldots, N + 1, 
\\
&\bx, \by \geq \bzero.
\end{align*}
}
Observe that the primal solution $\bq^R(\Gamma^{KS})$ and dual solution $\by = \bzero$, $t = g(\uhat_i^{(N+1)})$ and 
\[
x_j = \begin{cases} 
		g(\uhat_i^{(j+1)}) - g(\uhat_i^{(j)}) &\text{ for } N-j^* \leq j \leq N,
		\\
		0 &\text{ otherwise},
	\end{cases}
\]
constitute a primal-dual optimal pair.  This proves \eqref{eq:DisjKS} when $g$ is non-decreasing.  The case of $g(u_i)$ non-increasing is similar.  

\hfill \Halmos \endproof

\proof{Proof of Theorem~\ref{thm:UI}.}
Notice by Theorem~\ref{thm:KSFinite}, Eq.~\eqref{eq:IndepConjugate} is equivalent to the given expression for $\delta^*(\bv | \ \U^I_\epsilon)$.  
By our schema, it suffices to show then that this expression is truly the support function of $\U^I_\epsilon$.   
By Lagrangian duality, 
\begin{align*}
	\delta^*(\bv | \ \U^I_\epsilon) = \inf_{\lambda \geq 0} \left( 
	\begin{aligned}
	\lambda \log(1/\epsilon) + 
	\max_{\bq, \btheta} \quad & \sum_{i=1}^d v_i \sum_{j=0}^{N+1} \uhat_i^{(j)} q^i_j - \lambda \sum_{i=1}^d D(\bq^i, \theta_i \bq^L + (1-\theta_i)\bq^R )
	\\
	\text{s.t.} \quad &  \bq^i \in \Delta_{N+2}, 0\leq \theta_i \leq 1, \ \ i = 1, \ldots, d.
	 \end{aligned}
	  \right)
\end{align*}
The inner maximization decouples in the variables indexed by $i$.  The $i^{\text{th}}$ subproblem is
\begin{align*}
\max_{\theta_i \in [0, 1]} \lambda \left\{ \max_{\bq_i \in \Delta_{N+2}} \left\{ \sum_{j=0}^{N+1} \frac{v_i\uhat_i^{(j)}}{\lambda}   q_{ij} - D(\bq^i, \theta_i \bq^L + (1-\theta_i)\bq^R )\right\} \right\}.
\end{align*}
The inner maximization can be solved analytically \citep[][pg. 93]{boyd2004convex}, yielding:
%
%
\begin{equation}
\label{eq:OptRelEntropy}
q_j^i = \frac{p_j^i e^{v_i \uhat_{i}^{(j)} /\lambda}}{\sum_{j=0}^{N+1} p_j^i e^{v_i \uhat_{i}^{(j)} /\lambda} }, 
\quad
p_j^i = \theta_i q^L_j(\Gamma^{KS}) + (1-\theta_i)q^R_j(\Gamma^{KS}).
\end{equation}
Substituting in this solution and recombining subproblems yields 
\begin{equation} \label{eq:PartialSupp}
\lambda \log(1/\epsilon) + \lambda \sum_{i=1}^d  \log \left(\max_{\theta_i \in [0, 1] }\sum_{j=0}^{N+1} (\theta_i q^L_j(\Gamma^{KS}) + (1-\theta_i)q^R_j(\Gamma^{KS})) e^{v_i \uhat_{i}^{(j)} /\lambda}. \right)
\end{equation}
The inner optimizations over $\theta_i$ are all linear, and hence achieve an optimal solution at one of the end points, i.e., either $\theta_i = 0$ or $\theta_i =1$.  This yields the given expression for $\delta^*(\bv | \ \U)$.

Following this proof backwards to identify the optimal $\bq^i$, and, thus, $\bu \in \U^I$ also proves the validity of the procedure given in Remark~\ref{rem:UIAlg}
\hfill \Halmos \endproof

\label{sec:ProofFwdBack}
\proof{Proof Theorem~\ref{thm:UFB}.}
By inspection, \eqref{eq:suppFcnFB} is the worst-case value of \eqref{eq:ChenBound} over $\mathcal{P}^{FB}$. By Theorem~\ref{thm:uniform}, it suffices to show that this expression truly is the support function of $\U^{FB}_\epsilon$.  
 First observe
\[
\max_{\bu \in \U^{FB}_\epsilon} \bu^T\bv = \min_{\lambda \geq 0} \left\{ \lambda \log(1/\epsilon) +  \max_{\substack{\bm_b \leq \by_1 \leq \bm_b, \\ \by_2 \geq, \by_3 \geq \bzero}} \sum_{i=1}^d v_i (y_{1i} + y_{2i} - y_{3i} ) - \lambda \sum_{i=1}^d \frac{y_{2i}^2 }{2 \overline{\sigma}_{fi}^2} + \frac{y_{3i}^2 }{2 \overline{\sigma}_{bi}^2}\right\}
\]
by Lagrangian strong duality.  The inner maximization decouples by $i$.  The $i^{\text{th}}$ subproblem further decouples into three sub-subproblems.  The first is
$ \max_{m_{bi} \leq y_{i1} \leq m_{fi} } v_i y_{1i} $
with optimal solution
\[
y_{1i} = \begin{cases} m_{fi} &\text{ if } v_i \geq 0, \\ m_{bi} & \text{ if } v_i < 0. \end{cases}
\]
The second sub-subproblem is $\max_{y_{2i} \geq 0} v_i y_{2i} - \lambda \frac{y_{2i}^2}{2 \overline{\sigma}_{fi}^2}$.  This is maximizing a concave quadratic function of one variable.  Neglecting the non-negativity constraint, the optimum occurs at $y_{2i}^* = \frac{ v_i\sigma^2_{fi}}{\lambda}$.  If this value is negative, the optimum occurs at $y_{2i}^* = 0$.  Consequently, 
\[
\max_{y_{2i} \geq 0} \ \ v_i y_{2i} - \lambda \frac{y_{2i}^2}{2 \overline{\sigma}_{fi}^2}
= 
\begin{cases}
	\frac{ v_i \sigma^2_{fi}}{2\lambda} &\text{ if } v_i \geq 0,
	\\
	0 &\text{ if } v_i < 0.
\end{cases}
\]
Similarly, we can show that the third subproblem has the following optimum value
\[
\max_{y_{3i} \geq 0} \ \ -v_i y_{3i} - \lambda \frac{y_{3i}^2}{2 \overline{\sigma}_{bi}^2}
= 
\begin{cases}
	\frac{ v_i \sigma^2_{bi}}{2\lambda} &\text{ if } v_i \leq 0,
	\\
	0 &\text{ if } v_i > 0.
\end{cases}
\]
Combining the three sub-subproblems yields 
\[
\delta^*(\bv | \U^{FB}_\epsilon) = \sum_{i : v_i > 0 } v_i m_{fi} + \sum_{i: v_i \leq 0 } v_i m_{bi}  + \min_{\lambda \geq 0} \lambda \log(1/\epsilon) + \frac{1}{2\lambda} \left( \sum_{i: v_i > 0 } v_i^2 \overline{\sigma}_{fi}^2 + \sum_{i: v_i \leq 0 } v_i^2 \overline{\sigma}_{bi}^2 \right).
\]
This optimization can be solved closed-form, yielding 
\[
\lambda^* = \sqrt{\frac{\sum_{i: v_i > 0 } v_i^2 \overline{\sigma}_{fi}^2 + \sum_{i: v_i \leq 0 } v_i^2 \overline{\sigma}_{bi}^2}{2 \log (1/\epsilon)} }.
\]  
Simplifying yields the right hand side of \eqref{eq:suppFcnFB}.  Moreover, following the proof backwards to identify the maximizing $\bu \in \U^{FB}_\epsilon$ proves the validity of the procedure given in Remark~\ref{rem:UFBAlg}.
\hfill \Halmos \endproof


\subsection{Proof of Theorem~\ref{thm:UM}.}
\proof{Proof.} 
Observe, 
\begin{align} \label{eq:WCMarginalQuantile} 
\sup_{\P \in \InfCalP^M} \quant{\epsilon}{\P}(\bv) &\leq \sup_{\P \in \InfCalP^M} \sum_{i=1}^d \quant{\epsilon/d}{\P}(v_i \be_i)
= \sum_{i : v_i > 0 } v_i \uhat_i^{(s)} + \sum_{i : v_i \leq 0 } v_i \uhat_i^{(N-s+1)},
\end{align}
where the equality follows rom the positive homogeneity of $\quant{\epsilon}{\P}$,
and this last expression is equivalent to \eqref{eq:suppFcnUM} because $\uhat_i^{(N-s+1)} \leq  \uhat_i^{(s)}$.  By Theorem~\ref{thm:schema}, it suffices to show that $\delta^*(\bv | \ \U^M)$ truly is the support function of $\U^M_\epsilon$, and this is immediate.  
\hfill \Halmos \endproof

{\blockedit
\subsection{Proof of Theorem~\ref{thm:ULCX}.}
\proof{Proof.}
We first compute $\sup_{\P \in \InfCalP^{LCX}} \P(\bv^T\burv > t)$ for fixed $\bv, t$.  In this spirit of \cite{shapiro2001duality, BGKII}, this optimization admits the following strong dual:
\begin{align} \notag
\inf_{\theta, w_\sigma, \lambda(\ba, b)} \quad & \theta + \left(\frac{1}{N}\sum_{j=1}^N \| \buhat_j \|^2 - \Gamma_\sigma\right) w_\sigma + \int_{\mathcal{B}} \Gamma(\ba, b) d\lambda(\ba, b)
\\ \label{eq:wConstraint}
\text{s.t.} \quad &\theta - w_\sigma \| \bu \|^2 + \int_{\mathcal{B}} (\ba^T \bu - b)^+ d \lambda(a, b) \geq \I( \bu^T\bv > t ) \ \ \forall \bu \in \R^d, 
\\ \notag
&w_\sigma \geq 0, \ \ d\lambda(\ba, b) \geq 0,
\end{align}
where $\Gamma(\ba, b) \equiv \frac{1}{N} \sum_{j=1}^N (\ba^T\buhat_j - b)^+ + \Gamma_{LCX}$.  We claim that $w_\sigma = 0$ in any feasible solution.  Indeed, suppose $w_\sigma > 0$ in some feasible solution.  Note $(\ba, b) \in \mathcal{B}$ implies that $(\ba^T \bu - b)^+ = O(\|\bu\|)$ as $\|\bu\| \rightarrow \infty$.  Thus, the left-hand side of eq.~\eqref{eq:wConstraint} tends to $-\infty$ as $\| \bu \| \rightarrow \infty$ while the right-hand side is bounded below by zero.  This contradicts the feasibility of the solution.

Since $w_\sigma = 0$ in any feasible solution, rewrite the above as
\begin{align} \notag
\inf_{\theta, \lambda(\ba, b)} \quad & \theta + \int_{\mathcal{B}} \Gamma(\ba, b) d\lambda(\ba, b)
\\ \label{eq:WCProbLCX}
\text{s.t.} \quad &\theta  + \int_{\mathcal{B}} (\ba^T \bu - b)^+ d \lambda(a, b) \geq 0 \ \ \forall \bu \in \R^d, 
\\ \notag
&\theta  + \int_{\mathcal{B}} (\ba^T \bu - b)^+ d \lambda(a, b) \geq 1 \ \ \forall \bu\in \{\bu \in \R^d :
\bu^T\bv > t \},
\\ \notag
&d\lambda(\ba, b) \geq 0.
\end{align}
The two infinite constraints can be rewritten using duality.  Specifically, the first constraint is
 \begin{align*}
-\theta \leq \quad \min_{s(\ba, b) \geq 0, \burv \in \R^d} \quad &\int_{\mathcal{B}} s(\ba, b) d\lambda(\ba, b)
\\
\text{s.t.} \quad & s(\ba, b) \geq (\ba^T\burv - b) \quad \forall (\ba, b) \in \mathcal{B}, 
\end{align*}
which admits the dual:
\begin{align*}
-\theta \leq \max_{y_1(\ba, b)} \quad & -\int_{\mathcal{B}} b \ dy_1(\ba, b)
\\
\text{s.t.} \quad & 0 \leq dy_1(\ba, b) \leq d\lambda(\ba, b) \quad \forall (\ba, b) \in \mathcal{B}, 
\\
&\int_{\mathcal{B}} \ba \  dy_1(\ba, b) = 0.
\end{align*}
The second constraint can be treated similarly using continuity to take the closure of $\{\bu \in \R^d : \bu^T\bv > t\}$.  Combining both constraints yields the equivalent representation of \eqref{eq:WCProbLCX}
\begin{align} \notag 
\inf_{\substack{\theta, \tau, \lambda(\ba, b), \\ y_1(\ba, b), y_2(\ba, b)}} \quad & \theta + \int_{\mathcal{B}} \Gamma(\ba, b) d\lambda(\ba, b)
\\ \notag
\text{s.t.} \quad &\theta  -\int_{\mathcal{B}} b \ dy_1(\ba, b)  \geq 0 ,
\ \ \notag
\theta + t \tau - \int_{\mathcal{B}} b \ dy_2(\ba, b) \geq 1,
\\ \label{eq:WCProbLCXDual}
&0 \leq dy_1(\ba, b) \leq d\lambda(\ba, b) \quad \forall (\ba, b) \in \mathcal{B}, 
\\ \notag
&0 \leq dy_2(\ba, b) \leq d\lambda(\ba, b) \quad \forall (\ba, b) \in \mathcal{B},
\\ \notag
&\int_{\mathcal{B}} \ba \  dy_1(\ba, b) = 0,
\ \ \notag
 \tau \bv = \int_{\mathcal{B}} \ba \ d y_2(\ba, b),
\\ \notag
&\tau \geq 0.
\end{align}

Now the worst-case Value at Risk can be written as
\begin{align*}
\sup_{\P \in \InfCalP^{LCX}} \quant{\epsilon}{\P}(\bv) = \inf_{\substack{\theta, \tau, t, \lambda(\ba, b), \\ y_1(\ba, b), y_2(\ba, b)}} t
\\
\text{s.t.} \quad & \theta + \int_{\mathcal{B}} \Gamma(\ba, b) d\lambda(\ba, b) \leq \epsilon,
\\ &
(\theta, \tau, \lambda(\ba, b), y_1(\ba, b), y_2(\ba, b), t) \text{ feasible in \eqref{eq:WCProbLCXDual} }.
\end{align*}
We claim that $\tau > 0$ in an optimal solution.  Suppose to the contrary that $\tau=0$ in some solution.  Let $t \rightarrow -\infty$ in this solution.  The resulting solution remains feasible, implying that $\P(\burv^T\bv > -\infty) \leq \epsilon$ for all $\P \in \InfCalP^{LCX}$.  However, the empirical distribution $\hat{\P} \in \InfCalP^{LCX}$, a contradiction.

Since $\tau > 0$, apply the transformation
$(\theta/\tau, 1/\tau, \lambda(\ba, b)/\tau, \by(\ba, b)/\tau) \rightarrow (\theta, \tau, \lambda(\ba, b), \by(\ba, b))$
yielding
\begin{align*}
\inf_{\substack{\theta, \tau, t, \lambda(\ba, b), \\ y_1(\ba, b), y_2(\ba, b)}} t
\\
\text{s.t.} \quad & \theta + \int_{\mathcal{B}} \Gamma(\ba, b) d\lambda(\ba, b) \leq \epsilon \tau
\\
\quad &\theta  -\int_{\mathcal{B}} b \ dy_1(\ba, b)  \geq 0 ,
\ \ 
\theta + t - \int_{\mathcal{B}} b \ dy_2(\ba, b) \geq \tau,
\\ \label{eq:WCProbLCXDual}
&0 \leq dy_1(\ba, b) \leq d\lambda(\ba, b) \quad \forall (\ba, b) \in \mathcal{B}, 
\\ \notag
&0 \leq dy_2(\ba, b) \leq d\lambda(\ba, b) \quad \forall (\ba, b) \in \mathcal{B},
\\ \notag
&\int_{\mathcal{B}} \ba \  dy_1(\ba, b) = 0,
\ \
\bv = \int_{\mathcal{B}} \ba \ d y_2(\ba, b),
\\ \notag
&\tau \geq 0.  
\end{align*}
Eliminate the variable $t$,  and make the transformation $(\tau \epsilon, \theta - \int_{\mathcal{B}} b dy_1(\ba, b)) \rightarrow (\tau, \theta)$ to yield the righthand side of \eqref{eq:suppLCX}.

By Theorem~\ref{thm:uniform}, it suffices to show that the right hand side of \eqref{eq:suppLCX} is indeed the support function of $\U^{LCX}_\epsilon$.  Take the dual of \eqref{eq:suppLCX} and simplify to yield the given description of $\U^{LCX}_\epsilon$.

\hfill \Halmos \endproof   

}
\subsection{Proofs of Theorems~\ref{thm:CSSet} and \ref{thm:DY}}
\label{sec:ProofCSSet}
\label{sec:ProofLemmaDY}
\proof{Proof of Theorem~\ref{thm:CSSet}.}
By Theorem~\ref{thm:uniform}, it suffices to show that $\delta^*(\bv | \ \U^{CS}_\epsilon)$ is given by \eqref{eq:suppCS}, which follows immediately from two applications of the Cauchy-Schwartz inequality.  
\\ \hfill \Halmos \endproof

 \[ 
 \]

To prove Theorem~\ref{thm:DY} we require the following proposition:
\begin{proposition}
\label{proposition:DYBound}
\begin{align}
\label{eq:DYProbBound}
\sup_{\P \in \mathcal{P}^{DY}} \P( \burv^T\bv  > t ) = \min_{r, s, \theta, \by_1, \by_2, \bZ} \quad & r + s
\\ \notag 
\text{s.t.} \quad & 
\begin{pmatrix}
	r  + \by_1^{+T} \buhat^{(0)} - \by_1^{-T} \buhat^{(N+1)} & \frac{1}{2} (\bq - \by_1)^T, 
\\ \notag
	\frac{1}{2} (\bq - \by_1) & \bZ 
\end{pmatrix} \succeq \bzero,
\\ \notag
& \begin{pmatrix}
	r  + \by_2^{+T} \buhat^{(0)} - \by_2^{-T} \buhat^{(N+1)} + \theta t - 1 & \frac{1}{2} (\bq - \by_2 - \theta \bv)^T, 
\\ \notag
	\frac{1}{2} (\bq - \by_2 - \theta \bv) & \bZ 
\end{pmatrix} \succeq \bzero,
\\ \notag
& s \geq (\gamma^B_2 \hat{\bSigma} + \hat{\bmu} \hat{\bmu}^T) \circ \bZ + \hat{\bmu}^T\bq + \sqrt{\gamma^B_1} \| \bq + 2 \bZ \hat{\bmu} \|_{\hat{\bSigma}^{-1}},
\\ \notag
&\by_1 = \by_1^+ - \by_1^-,  \ \ \by_2 = \by_2^+ - \by_2^-,
\ \ 
 \by_1^+, \by_1^-, \by_2^+,\by_2^- \theta \geq \bzero.
\end{align}
\end{proposition}
\proof{Proof.}
We claim that $\sup_{\P \in \mathcal{P}^{DY}} \P(\burv^T\bv > t)$ has the following dual representation:
\begin{align} \nonumber
 \min_{r, s, \bq, \bZ, \by_1, \by_2, \theta} \quad r + s
\\ \nonumber
\text{s.t.} \quad & r + \bu^T\bZ\bu + \bu^T\bq \geq 0 \quad \forall \bu \in [\buhat^{(0)}, \buhat^{(N+1)}],
\\ \label{eq:DualDelageYehBound}
&  r + \bu^T\bZ\bu + \bu^T\bq \geq 1 \quad \forall \bu \in [\buhat^{(0)}, \buhat^{(N+1)}] \cap \{ \bu : \bu^T\bv > t \},
\\ \nonumber
& s \geq (\gamma^B_2 \hat{\bSigma} + \hat{\bmu} \hat{\bmu}^T) \circ \bZ + \hat{\bmu}^T\bq +, \sqrt{\gamma^B_1} \| \bq + 2 \bZ \hat{\bmu} \|_{\hat{\bSigma}^{-1}},
\\ \nonumber
& \bZ \succeq \bzero.
\end{align}
See the proof of Lemma 1 in \citet{delage2010distributionally} for details.
Since $\bZ$ is positive semidefinite, we can use strong duality to rewrite the two semi-infinite constraints:
\begin{align*}
\begin{aligned}
\min_{\bu} \quad &\bu^T \bZ \bu + \bu^T \bq 
\\
\text{s.t.} \quad & \buhat^{(0)} \leq  \bu \leq \buhat^{(N+1)},
\end{aligned}
&\iff
\begin{aligned}
\max_{\by_1, \by_1^+, \by_1^-} \quad & -\frac{1}{4} (\bq - \by_1)^T \bZ^{-1} (\bq - \by_1) + \by_1^+\buhat^{(0)} - \by_1^{-}\buhat^{(N+1)}
\\
\text{s.t.} \quad & \by_1 = \by_1^+ - \by_1^-, \ \ \by_1^+, \by_1^- \geq \bzero,
\end{aligned}
\\
\\
\begin{aligned}
\min_{\bu} \quad &\bu^T \bZ \bu + \bu^T \bq
\\
\text{s.t.} \quad & \buhat^{(0)} \leq  \bu \leq \buhat^{(N+1)},
\\
& \bu^T\bv \geq t, 
\end{aligned}
&\iff
\begin{aligned}
\max_{\by_2, \by_2^+, \by_2^-} \quad & -\frac{1}{4} (\bq - \by_2 - \theta \bv)^T \bZ^{-1} (\bq - \by_2 - \theta \bv) + \by_2^+ \buhat^{(0)} - \by_2^{-} \buhat^{(N+1)} + \theta t
\\
\text{s.t.} \quad & \by_2 = \by_2^+ - \by_2^-, \ \ \by_2^+, \by_2^- \geq \bzero, \ \theta \geq 0.
\end{aligned}
\end{align*}
Then, by using Schur-Complements, we can rewrite  Problem~\eqref{eq:DualDelageYehBound} as in the proposition.  
\hfill \Halmos \endproof
\noindent We can now prove the theorem.  
\proof{Proof of Thm.~\ref{thm:DY}.}
Using Proposition~\ref{proposition:DYBound}, we can characterize the worst-case VaR by
\begin{equation}
\label{eq:Attempt1}
\sup_{\P \in \mathcal{P}^{DY}} \quant{\epsilon}{\P}(\bv)  = \inf \left\{ t: r + s  \leq \epsilon, (r, s, t,  \theta, \by_1, \by_2, \bZ) \text{ are feasible in problem~\eqref{eq:DYProbBound}} \right\}.
\end{equation}
We claim that $\theta > 0$ in any feasible solution to the infimum in Eq.~\eqref{eq:Attempt1}. Suppose to the contrary that $\theta = 0$.  Then this solution is also feasible as $t \downarrow \infty$, which implies that $\P(\burv^T\bv > -\infty) \leq \epsilon$ for all $\P \in \mathcal{P}^{DY}$.  On the other hand, the empirical distribution $\hat{\P} \in \mathcal{P}^{DY}$, a contradiction.

Since $\theta > 0$, we can rescale all of the above optimization variables in problem~\eqref{eq:DYProbBound} by $\theta$.  Substituting this into Eq.~\eqref{eq:Attempt1} yields the given expression for $\sup_{\P \in \mathcal{P}^{DY}} \quant{\epsilon}{\P}(\bv)$.
Rewriting this optimization problem as a semidefinite optimization problem and taking its dual yields $\U^{DY}_\epsilon$ in the theorem.  By Theorem~\ref{thm:uniform}, this set simultaneously implies a probabilistic guarantee.
\hfill \Halmos \endproof

{ \blockedit 
\subsection{Proof of Theorem~\ref{thm:biconvex}.}
\proof{Proof.}
For each part, the convexity in $(\bv, t)$ is immediate since $\delta^*(\bv | \ \U_\epsilon)$ is a support function of a convex set.  For the first part, note that from the second part of Theorem~\ref{thm:CSSet}, $\delta^*(\bv | \ \U^{CS}_\epsilon) \leq t$ will be convex in $\epsilon$ for a fixed $(\bv, t)$ whenever $\sqrt{1/\epsilon -1}$ is convex.  Examining the second derivative of this function, this occurs on the interval $0 < \epsilon < .75$.  Similarly, for the second part, note that from the second part of Theorem~\ref{thm:UFB}, $\delta^*(\bv | \ \U^{FB}_\epsilon) \leq t$ will be convex in $\epsilon$ for a fixed $(\bv, t)$ whenever $\sqrt{2\log(1/\epsilon)}$ is convex.  Examining the second derivative of this function, this occurs on the interval $0 < \epsilon < 1\sqrt{e}$.  

From the representations of $\delta^*(\bv | \U^{\chi^2}_\epsilon)$ and $\delta^*(\bv | \U^{G}_\epsilon)$ in Theorem~\ref{thm:discretesets}, we can see they will be convex in $\epsilon$ whenever $1/\epsilon$ is convex, i.e., $0 < \epsilon < 1$. 
From the representation of $\delta^*(\bv | \U^I_\epsilon)$ in Theorem~\ref{thm:UI} and since $\lambda \geq 0$, we see this function will be convex in $\epsilon$ whenever $\log(1/\epsilon)$ is convex, i.e., $0 < \epsilon < 1$. 

Finally, examining the support functions of $\U^{LCX}_\epsilon$ and $\U^{DY}_\epsilon$ shows that $\epsilon$ occurs linearly in each of these functions.  
\hfill \Halmos \endproof

\section{Omitted Figures} \label{sec:OmittedFigs}
This section contains additional figures omitted from the main text.  
\begin{figure}[h!]
\centering
\begin{subfigure}{0.3\textwidth}
	\includegraphics[width=\textwidth]{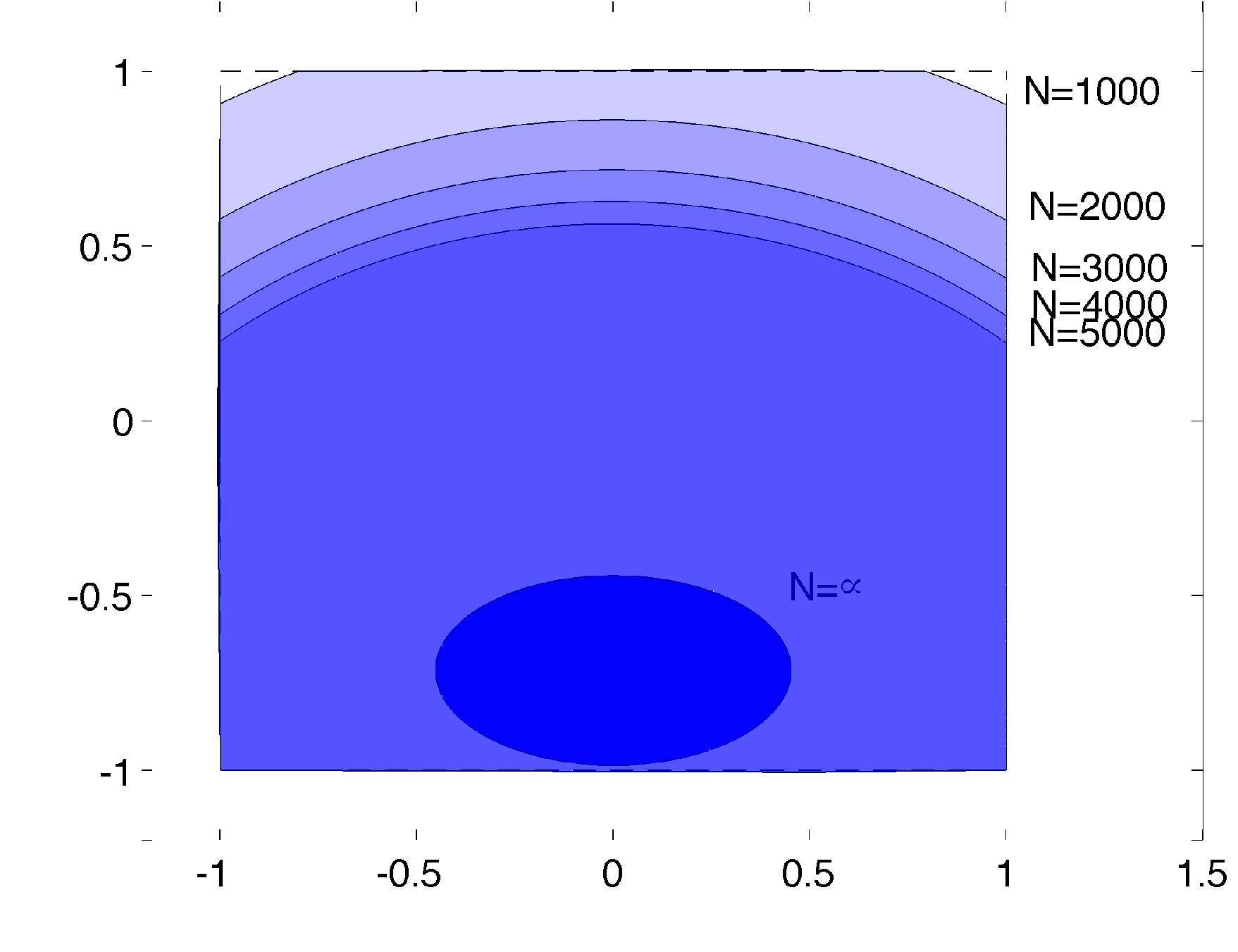}
	\caption{\label{fig:UCS_BootOrig} Not Bootstrapped}
\end{subfigure}
\begin{subfigure}{0.3\textwidth}
	\includegraphics[width=\textwidth]{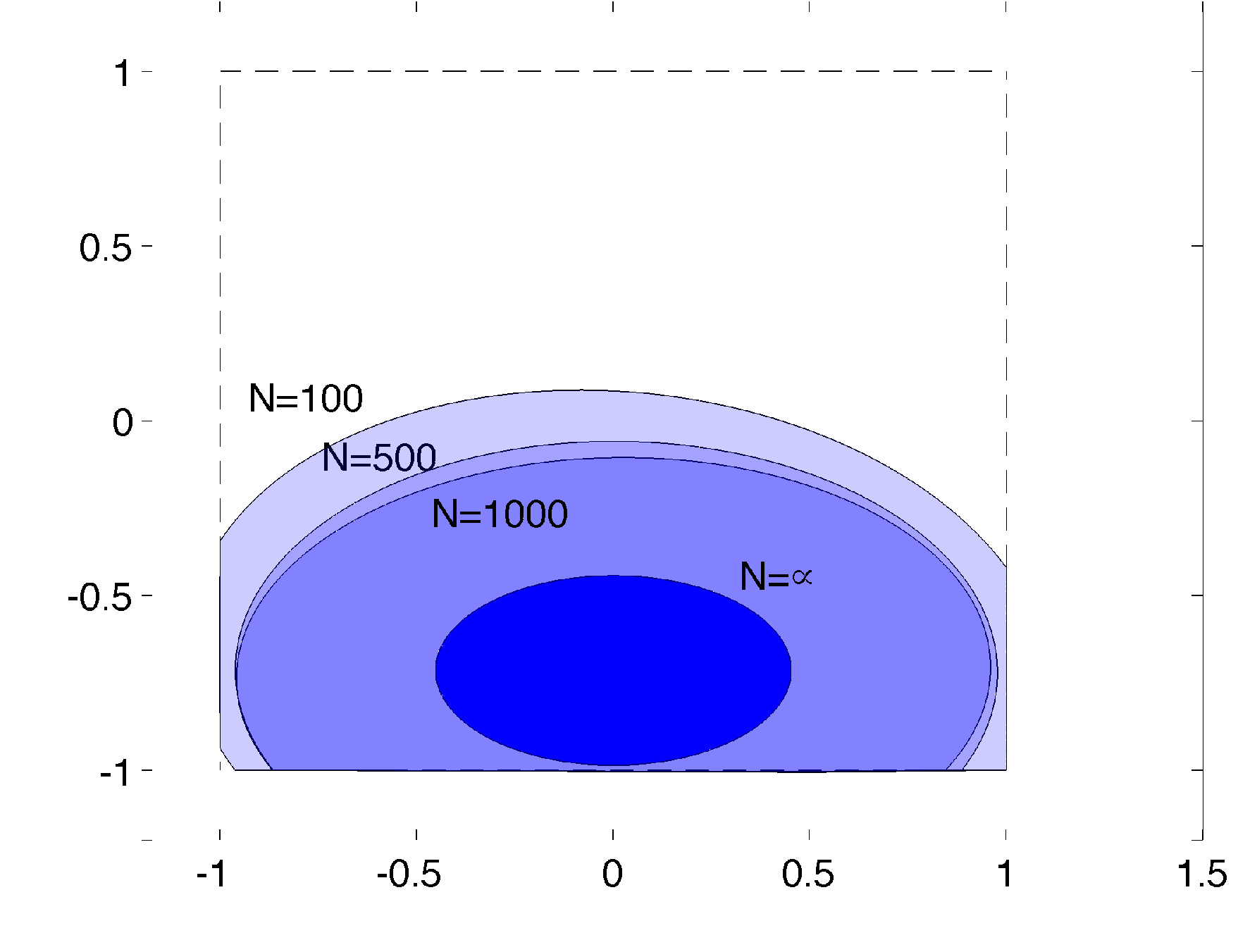}
	\caption{\label{fig:UCS_BootBoot} Bootstrapped}
\end{subfigure}
\caption{\label{fig:UCSBoot}  $\U^{CS}_\epsilon$ with and without bootstrapping for the example from Fig.~\ref{fig:UI}.  $N_B=10,000$, $\alpha=10\%$, $\epsilon = 10\%$.  Notice that for $N=1,000$, the non-bootstrapped set is almost as big as the full support and shrinks slowly to its infinite limit.  The bootstrapped set with $N=100$ points is smaller than the non-bootstrapped version with $50$ times as many points.  
}
\end{figure}

}

{\blockedit
\section{Optimizing $\epsilon_j$'s for Multiple Constraints}
\label{sec:EpsilonOpt}
In this section we specify the optimization problem that we solve in $\epsilon_j$'s as part of our alternating optimization heuristic for treating multiple constraints.  
We first present our approach using $m$ constraints of the form $\delta^*(\bv | \ \U_\epsilon^{CS} ) \leq t$.  Without loss of generality, assume the overall optimization problem is a minimization.  Consider the $j^\text{th}$ constraint, and let $(\bv^\prime, t^\prime)$ denote the subset of the solution to the original optimization problem at the current iterate pertaining to the $j^\text{th}$ constraint.  Let $\epsilon^\prime_j$, $j=1, \ldots, m$ denote the current iterate in $\epsilon$.  Finally, let $\lambda_j$ denote the shadow price of the $j^\text{th}$ constraint in the overall optimization problem.    

Notice from the second part of Theorem~\ref{thm:CSSet} that $\delta^*(\bv | \ \U^{CS}_{\epsilon})$ is decreasing in $\epsilon$.  Thus, for all $\epsilon_j \geq \underline{\epsilon}_j$, $\delta^*(\bv^\prime | \ \U^{CS}_{{\epsilon}_j}) \leq t^\prime$, where, 
\[
\underline{\epsilon}_j \equiv \left[ \frac{(t^\prime - \hat{\bmu}^T\bv^\prime - \Gamma_1 \| \bv^\prime \|_2)^2}{{\bv^\prime}^T (\bSigma + \Gamma_2 \bI)\bv^\prime} + 1 \right]^{-1}.
\]

Motivated by the shadow-price $\lambda_j$, we define the next iterates of $\epsilon_j$, $j=1, \ldots, m$ to be the solution of the linear optimization problem
\begin{align} \notag
\min_{\bepsilon} \quad & - \sum_{j=1}^d \left(\frac{ \sqrt{{\bv^ \prime}^T (\bSigma + \Gamma_2 \bI)\bv^\prime} } {  2 {\epsilon^\prime}^2 \sqrt{\frac{1}{\epsilon^\prime}-1} } \right)\lambda_j \cdot \epsilon_j
\\ \label{eq:EpsProb}
\text{s.t.} \quad &\underline{\epsilon}_j \leq \epsilon_j \leq .75, \ \ j=1, \ldots, m, 
\\ \notag
&\sum_{j=1}^m \epsilon_j  \leq \overline{\epsilon}, \ \ \| \bepsilon^\prime - \bepsilon \|_1 \leq \kappa.
\end{align}
The coefficient of $\epsilon_j$ in the objective function is $\lambda_j \cdot \partial_{\epsilon_j} \delta^*(\bv^\prime | \ \U^{CS}_{\epsilon_j} )$ which is intuitively a first-order approximation to the improvement in the overall optimization problem for a small change in $\epsilon_j$.  The norm constraint on $\bepsilon$ ensures that the next iterate is not too far away from the current iterate, so that the shadow-price $\lambda_j$ remains a good approximation.  (We use $\kappa = .05$ in our experiments.)  The upper bound ensures that we remain in a region where $\delta^*(\bv | \ \U^{CS}_{\epsilon_j})$ is convex in $\epsilon_j$.  
Finally, the lower bounds on $\epsilon_j$ ensure that the previous iterate of the original optimization problem $(\bv^\prime, t^\prime)$ will still be feasible for the new values of $\epsilon_j$.  Consequently, the objective value of the original optimization problem is non-increasing.  We terminate the procedure when the objective value no longer makes significant progress. 

With the exception of $\U^{LCX}_\epsilon$, we can follow an entirely analogous procedure, 
simply adjusting the formulas for $\underline{\epsilon}_j$, the upper bounds, and the objective coefficient appropriately.  We omit the details.  Computing the relevant objective coefficient for $\U^{LCX}$ is more subtle.  From \eqref{eq:suppLCX}, we require the optimal $\tau$ corresponding to $\bv^\prime$.  This $\tau$ is dual to the constraint $z \leq \frac{1}{\epsilon}$.  Thus, our strategy is to evaluate $\delta^*(\bv^\prime | \ \U_{\epsilon^\prime_j})$ by generating $(\ba, b)$'s via the separation routine of Remark~\ref{rem:SeparateLCX}.  At termination, we let $\tau$ be the dual variable to the constraint $z \leq \frac{1}{\epsilon}$, and, finally, we set the objective coefficient of $\epsilon_j$ in \eqref{eq:EpsProb} to be
$-\frac{\tau}{\epsilon^\prime_j}^2$.  Again, intuitively, this coefficient corresponds to the change in the overall optimization problem for a small change in $\epsilon_j$.  
}

{\blockedit
\section{Additional Portfolio Results}
\label{sec:PortAppendix}
Fig.~\ref{fig:Port2} summarizes the case $N=2000$ for the experiment outlined in Sec.~\ref{sec:portfolio}.
\begin{figure}
\begin{subfigure}{.49 \textwidth}
	\includegraphics[width=\textwidth]{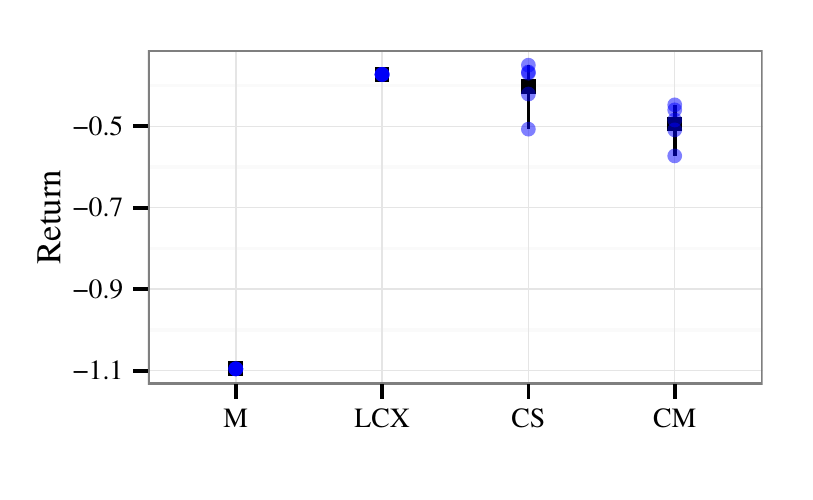}
\end{subfigure}	
\begin{subfigure}{.49 \textwidth}
	\includegraphics[width=\textwidth]{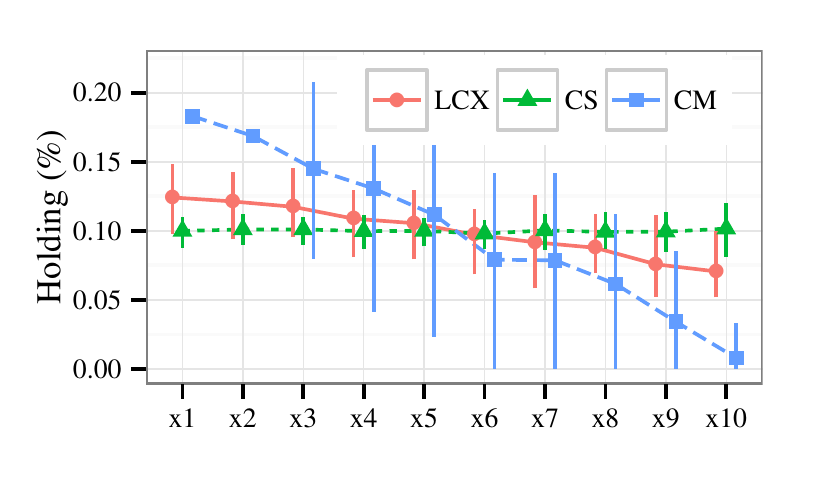}
\end{subfigure}	
\caption{\label{fig:Port2} The case $N=2000$ for the experiment outlined in Sec.~\ref{sec:portfolio}.  The left panel shows the cross-validation results.  The right panel shows the average holdings by method.  $\alpha = \epsilon = 10\%$.}  
\end{figure}
}
{\blockedit
\section{Additional Queueing Results}
\label{sec:QueueAppendix}

We first derive the bound $W_n^{3, FB}$.  Notice that in \eqref{eq:SingleServerQueue}, 
the optimizing index $j$ represents the most recent customer to arrive when the queue was empty.  Let $\tilde{n}$ denote the number of customers served in a typical busy period.  
Intuitively, it suffices to truncate the recursion \eqref{eq:SingleServerQueue} at customer $\min (n, n^{(k)} )$ where, with high probability, $\tilde{n} \leq n^{(k)}$.  More formally, considering only the first half of the data $\hat{x}^1, \ldots, \hat{x}^{\lceil N/2 \rceil}$  and $\hat{t}^1, \ldots, \hat{t}^{\lceil N/2 \rceil}$, we compute the number of customers served in each busy period of the queue, denoted $\hat{n}^1, \ldots, \hat{n}^K$, which are i.i.d. realizations of $\tilde{n}$.  Using the KS test at level $\alpha_1$, we observe that with probability at least $1-\alpha$ with respect to the sampling,
\begin{align} \label{bad step}
\P( \tilde{n} > \hat{n}^{(k)}) \leq 1 - \frac{k}{K}  + \Gamma^{KS}(\alpha), \quad  \forall k =1, \ldots, K. 
\end{align}
In other words, the queue empties every $\hat{n}^{(k)}$ customers with at least this probability.

Next, calculate the constants $\mathbf{m}_f, \mathbf{m}_b, \boldsymbol{\sigma}_f, \boldsymbol{\sigma}_b$ using only the second half of the data.  Then, truncate the sum in \eqref{eq:SumExpr} at $\min(n, n^{(k)})$ and replace the righthand side by $\overline{\epsilon} - 1 + \frac{k}{K} - \Gamma^{KS}(\alpha /2 )$.  
Denote the solution of this equation by $W_n^{2, FB}(k)$.  Finally, let $W^{3, FB}_n \equiv \min_{1 \leq k < K} W_n^{2, FB}(k)$, obtained by grid-search.  

We claim that with probability at least $1-2\alpha$ with respect to the sampling, $\P(\tilde{W}_n > W^{3, FB}_n) \leq \overline{\epsilon}$.
Namely, from our choice of parameters, eqs.\eqref{eq:SumExpr} and \eqref{bad step} hold simultaneously with probability at least $1-2\alpha$.  Restrict attention to a sample path where these equations hold.  Since \eqref{bad step} holds for the optimal index $k^*$, recursion \eqref{eq:SingleServerQueue} truncated at $n^{(k^*)}$ is valid with probability at least $1 - \frac{k^*}{K}  + \Gamma^{KS}(\alpha)$.  Finally, $\P(\tilde{W}_n > W^{3, FB}_n) \leq \P( \text{ \eqref{eq:SingleServerQueue} is invalid } ) + \P((\tilde{W}_n > W^{2, FB}_n(k^*) \text{ and\eqref{eq:SingleServerQueue} is valid } ) \leq \overline{\epsilon}$.  This proves the claim.  

We observe in passing that since the constants $\mathbf{m}_f, \mathbf{m}_b, \boldsymbol{\sigma}_f, \boldsymbol{\sigma}_b$ are computed using only half the data, it may not be the case that $W^{3, FB}_n < W^{2, FB}_n$, particularly for small $N$, but that typically $W^{3, FB}_n$ is a much stronger bound than $W^{2, FB}_n$.

Applying a similar analysis with set $\U^{CS}_{\overline{\epsilon}}$, yields the following bounds:
\begin{align*}
W_n^{1, CS} 
\leq \begin{cases}
				(\hat{\mu}_1 - \hat{\mu}_2) n + \left(\Gamma_1 + \sqrt{(\frac{n}{\overline{\epsilon}} - 1)(\sigma_1^2 + \sigma_2^2 + 2\Gamma_2 ) } \right)\sqrt{n}
& \text{ if } n < \frac{\left( \Gamma_1 + \sqrt{(\frac{n}{\overline{\epsilon}} - 1)(\sigma_1^2 + \sigma_2^2 + 2\Gamma_2 ) } \right)^2} {4 ( \hat{\mu}_1 - \hat{\mu}_2)^2}
\\
&\quad \text{ or } \hat{\mu}_1 > \hat{\mu}_2,
\\
				\frac{\left( \Gamma_1 + \sqrt{(\frac{n}{\overline{\epsilon}} - 1)(\sigma_1^2 + \sigma_2^2 + 2\Gamma_2 ) } \right)^2} {4 ( \hat{\mu}_2 - \hat{\mu}_1)}
				& \text{ otherwise }
			\end{cases}
\end{align*}
$W_n^{2, CS}$ is the solution to 
\begin{equation} \label{eq:SumExpr2}
\sum_{j=1}^{n-1} \left[ \left( \frac{W_n^{2, CS}- ( \hat{\mu}_1 - \hat{\mu}_2) (n-j)}{\sqrt{n-j}\sqrt{\sigma_1^2 + \sigma_2^2 + 2\Gamma_2 } } - \frac{\Gamma_1}{\sqrt{\sigma_1^2 + \sigma_2^2 + 2\Gamma_2}} \right)^2 + 1 \right]^{-1} = \overline{\epsilon},
\end{equation}
and $W_n^{3, CS}$ defined analogously to $W_n^{3, FB}$ but using \eqref{eq:SumExpr2} in lieu of \eqref{eq:SumExpr}.  
}

\section{Constructing $\U^I_\epsilon$ from Other EDF Tests}  
\label{sec:OtherEDF}
In this section we show how to extend our constructions for $\U^I_\epsilon$ to other EDF tests.  We consider several of the most popular, univariate goodness-of-fit, empirical distribution function test.  Each test below considers the null-hypothesis $H_0: \P^*_i = \P_{0,i}$.  
{ \small 
\begin{description}
\item[Kuiper (K) Test:]
	The K test rejects the null hypothesis at level $\alpha$ if
	\begin{equation*}
	\max_{j=1, \ldots, N} \left( \frac{j}{N} - \P_{0,i}(\urv_i \leq \uhat_i^{(j)}) \right)+ 
	\max_{j=1, \ldots, N} \left( \P_{0,i}(\urv_i < \uhat_i^{(j)}) - \frac{j-1}{N} \right) > V_{1-\alpha}.
	\end{equation*}
\item[Cramer von-Mises (CvM) Test:]
	The CvM test rejects the null hypothesis at level $\alpha$ if
	\begin{equation*}
	\frac{1}{12N^2} + \frac{1}{N}\sum_{j=1}^N \left( \frac{ 2j-1}{2N} - \P_{0,i}(\urv_i \leq \uhat_i^{(j)}) \right)^2 > (T_{1-\alpha})^2.
	\end{equation*}
\item[Watson (W) Test:]
	The W test rejects the null hypothesis at level $\alpha$ if
	\begin{equation*}
	\frac{1}{12N^2} + \frac{1}{N} \sum_{j=1}^N \left( \frac{ 2j-1}{2N} - \P_{0, i}(\urv_i \leq \uhat_i^{(j)}) \right)^2 
	- \left( \frac{1}{N} \sum_{j=1}^N \P_{0, i}(\urv_i \leq \uhat_i^{(j)}) - \frac{1}{2} \right)^2
	> (U_{1-\alpha})^2.	
	\end{equation*}
\item[Anderson-Darling (AD) Test:]
	The AD test rejects the null hypothesis at level $\alpha$ if 
	\begin{equation*}
	-1 - \sum_{j=1}^N \frac{2j-1}{N^2} \left( \log\left( \P_{0,i}(\urv_i \leq \uhat_i^{(j)}) \right) + \log\left(1- \P_{0, i}(\urv_i \leq \uhat_i^{(N+1-j)}) \right)\right) > (A_{1-\alpha})^2
	\end{equation*}
\end{description}
}
Tables of the thresholds above are readily available \citep[e.g.,][and references therein]{stephens1974edf}.  

{\blockedit
As described in \cite{BGKII}, the confidence regions of these tests can be expressed in the form
\[
\InfCalP^{EDF}_i = \{ \P_i \in \theta[\uhat_i^{(0)}, \uhat_i^{(N+1)}] : \exists \zeta \in \R^{N}, \ \  \P_i(\urv_i \leq \uhat_i^{(j)} ) = \zeta_i, \ \ \bA_S \bzeta - \bb_ S \in \mathcal{K}_S  \},
\]
where the the matrix $\bA_S$, vector $\bb_S$ and cone $\mathcal{K}_S$ depend on the choice of test.  Namely, 
{ \small
\begin{align}\notag 
&\mathcal K_{K}=\{ (\bx, \by) \in \R^{2N} : \min_i x_i + \min_i y_i \geq 0 \},
\quad \bb_{K} = \begin{pmatrix} \frac1N - V_{1-\alpha}/2\\\vdots\\\frac NN - V_{1-\alpha}/2 \\ - \frac0N - V_{1-\alpha}/2\\\vdots\\-\frac{N-1}N-V_{1-\alpha}/2\end{pmatrix},
\quad \bA_{K}= \begin{pmatrix} \ \  [\bI_N] \\ [-\bI_N]  \\    \end{pmatrix},
\\ \label{eq:CvMConic}
&\mathcal K_{CvM} = \{ \bx \in \R^N, t \in \R_+ : \| \bx \| \leq t \}, \quad 
\bb_{CvM}    =   \begin{pmatrix} \sqrt{N (T^2_{1-\alpha})^2 -\frac{1}{2N}}\\\frac{1}{2N}\\\frac{3}{2N}\\\vdots\\\frac{2N-1}{2N}\end{pmatrix},\quad \bA_{CvM}= \begin{pmatrix} 0\cdots0\\ {[\bI_N]}\\ \end{pmatrix},
\\ \label{eq:WConic}
&\mathcal{K}_{W}=  \{ \bx \in \R^{N+1}, t \in \R_+ : \| \bx \| \leq t \},
\quad \bb_{W}= \begin{pmatrix}-\frac{1}{2}+\left(\frac{N}{24}-\frac{N}{2} (U_{1-\alpha})^2\right)\\ -\frac{1}{2}-(\frac{N}{24}-\frac{N}{2} (U_{1-\alpha})^2)\\0\\\vdots\\0\end{pmatrix},
\quad 
\bA_{U_N}= \begin{pmatrix} \frac{1-N}{2N}&\frac{3-N}{2N}&\dots&\frac{N-1}{2N}\\\frac{N-1}{2N}&\frac{N-3}{2N}&\dots&\frac{1-N}{2N}\\\multicolumn{4}{c}{{{[\bI_N-\frac{1}{N}\mathbf{E}_N]}}}\\ \end{pmatrix},
\\ \label{eq:ADConic} 
&\mathcal K_{AD}= \left\{(z, \bx, \by)\in\R \times \R^{2N}_+ :  |z | \leq \prod_{i=1}^N(x_iy_i)^{\frac{2i-1}{2N^2}}\right\},
\quad \bb_{AD}=\begin{pmatrix}e^{-(A_{1-\alpha})^2-1}\\0\\\vdots\\0\\-1\\\vdots\\-1\end{pmatrix},\quad 
\bA_{AD}=\begin{pmatrix} 0\cdots0\\ [I_N] \\ [-\tilde I_N] \\ \end{pmatrix},
\end{align}	
}
where $\bI_N$ is the $N\times N$ identity matrix, $\tilde \bI_N$ is the skew identity matrix $([\tilde \bI_N]_{ij}=\mathbb I[i=N-j])$, and $\mathbf{E}_N$ is the $N\times N$ matrix of all ones.

Let $\mathcal{K}^*$ denote the dual cone to $\mathcal{K}$.  By specializing Theorem 10 of \cite{BGKII}, we obtain the following theorem, paralleling Theorem~\ref{thm:KSFinite}.  
\begin{theorem}
\label{thm:OtherFinite}
Suppose $g(u)$ is monotonic and right-continuous, and let $\InfCalP^S$ denote the confidence region of any of the above EDF tests.  
\begin{align} \notag
\sup_{\P_i \in \InfCalP^{EDF}_i} \E^{\P_i}[ g(\urv_i)] = \min_{\br, \bc} \quad &\bb_S^T \br + c_{N+1}
\\ \notag
\text{s.t.} \quad & -\br \in \mathcal{K}_S^*, \ \ \bc \in \R^{N+1}, 
\\ \notag
& (\bA_S^T \br)_j = c_j - c_{j+1} \quad \forall j=1, \ldots, N, 
\\ \label{eq:minEDF}
&c_j \geq g(\uhat_i^{(j-1)}), \ \ c_j \geq g(\uhat_i^{(j)}),  \quad j=1, \ldots, N+1.  
\\ \notag
=\max_{\bz, \bq^L, \bq^R, \bp} \quad &\sum_{j=0}^{N+1} p_j g(\uhat_i^{(j)})
\\ \notag
\text{s.t.} \quad & \bA_S \bz - \bb_S \in \mathcal{K}_S, \ \ \bq^L, \bq^R, \bp \in  \R_+^{N+1} 
\\ \label{eq:maxEDF}
& q^L_j + q^R_j = z_j - z_{j-1},  \ \ j=1, \ldots, N, 
\\ \notag
& q^L_{N+1} + q^R_{N+1} = 1 - z_{N}
\\ \notag
& p_0 = q^L_1, \ \ p_{N+1} = q^R_{N+1}, \ \  p_j = q^L_{j+1} + q^R_j, \ \ j = 1, \ldots, N, 
\end{align}
where $\bA_S, \bb_S, \mathcal{K}_S$ are the appropriate matrix, vector and cone to the test.  Moreover, when $g(u)$ is non-decreasing (resp. non-increasing), there exists an optimal solution where $\bq^L = \bzero$ (resp. $\bq^R = \bzero$) in \eqref{eq:maxEDF}.  
\end{theorem}
\begin{proof}{Proof.}
Apply Theorem 10 of \cite{BGKII} and observe that since $g(u)$ is monotonic and right continuous, 
\[
c_j \geq \sup_{u \in (\uhat_i^{(j-1)}, \uhat_i^{(j)}]} g(u) \iff c_j \geq g(\uhat_i^{(j-1)}), \ c_j \geq g(\uhat_i^{(j)}).  
\]
Take the dual of this (finite) conic optimization problem to obtain the given maximization formulation.

To prove the last statement, suppose first that $g(u)$ is non-decreasing and fix some $j$.  If $g(\uhat_i^{(j)}) > g(\uhat_i^{(j-1)})$, then by complementary slackness, $\bq^L = 0$.  If $g(\uhat_i^{(j)}) = g(\uhat_i^{(j-1)})$, 
then given any feasible $(q^L_j, q^R_j)$, the pair $(0, q^L_j+ q^R_j)$ is also feasible with the same objective value.  Thus, without loss of generality, $\bq^L= 0$.  The case where $g(u)$ is non-increasing is similar.

\end{proof}
\begin{remark}
At optimality of \eqref{eq:maxEDF}, $\bp$ can be considered a probability distribution, supported on the points $\uhat_i^{(j)}$ $j=0, \ldots, N+1$.  This distribution is analogous to $\bq^L(\Gamma), \bq^R(\Gamma)$ for the KS test.  
\end{remark}

In the special case of the $K$ test, we can solve \eqref{eq:maxEDF} explicitly to find this worst-case distribution.  
\begin{corollary} \label{corr:K}
When $\InfCalP^{EDF}_i$ refers specifically to the K test in Theorem~\ref{thm:OtherFinite} and if $g$ is monotonic, we have
\begin{equation}
\label{eq:DisjK}
\sup_{\P_i \in \InfCalP^{EDF}_i} \E^{\P_i}[ g(\urv_i) ] = \max\left( \sum_{j=0}^{N+1} q^L_j(\Gamma^{K}) g(\uhat_i^{(j)}), \sum_{j=0}^{N+1} q^R_j(\Gamma^{K}) g(\uhat_i^{(j)}) \right).
\end{equation}
\end{corollary}
\begin{proof}{Proof.}
One can check that in the case of the $K$ test, the maximization formulation given is equivalent to \eqref{eq:LPForKS} with $\Gamma^{KS}$ replaced by $\Gamma^{K}$.  Following the proof of Theorem~\ref{thm:KSFinite} yields the result.
\end{proof}
\begin{remark}
One an prove that $\Gamma^K \geq \Gamma^{KS}$ for all $N$, $\alpha$.  Consequently, $\InfCalP_i^{KS} \subseteq \InfCalP_i^K$.  For practical purposes, one should thus prefer the KS test to the K test, as it will yield smaller sets.  
\end{remark}

We can now generalize Theorem~\ref{thm:UI}.  For each of K, CvM, W and AD tests, define the (finite dimensional) set
\begin{equation} \label{eq:FiniteEDFFamily}
\mathcal{P}_i^{EDF} = \{ \bp \in \R^{N+2}_+ : \exists \bq^L, \bq^R \in \R^{N+2}_+, \bz \in \R^N \text{ s.t. } \bp, \bq^L, \bq^R, \bz \text{ are feasible in \eqref{eq:maxEDF}} \}, 
\end{equation}
using the appropriate $\bA_S, \bb_S,  \mathcal{K}_S$.  
}
\begin{theorem}
\label{thm:UI2}
Suppose $\P^*$ is known to have independent components, with $\supp(\P^*) \subseteq [\buhat^{(0)}, \buhat^{(N+1)}]$.  
\begin{enumerate}[label = \roman*)]
\item
With probability at least $1-\alpha$ over the sample, the family $\{\U^I_\epsilon : 0 < \epsilon < 1 \}$ simultaneously implies a probabilistic guarantee, where
\begin{equation} \label{def:UI2}
\begin{aligned}
\U^I_\epsilon = \Biggr\{ \bu \in \R^d : &\ \exists \bp^i \in \mathcal{P}_i^{EDF}, \ \bq^i \in \Delta_{N+2}, \ i = 1\ldots, d,   
\\ &\sum_{j=0}^{N+1} \uhat_i^{(j)} q_j^i = u_i  \ i = 1, \ldots, d,  
\ \  \sum_{i=1}^d D( \bq^i,  \bp^i ) \leq \log( 1/\epsilon)
\Biggr\}.
\end{aligned}
\end{equation}
\item In the special case of the K test, the above formulation simplifies to \eqref{def:UI} with $\Gamma^{KS}$ replaced by $\Gamma^K$.  
\end{enumerate}
\end{theorem} 
The proof of the first part is entirely analogous to Theorem~\ref{thm:UI}, but uses Theorem~\ref{thm:OtherFinite} to evaluate the worst-case expectations. The proof of the second part follows by applying Corollary~\ref{corr:K}.  We omit the details.
\begin{remark}
In contrast to our definition of $\U^I_\epsilon$ using the KS test, we know of no simple algorithm for evaluating $\delta^*(\bv | \  \U^I_\epsilon)$ when using the CvM, W, or AD tests.  (For the K test, the same algorithm applies but with $\Gamma^K$ replacing $\Gamma^{KS}$.)  Although it still polynomial time to optimize over constraints $\delta^*(\bv | \ \U^I_\epsilon) \leq t$ for these tests using interior-point solvers for conic optimization, it is more challenging numerically.  
\end{remark}

\end{document}